%% file: experiods.tex
\documentclass[11pt]{article} 
\usepackage{amsmath,amssymb,amsfonts,amsthm,amscd,graphicx,psfrag,epsfig}
\usepackage{color}
\definecolor{blue}{rgb}{0,0,0.7}
\definecolor{red}{rgb}{0.75, 0, 0}
\usepackage{stmaryrd}
\usepackage[titletoc,toc]{appendix}
\usepackage{comment}

\newtheorem{theorem}{Theorem}[section]

\newtheorem{theorem-definition}[theorem]{Theorem-Definition}
\newtheorem{theorem-construction}[theorem]{Theorem-Construction}
\newtheorem{lemma-definition}[theorem]{Lemma--Definition}
\newtheorem{lemma-construction}[theorem]{Lemma--Construction}
\newtheorem{lemma}[theorem]{Lemma}
\newtheorem{proposition}[theorem]{Proposition}
\newtheorem{corollary}[theorem]{Corollary}
\newtheorem{conjecture}[theorem]{Conjecture}
\newtheorem{definition}[theorem]{Definition}

\newcommand{\old}[1]{}

\newcommand{\Z}{{\mathbb Z}}
\newcommand{\R}{{\mathbb R}}
\newcommand{\Q}{{\mathbb Q}}
\newcommand{\C}{{\mathbb C}}

\newcommand{\G}{{\rm G}}

\newcommand{\lms}{\longmapsto}
\newcommand{\lra}{\longrightarrow}
\newcommand{\hra}{\hookrightarrow}

\newcommand{\be}{\begin{equation}}
\newcommand{\ee}{\end{equation}}
\newcommand{\bt}{\begin{theorem}}
\newcommand{\et}{\end{theorem}}
\newcommand{\bd}{\begin{definition}}
\newcommand{\ed}{\end{definition}}
\newcommand{\bp}{\begin{proposition}}
\newcommand{\ep}{\end{proposition}}

\newcommand{\bl}{\begin{lemma}}
\newcommand{\el}{\end{lemma}}
\newcommand{\bc}{\begin{corollary}}
\newcommand{\ec}{\end{corollary}}
\newcommand{\bcon}{\begin{conjecture}}
\newcommand{\econ}{\end{conjecture}}
\newcommand{\la}{\label}

\begin{document}
\date{\it To Vadim Schechtman on the occasion of his 60th birthday}

\title{Exponential complexes, period morphisms,  and  characteristic classes }

\author{A. B.  Goncharov}
\maketitle

\tableofcontents

\begin{abstract}
We introduce a {\it weight $n$ exponential complex}  
of sheaves $\Q^\bullet_{\bf E}(n)$ on a manifold $X$: 
\begin{equation} \label{ana1eqa}
{\cal O} (n-1) \lra {\cal O}^*  \otimes {\cal O} (n-2) \lra ... \lra  
\otimes^{n-1}{\cal O} ^*\otimes {\cal O}  
    \lra \otimes^{n }{\cal O}^*.
\end{equation} 
It is a resolution 
 of the constant sheaf 
$\Q(n)$, generalising the 
classical  exponential sequence:
$$
\Z(1)  \lra {\cal O}  \stackrel{\rm exp}{\lra} {\cal O}^*,
~~~~ \Z(1):= 2\pi i\Z.
$$

There is a canonical map from the complex $  \Q^\bullet_{\bf E}(n)$ to the de Rham complex 
$\Omega^\bullet$ of $X$. Using it, we define 
a {\it weight $n$ exponential Deligne  complex}, calculating rational Deligne cohomology:
$$
\Gamma_{\cal D}(X; n):= {\rm Cone}\Bigl(
\Q^\bullet_{\bf E}(n) \oplus F^{\geq n}\Omega^\bullet \lra \Omega^\bullet\Bigr)[-1]. 
$$
Its main advantage is that, at least at the generic point  
${\cal X}$ of a complex variety $X$, it allows to define Beilinson's 
regulator map to the rational Deligne cohomology on the level of complexes. 
(A regulator map to real Deligne complexes for any regular complex variety 
 is known  \cite{G95}).

Namely, we define a weight $n$ {\it period morphism}. 
We use it to define 
a map of complexes
\be \la{4.22.15.2a}
\mbox {\it a weight $n$ motivic complex of ${\cal X}$} \lra  \mbox{\it the weight $n$ exponential complex 
of ${\cal X}$}. 
\ee
 We show that it gives rise to a map of complexes
\be \la{4.22.15.2a1}
\mbox {\it a weight $n$ motivic complex of ${\cal X}$} \lra   
\mbox {\it the weight $n$ exponential Deligne  complex of ${\cal X}$}.
\ee
It induces Beilinson's regulator map  on the cohomology.

\vskip 2mm
Combining  the map (\ref{4.22.15.2a1}) with the construction of Chern classes 
with coefficients in the bigrassmannian complexes \cite{G93}, we get 
a local explicit formula for the $n$-th Chern class in the rational Deligne cohomology via polylogarithms, 
at least 
for $n \leq 4$. 
Equivalently, we get an explicit construction for the 
universal Chern class in the rational  Deligne cohomology
$$
c_n^{\cal D}\in H^{2n}(BGL_N(\C), \Gamma_{\cal D}(n)), ~~~~n\leq 4.
$$
In particular, this gives 
explicit formulas for  Cech cocycles for the topological Chern classes.

\end{abstract}
\section{Introduction, main definitions, and examples}

\subsection{A motivation: local construction of Chern classes} 

Topological invariants can often be localised by introducing  additional 
structures of local nature. 

For example, the topological Chern classes 
of a vector bundle $E$ on a manifold can be localised by introducing a connection $\nabla$ on $E$:  
the differential form $(2\pi i)^{-n}{\rm Tr}(F_\nabla^n)$, where $F_\nabla$ is the curvature of 
 $\nabla$, is a de Rham representative
 of the Chern class $c_n(E)$.

In this paper we address the problem of a local construction of explicit Cech cocycles 
representing the Chern classes. 
A construction 
of Chern classes with values in the bigrassmannian complex was given in \cite{G93}. 
To get from there a local formula for  topological Chern classes, or    
Chern classes in the rational  
Deligne cohomology,  
one needs a transcendental construction relying on polylogarithms. 
It should handle the complicated multivalued nature of  polylogarithms. 

We develop such a   construction. 
We define a weight $n$ exponential complex, 
which is a resolution of the constant sheaf $  \Q(n)$ 
on a manifold. Using it, we define a new complex calculating the rational Deligne cohomology, 
and construct a period morphism, which gives rise to 
a regulator map on the level of complexes at the generic point of a complex algebraic variety. 
Yet, more work needs to be done to find  
a local construction of the Chern classes $c_n(E)$ when $n>4$. 

Let us now look at the problem in detail in the simplest possible case. 

\paragraph{1. The first Chern class.} Let $E$ be a complex line bundle on a real manifold $X$. 
Here is  a classical construction of a Cech cocycle representing the first Chern class 
$$
c_1(E) \in H^2(X, \Z(1)).
$$ Take a cover of $X$ by  
open subsets $U_i$ such that all intersections $U_{i_0} \cap ... \cap U_{i_k}$ are empty or contractible. The restriction of $E$ to $U_i$ is trivial, so we may choose a nonvanishing section $s_i$. The ratio $s_i/s_j$ is an invertible function on $U_i \cap U_j$.  
Choose a branch of $\log (s_i/s_j)$. Then there 
is a $2$-cocycle in the  Cech  complex of the cover, whose cohomology class is $c_1(E)$:
$$
U_i \cap U_j \cap U_k \lms \log (s_i/s_j) - \log (s_j/s_k) + \log (s_k/s_i) \in 2 \pi i\Z.
$$
Equivalently, take the short exact exponential sequence of sheaves on 
$X$, where ${\cal O}$ is the structure sheaf of continuous complex valued functions:
$$
\Z(1) \lra {\cal O} \stackrel{\rm exp}{\lra} 
{\cal O}^*.
$$
Then the above construction just means  the following: 

1. We assign to a complex line bundle $E$ on  $X$ a Cech cocycle 
representing its class 
$$
{\rm cl}(E) \in H^1(X, {\cal O}^*).
$$

2. We calculate  the coboundary map in the exponential complex:
$$
\delta: H^1(X, {\cal O}^*) \lra H^2(X, \Z(1)). 
$$
Then $$
\delta({\rm cl}(E)) = c_1(E)\in H^2(X, \Z(1)).
$$

For an arbitrary vector bundle $E$, 
$c_1(E) := c_1({\rm det}(E))$. 
The construction works the same way for complex  manifolds.

The first step is algebraic: 
the class ${\rm cl}(E) \in H^1(X, {\cal O}^*)$  makes 
sense in Zariski topology. 

The second step is transcendental. 
The very existence of the integral class $c_1(E)$ reflects the failure 
of the complex logarithm  $\log (z)$ to satisfy the functional 
equation. And yet  the functional equation  $\log (xy) = \log x + \log y$ 
is satisfied on the real positive axis, and determines the logarithm uniquely.  

Our starting point  was the following problem: 
\vskip 2mm
{\it Find similar in spirit "local  formulas" for all Chern classes  
 of  a vector bundle on $X$}.
\vskip 2mm

\paragraph{2. The second Chern class.} The next is a local formula for the second Chern class. 
It is much deeper. We  discuss in Section \ref{sec1.2} the case when 
the vector bundle is two-dimensional - the case of an arbitrary vector bundle 
requires additional ideas, and postponed till Section 4. 

Here we see a similar phenomenon: the  local  formula  for the 
second Chern class in $H^4(X, \Z(2))$ requires the dilogarithm function, 
and  reflects all its beautiful properties at once:

\begin{itemize}

\item The monodromy of the dilogarithm.

\item The differential equation of the dilogarithm.

\item Abel's five term relation, or better the failure of the 
\underline{complex} dilogarithm  to satisfy it. Yet, the five term relation 
on the \underline{real} positive locus is clean, and 
determines the  dilogarithm. 

\end{itemize}

The relevance of the real dilogarithm for the first Pontryagin class  was discovered 
 by Gabrielov, Gelfand and Losik \cite{GGL}. Few years later,  
the relevance of the complex dilogarithm for the 
codimension two algebraic cycles  and regulators was discovered by Spencer 
Bloch \cite{B77}, \cite{B78}. 

Our formula for the second Chern class of a two-dimensional vector bundle 
is in the middle. 

\vskip 3mm

The construction of the universal second motivic Chern class from \cite{G93} had several 
applications in low dimensional geometry and mathematical physics, 
e.g. \cite{FG1}. It provides  a motivic point of view on the Chern-Simons theory. 
It is of cluster nature, 
and can be quantised using the quantum dilogarithm \cite{FG2}. 
The present paper just clarifies its Hodge-theoretic aspect. 

The local formula for the third motivic Chern class has the same 
level of precision. In particular it is of cluster nature. 
However, strangely enough, it did not have any application in geometry yet. 
Its quantisation is a tantalising open problem.


\subsection{Exponential complexes}

\bd The weight $n$  exponential complex $  \Q_{\bf E}^{\bullet}(n)$  is the following complex 
    of sheaves  
on a manifold $X$, concentrated in degrees $[0,n]$:
 \begin{equation} \label{ana1e}
{\cal O} (n-1) \lra {\cal O}^*  \otimes {\cal O} (n-2) \lra ... \lra  
\otimes^{n-1}{\cal O} ^*\otimes {\cal O}  
    \lra \otimes^{n }{\cal O}^*.
\end{equation} 
The differential is  
\begin{equation} \label{ana1*}
d:   \underbrace {{\cal O}^*\otimes \ldots \otimes{\cal O}^*}_{\text {k-1 times}}  \otimes 
{\cal O} \otimes 
\underbrace{2\pi i \otimes \ldots \otimes 2\pi i}_{\text {n-k times}}\lra 
 \underbrace {{\cal O}^*\otimes \ldots \otimes{\cal O}^*}_{\text {k times}} \otimes {\cal O} \otimes 
\underbrace{2\pi i \otimes \ldots \otimes 2\pi i}_{\text {n-k-1 times}}, 
\ee
\be
\begin{split}
&a_1 \otimes ... \otimes a_{k-1} \otimes b \otimes \underbrace{2\pi i \otimes \ldots \otimes 2\pi i}_{\text {n-k times}}   
\lms \\
&a_1 \otimes ... \otimes  a_{k-1} \otimes  {\rm exp} (b) 
\otimes    2\pi i\otimes\underbrace{2\pi i \otimes \ldots \otimes 2\pi i}_{\text {n-k-1 times}}. 
\end{split}
\ee
\ed

To check that we get a complex, observe that each map $d^2$ involves a factor ${\rm exp}(2\pi i) =1$: 
$$
\ldots  \otimes b \otimes  2\pi i \otimes 2\pi i \ldots \stackrel{d}{\lra} 
\ldots \otimes {\rm exp}(b) \otimes 2\pi i \otimes 2\pi i \ldots \stackrel{d}{\lra} 
\ldots {\rm exp}(b) \otimes {\rm exp}(2\pi i) \otimes 2\pi i   \ldots.
$$

 For example,  $  \Z_{\bf E}^{\bullet}(1)$ 
is the classical exponential resolution of $  \Z(1)$.

The 
complex $  \Q_{\bf E}^{\bullet}(2)$ looks as follows:
$$
{\cal O}(1) \lra  {\cal O}^* \otimes  {\cal O}
    \stackrel{}{\lra}  {\cal O}^* \otimes  {\cal O}^*.
$$
$$
b \otimes 2\pi i \lms {\rm exp}(b) \otimes 2\pi i, ~~~~a \otimes b \lms a \otimes {\rm exp}(b).
$$

The map $  \Q(n) \hra {\cal O}(n-1)$ gives rise to a map of complexes 
$  \Q(n) \lra   \Q_{\bf E}^{\bullet}(n)$. The cone of this map is acyclic. 
So the exponential complex is a resolution of the constant sheaf $  \Q(n)$. 

The holomorphic de Rham complex  on a complex manifold $X$ 
is a resolution of the constant sheaf ${ \C}$:
\be \la{DCOMP}
\Omega^\bullet:= ~~\Omega^0   \stackrel{d}{\lra}   \Omega^1  
  \stackrel{d}{\lra}  \Omega^{2} \stackrel{d}{\lra} ...  
\ee
Let $X$ be a regular complex algebraic variety.  
Take  a compactification $\overline X$ of $X$ such that ${\rm D}:= \overline X-X$ 
is a normal crossing divisor. The 
de Rham complex $\Omega_{\log}^\bullet$ of forms with logarithmic singularities at infinity 
is a complex of sheaves in the classical topology on $X$, given by 
the forms with logarithmic singularities at ${\rm D}$.

The canonical embedding 
$ {\Q}(n) \hra  {\C}$  
gives rise to a canonical morphism of the resolutions 
\be \la{OMEGA}
\Omega_n^{(\bullet)}: {  \Q}^\bullet_{\bf E}(n) ~~\lra ~~\Omega_{\rm log}^\bullet
\ee
defined in the next Lemma. 

\bl 
There is 
a canonical morphism of  complexes of sheaves on $X$: 
$$
\begin{array}{ccccccc}
 {\cal O}(n-1) &    \lra & {\cal O}^*\otimes {\cal O}(n-2) &    \lra ... \lra & \otimes^{n-1}   {\cal O}^* \otimes {\cal O} & \lra&\otimes^{n }   {\cal O}^*\\
&&&&&& \\
\downarrow \Omega^{(0)}_n  &&\downarrow \Omega^{(1)}_n  
&  ... &\downarrow \Omega^{(n-1)}_n&& \downarrow \Omega^{(n )}_n\\
&&&&&& \\
 \Omega^0  & \stackrel{d}{\lra} &\Omega_{\rm log}^1  &    
 \stackrel{d}{\lra} ... \stackrel{d}{\lra} &\Omega_{\rm log}^{n-1}  & \lra &  \Omega^{n}_{{\rm log},  cl}
\end{array}
$$ 
Here 
$$
\Omega^{(m)}_n( (2\pi i)^{n-m-1} \cdot f_1 \otimes ... \otimes f_m\otimes g) := 
$$
$$
(2\pi i)^{n-m-1} 
  (-1)^mg \cdot d \log f_1 \wedge ... \wedge d \log f_m,   \qquad  m<n,
$$ 
$$
 \Omega^{(n )}_n(f_1 \otimes 
... \otimes f_n) := (-1)^nd \log  f_1  \wedge ... \wedge d \log f_n.
$$
\el

\subsection{Exponential Deligne complexes.} 

Let $X$ be a complex manifold. 
Consider a subcomplex  of the holomorphic de Rham complex:
\be \la{HF}
F^n\Omega^{\bullet} := \Omega^n \to \Omega^{n+1} \to 
\ldots \subset \Omega^{\bullet}.
\ee
 The  weight $n$ rational Deligne complex  on $X$ is defined as a complex of sheaves
\be \la{DCOMPas}
{  \Q}_{{\cal D}}(n): = {\rm Cone}\Bigl({  \Q}(n) \oplus F^n\Omega^{\bullet} \lra  
\Omega^{\bullet}\Bigr)[-1]. 
\ee 
Complex (\ref{DCOMPas})   is quasiisomorphic to 
\be \la{DCOMP21}
{  \Q}(n) \hookrightarrow  \Omega^0   \stackrel{d}{\lra}   \Omega^1  
  \stackrel{d}{\lra} ... \stackrel{d}{\lra}  \Omega^{n-1}.
\ee

Let $X$ be a regular complex algebraic variety.  
The Beilinson-Deligne complex ${  \Q}_{{\cal D}}(n)$ \cite{B84} is
a complex of sheaves in the classical topology on $X$ given 
by the total complex 
of the bicomplex
$$
\begin{array}{ccccccccccccccc}
&&&&{  \Q}(n)&&&&&&\Omega^n_{\log}&\stackrel{d}{\lra}& \Omega^{n+1}_{\log}&\stackrel{d}{\lra}&\ldots\\
{  \Q}_{{\cal D}}(n):=&&&&\downarrow &&&&&&\downarrow =&&\downarrow =&&\\
&&&&\Omega_{\log}^0& \stackrel{d}{\lra}&\Omega_{\log}^1& \stackrel{d}{\lra}&\ldots&\stackrel{d}{\lra} &\Omega_{\log}^n& \stackrel{d}{\lra}&\Omega_{\log}^{n+1}&\stackrel{d}{\lra}&\ldots
\end{array}
$$

\bd The weight $n$ exponential Deligne complex is a complex 
\be \la{DCOMPas1}
{  \Gamma}_{{\cal D}}(n): = 
{\rm Cone}\Bigl({  \Q}_{\bf E}^{\bullet}(n) \oplus F^n\Omega^{\bullet}_X \lra  \Omega^{\bullet}_X\Bigr)[-1] 
\ee 
obtained by replacing ${  \Q}(n)$ in (\ref{DCOMPas}) by  its exponential resolution 
${  \Q}_{\bf E}^{\bullet}(n)$, and using the map  (\ref{OMEGA}).
\ed

For example, when $n=2$ we get the total complex of the following bicomplex: 
 $$
\begin{array}{ccccccccccccc}
&&{\cal O}(1)&\lra &{\cal O}^*\otimes {\cal O}&\lra 
&{\cal O}^*\otimes {\cal O}^*&\bigoplus& \Omega^2_{\log}&\stackrel{d}{\lra}& 
\Omega^{3}_{\log}&\stackrel{d}{\lra}&\ldots\\
&&&&&&&&&&&\\
  \Gamma_{\cal D}(2):=&&\downarrow \Omega^{(0)}_2 &&\downarrow \Omega^{(1)}_2&&~~~~~~~~\Omega^{(2)}_2\searrow &&\swarrow = ~ ~~~~~~~
&&\downarrow =&&\\
&&&&&&&&&&&&\\
&&\Omega^0& \stackrel{d}{\lra}&\Omega_{\log}^1& \stackrel{d}{\lra}&
&\Omega_{\log}^2&&\stackrel{d}{\lra} &\Omega_{\log}^{3}&\stackrel{d}{\lra} &\ldots
\end{array}
$$ 
The quotient of  complex (\ref{DCOMPas1}) by the acyclic subcomplex 
$
{\rm Cone}\Bigl(F^n\Omega^{\bullet}_X \lra  F^n\Omega^{\bullet}_X\Bigr)[-1] 
$ 
is a quasiisomorphic complex

$$
\begin{array}{cccccccc}
&{\cal O}(n-1)&\lra &{\cal O}^*\otimes {\cal O}(n-2)&\lra \ldots \lra 
&{\otimes^{n-1}}{{\cal O}^*}
\otimes {\cal O}& \lra &{\otimes^n}{{\cal O}^*}\\
&\downarrow &&\downarrow &&\downarrow &&  \\
&\Omega^0& \stackrel{d}{\lra}&\Omega_{\log}^1& \stackrel{d}{\lra} \ldots 
\stackrel{d}{\lra}&\Omega_{\log}^{n-1}&& 
\end{array}
$$

\subsection{Period morphisms} 

Recall the basic fact, reviewed in Section 2.1,  that the 
equivalence classes of variations of framed mixed $\Q$-Hodge-Tate structures on a complex manifold $X$ 
give rise to a sheaf of graded commutative Hopf algebras over $\Q$: 
$$
  {\cal H}_\ast = \bigoplus_{n=0}^\infty{\cal H}_n. 
$$
One has ${\cal H}_0=\Q, ~{\cal H}_1={\cal O}^*_\Q:= {\cal O}^*\otimes \Q$. 
The reduced coproduct $\Delta':   {\cal H}_{>0} \lra \otimes^2  {\cal H}_{>0}$
 give rise to the reduced cobar complex, graded by the weight:
$$
  {\cal H}_{>0}   \stackrel{\Delta'}{\lra} 
  {\cal H}_{>0} \otimes   {\cal H}_{>0}    \stackrel{\Delta'}{\lra} ... 
\stackrel{\Delta'}{\lra}  \otimes^n   {\cal H}_{>0}. 
$$
In Section 3 we present our main  construction,  valid in the category of complex manifolds:  
\begin{theorem} \label{periods1A}
There is a canonical map of complexes of sheaves, called the \underline{period morphism}:
\be \la{4.22.15.1A}
\mbox {\it the weight $n$ part of cobar complex 
of $\underline {\cal H}_\ast$} \lra \mbox{\it the weight $n$ exponential complex 
$  \Q_{\bf E}^{\bullet}(n)$}. 
\ee
In a more elaborate form, it looks as follows: 
\be \la{REDVEDC}
\begin{array}{ccccccc}
&&{  {\cal H}}_n  & \stackrel{\Delta'}{\lra} & 
({  {\cal H}_{>0} } \otimes {  {\cal H}_{>0} } )_n  & \stackrel{\Delta'}{\lra} ... 
\stackrel{\Delta'}{\lra} & \otimes^n {  {\cal H}}_1 \\
&&&&&&\\
&&\downarrow P_n^{1} &&\downarrow P_n^{2}   &&=\downarrow P_n^{n} \\
&&&&&&\\
 {\cal O}(n-1) &\lra &{\cal O}^* \otimes {\cal O}(n-2) & \lra &   \otimes^2  {\cal O}^* \otimes {\cal O}(n-3) 
& \lra... \lra & \otimes^n  {\cal O}^*_\Q\\
\end{array}
\ee
The map (\ref{REDVEDC}) has the following properties: 
\begin{enumerate}

\item After the identification  
${  {\cal H}}_1 = {\cal O}^*_\Q $ the map $P_n^{n}$ is the identity map.

\item The map $P_n^{1}$ is the { big period map} from \cite{G96}. 
 
\item The composition   $\Omega_n^{k}\circ  P_n^{k}$ is zero unless $k=n$, i.e. 
everywhere except on the very right. 
\end{enumerate} 
\end{theorem}

Condition 3) just means that the following composition is zero:
$$
\begin{array}{ccccc}
{  {\cal H}}_n  & \stackrel{\Delta'}{\lra} & 
({  {\cal H}}_{>0} \otimes {  {\cal H}}_{>0})_n  & \stackrel{\Delta'}{\lra} ... 
\stackrel{\Delta'}{\lra} & (\otimes^{n-1} {  {\cal H}}_{>0})_n \\
\downarrow  &&\downarrow    &&\downarrow   \\
 {\cal O}^* \otimes {\cal O}(n-2) & \lra &      {\cal O}^*\otimes   {\cal O}^* \otimes  {\cal O}(n-3)& \lra ... \lra &  
\otimes^{n-1}  {\cal O}^* \otimes{\cal O}\\
\downarrow   &&\downarrow    &&\downarrow   \\
\Omega^1 & \stackrel{d}{\lra} & \Omega^2 &  \stackrel{d}{\lra} ... \stackrel{d}{\lra} &  \Omega^{n-1}
\end{array}
$$

\paragraph{Example: $n=2$.} Then we have a map
$$
\begin{array}{ccccccc}
&&{  {\cal H} }_2 & \lra &  \otimes^2   {  {\cal H} }_1\\
&&\downarrow &&\downarrow =\\
{\cal O}(1)& \lra &  {\cal O}^*\otimes    {\cal O}&  \lra & \otimes^2{\cal O}_\Q^*\\
\end{array}
$$

\paragraph{The period morphism to the exponential Deligne complex.} 
Let us use  period morphism (\ref{REDVEDC}) and its properties 
provided by Theorem \ref{periods1A} to define  
a map of complexes of sheaves
\be \la{CHRMDC1}
\begin{split}
&\mbox{\it the weight $n$ part of the reduced cobar complex of $  {\cal H}_\bullet$} \lra \\
&\mbox{\it the weight $n$ exponential Deligne complex $\Gamma_{\cal D}(X; n)$}. 
\end{split}
\ee
Let us recall that 
\be \la{DCOMPas1a1}
{  \Gamma}_{{\cal D}}(X; n) = 
{\rm Cone}\Bigl({  \Q}_{\bf E}^{\bullet}(X; n) 
\oplus F^n\Omega^{\bullet}_X \lra  \Omega^{\bullet}_X\Bigr)[-1]. 
\ee 
Therefore a map to the complex (\ref{DCOMPas1a1})  has three components: 

\begin{enumerate}

\item The exponential complex 
${  \Q}_{\bf E}^\bullet(X, n)$ component;

\item The Hodge filtration $F^{n}\Omega^{\bullet}$ component;

\item The de Rham complex $\Omega_X^\bullet$ component.
\end{enumerate}

We define these components as follows. 

\begin{enumerate}

\item 
The  
${  \Q}_{\bf E}^\bullet(X, n)$-component is just  
the period morphism: 
$$
\begin{array}{ccccccc}
&&{  {\cal H}}_n  & \stackrel{\Delta'}{\lra} & 
({  {\cal H}_{>0}} \otimes {  {\cal H}_{>0}} )_n  & \stackrel{\Delta'}{\lra} ... 
\stackrel{\Delta'}{\lra} & \otimes^n {  {\cal H}}_1 \\
&&&&&&\\
&&\downarrow P_n^{1} &&\downarrow P_n^{2}   &&=\downarrow P_n^{n} \\
&&&&&&\\
 {\cal O}(n-1) &\lra &{\cal O}^* \otimes {\cal O}(n-2) & \lra &   \otimes^2  {\cal O}^* \otimes {\cal O}(n-3) 
& \lra... \lra & \otimes^n  {\cal O}^*_\Q\\
\end{array}
$$

\item 
The $F^{n}\Omega^{\bullet}$-component 
is given by  the map 
\be \la{2mapcompl}
-\otimes^nd\log: \otimes^n  {\cal H}_1  \lra \Omega_X^n, 
~~~~(f_1, ..., f_n)\lms -d\log f_1 \wedge \ldots \wedge d\log f_n. 
\ee

\item 
The de Rham complex component is  zero. \end{enumerate}

Here is how the map (\ref{CHRMDC1}) looks in the weight two. 
The top row is the weight 2 reduced cobar complex. The second and third rows 
provide us a bicomplex whose total is the  weight two exponential Deligne complex. 
The map is given by the first row of 
vertical arrows:
 $$
\begin{array}{ccccccccccccccc}
&&&&&&{  {\cal H}}_2&\stackrel{\Delta'}{\lra}  & 
&{  {\cal H}}_1 \otimes {  {\cal H}}_1&&&&\\
&&&&&&&&&&&&&\\
&&&&&&&&&&&&&\\
&&&& &&\downarrow P_2^{1} &&~~~~~\swarrow =&&\searrow -\otimes^2 d\log &&&\\
&&&&&&&&&&&&&\\
&&&&&&&&&&&&&\\
&&&&{\cal O}(1)&\lra &{\cal O}^*\otimes {\cal O}&\lra 
&{\cal O}^*\otimes {\cal O}^*&\bigoplus& \Omega_{\rm log}^2&\stackrel{d}{\lra}& 
\Omega_{\rm log}^{3}&\stackrel{d}{\lra}&\ldots\\
&&&&&&&&&&&&&\\
&&&&\downarrow = &&\downarrow \Omega_2^{(1)}&&~~~~~~~~\Omega_2^{(2)}\searrow &&\swarrow =~ ~~~~~~~
&&\downarrow =&&\\
&&&&&&&&&&&&&&\\
&&&&\Omega^0& \stackrel{d}{\lra}&\Omega_{\rm log}^1& \stackrel{d}{\lra}&
&\Omega_{\rm log}^2&&\stackrel{d}{\lra} &\Omega_{\rm log}^{3}&\stackrel{d}{\lra} &\ldots
\end{array}
$$ 
\bt
The map defined by the components 1)-3) is a homomorphism of complexes. 
\et

\begin{proof}
By Theorem \ref{periods1A}, the component 1) is a homomorphism of complexes. 
The component 2) is also a homomorphism of complexes. 
Indeed, the forms in the image of the map (\ref{2mapcompl}) are evidently closed. 
So the statement reduces to the claim that the following composition is zero: 
$$
\Bigl(\otimes^{n-1}  {\cal H}_{>0}\Bigr)_n \stackrel{\Delta'}{\lra} \otimes^n  {\cal H}_1  \lra \Omega^n.
$$
This follows from the $n=2$ case, telling that (see  Theorem 2.11) 
thanks to the Griffith transversality, 
the following composition is zero: 
$$
  {\cal H}_2 \stackrel{\Delta'}{\lra}
   {\cal H}_1\otimes  {\cal H}_1  \lra \Omega^1.
$$
After that the Theorem reduces to properties 1) and 3) of 
the period map in Theorem \ref{periods1A}.
\end{proof}

Alternatively, using the reduced model (\ref{REDVEDC}) for
 the exponential Deligne complex, the homomorphism (\ref{CHRMDC1}) is given  
\underline{just} by the period morphism: 
$$
\begin{array}{ccccccc}
  {\cal H}_n& \stackrel{\Delta'}{\lra}& 
\Bigr(  {\cal H}_{>0} \otimes  {\cal H}_{>0}\Bigl)_n
&\stackrel{\Delta'}{\lra}\ldots \stackrel{\Delta'}{\lra}&
\Bigl(\otimes^{n-1}  {\cal H}_{>0} \Bigr)_n&
\stackrel{\Delta'}{\lra}&\otimes^n   {\cal H}_1\\
&&&&&&\\
\downarrow P_n^{(1)}&&\downarrow P_n^{(2)}&&\downarrow P_n^{(n-1)}&&\downarrow P_n^{(n)}\\
&&&&&&\\
{\cal O}(n-1)&\to &{\cal O}^*\otimes {\cal O}(n-2)&\to \ldots \to 
&{\otimes^{n-1}}{\cal O}^*
\otimes {\cal O}& \to &{\otimes^n}{\cal O}^*_\Q\\
\downarrow &&\downarrow &&\downarrow && \\
\Omega^0& \stackrel{d}{\lra}&\Omega_{\rm log}^1& \stackrel{d}{\lra} \ldots 
\stackrel{d}{\lra}&\Omega_{\rm log}^{n-1}&& 
\end{array}
$$ 

\subsection{A map: Bloch complex $\to$ weight two  
 exponential Deligne complex}

\paragraph{\it The Bloch complex as a ``resolution'' of Milnor's $K_2$.} 
Given a field $F$, the Milnor group $K_2(F)$ is 
the quotient of the group $F^* \otimes F^*$ by a subgroup generated by 
{\it Steinberg relations} $(1-x)\otimes x$ where $x \in F^*-\{1\}$ \cite{Mi71}. 
Since $x\otimes y + y \otimes x$ is a sum of Steinberg relations, 
\be \la{kk}
K_2(F) = \frac{\wedge^2F^*}{\mbox{subgroup generated by Steinberg relations}}.
\ee
In other words, the group $K_2(F)$ is the cokernel of the map
$$
\delta: \Z[F^*-\{1\}] \lra {\wedge}^2F^*, \qquad \{x\} \lms (1-x) \wedge x.
$$
where $\{x\}$ is the generator of $\Z[F^* - \{1\}]$ corresponding to an $x \in F^*- \{1\}$.

Recall the cross-ration of four points on the projective line: 
\be \la{item24p1}
r(s_1, s_2, s_3, s_4) := \frac{(s_1- s_4) (s_2- s_3) }{(s_1- s_3) (s_2- s_4) }.
\ee Let $R_2(F)$ be the 
subgroup of $\Z[F^*-\{1\}]$ generated by the ``five term relations''
\be \la{item4}
\sum_{i=1}^5 (-1)^i\{r(s_1, ..., \widehat s_i, ..., s_5)\}, \qquad s_i \in {\rm P}^1(F), 
\quad s_i \not = s_j. 
\ee
It is well known that $\delta (R_2(F)) =0$ (see 
Lemma 1.8).  
Let us set
$$
B_2(F):= \frac{\Z[F^*-\{1\}]}{R_2(F)}.
$$
Then the map $\delta$ gives rise to a homomorphism
\be \la{BLCOM}
\delta: B_2(F) \lra {\wedge}^2F^*.
\ee
Let $\{x\}_2\in B_2(F)$ be  the 
image of $\{x\}$.   
We add   $\{0\}_2 =  \{1\}_2 =  \{\infty\}_2 =0$, annihilated by $\delta$. 
We view (\ref{BLCOM}) as a complex, called the {\it Bloch complex} \cite{B78}, \cite{Su82}, \cite{DS82}, 
placed
 in degrees $[1,2]$.

Consider a twin of the weight two exponential complex, which we call  
the weight two {\rm Lie}-exponential complex,\footnote{The prefix {\rm Lie} refers to the fact that the period map in this case is a map from the standard Chevalley-Eilenberg complex of the Lie coalgebra ${\cal L}_\ast$ 
associated with the Hopf algebra ${\cal H}_\ast$. See Section 2.4 for the definition of {\rm Lie}-exponential complexes and discussion of the {\rm Lie}-period maps for them.} which   
 is a complex of sheaves  on $X$ 
in degrees $[0,2]$:
\be \la{EXPC}
  \Q_{\cal E}^\bullet(2):= ~~~~  {\cal O}(1)\lra  \Lambda^2 {\cal O} \stackrel{\wedge^2{\rm exp}}{\lra} \Lambda^2 {\cal O}^*.
\ee 
The differentials are given as follows:
\be \la{DEXPC}
2\pi i\otimes a \lms 2\pi i \wedge a, ~~~~
a \wedge b \lms {\rm exp}(a) \wedge {\rm exp}(b). 
\ee
 There is a canonical map of complexes: 
$$
 \Q(2) \lra   \Q_{\cal E}^\bullet(2), ~~~~
(2\pi i)^2 \lms 2\pi i \otimes 2\pi i.
$$ 
Therefore one can easily see  that $ \Q_{\cal E}^\bullet(2)$ is a resolution of $ \Q(2)$.

Let us sheafify the Bloch complex to a complex of sheaves on $X$: 
\be \la{SBC}
  {\rm B}^\bullet(2):= ~~~~  B_2({\cal O}) \lra \Lambda^2 {\cal O}^*.
\ee
Let us define a map of complexes 
\be \la{4.8.15.1aaa}
\begin{array}{cccccccc}
  &&&{  B}_2({\cal O}) & \lra & \Lambda^2 {\cal O}^*\\
&&&&&\\
&&&\downarrow p_2&& \downarrow =\\
 &&&&&\\
 &{\cal O}(1)&\lra  &\Lambda^2 {\cal O}& 
\stackrel{\wedge^2{\rm exp}}{\lra} &\Lambda^2 {\cal O}^*
\end{array}
\ee
 To define the  homomorphism $p_2$, we  set
\be \la{4.10.15.1}
{\rm Li}_2(x):= \int_0^x\frac{dt}{1-t}\circ \frac{dt}{t}, ~~~~-\log (1-x)= \int_0^x\frac{dt}{1-t},
~~~~ \log x:= \int_0^x \frac{dt}{t}.
\ee
Here all integrals are along the same path from $0$ to $x$. The last one is regularised  
using the tangential base point  at $0$ dual to $dt$.  
When $|x| <1$, we have standard power series expansions
$$
-\log (1-x) = \sum_{n=1}^\infty \frac{x^n}{n}, ~~~~{\rm Li}_2(x) = \sum_{n=1}^\infty \frac{x^n}{n^2}.
$$
Then  we set, modifying slightly the original construction of Spencer Bloch \cite{B78},  
$$
{\rm L}_2(x) :=  {\rm Li}_2(x) + \frac{1}{2}\cdot  \log (1-x) \log x +\frac{(2\pi i)^2}{24},  
$$
$$
p_2(\{x\}_2):= \frac{1}{2}\cdot \log(1-x) \wedge \log x + 2\pi i \wedge 
\frac{1}{2\pi i}{\rm L}_2(x).
$$
Notice that  $2\pi i\wedge \frac{2\pi i}{24}=0$ in $\Lambda^2\C$. Indeed, for any integer $N$ we have 
$
2\pi i\wedge \frac{2\pi i}{N}= 
-N \cdot \frac{2\pi i}{N}\wedge \frac{2\pi i}{N} = 0.$ 
Yet it is handy to keep the summand $\frac{(2\pi i)^2}{24}$ in ${\rm L}_2(x)$, although 
it does not change $2\pi i \wedge 
\frac{1}{2\pi i}{\rm L}_2(x)$. 

\bl
i) The map $p_2$ is well defined on $\Z[\C^*-\{1\}]$, i.e. does not depend on the monodromy of 
the logarithms and the dilogarithm along the path $\gamma$ in (\ref{4.10.15.1}).  

ii) The map $p_2$ sends the five term relations to zero. 
\el

\begin{proof} The part i) is easy to check using well known monodromy properties of the dilogarithm.

Let us prove the five term relation. Recall   
the map 
$$
\delta_2: \Z[\C(t)^*-\{1\}] \lra \C(t)^* \wedge \C(t)^*, ~~~~\{x\} \lms (1-x) \wedge x.
$$ 
Then we have a commutative diagram: 
$$
\begin{array}{ccccccc}
&&{\rm Ker}~\delta_2&\lra &\Z[\C(t)^*-\{1\}]]&\stackrel{\delta_2}{\lra}&\C(t)^* \wedge \C(t)^*\\
&&&&&&\\
&&\downarrow &&\downarrow p_2&&\downarrow =\\
&&&&&&\\
&&\C(t)(1)&\lra&\C(t) \wedge \C(t)& \stackrel{\rm exp}{\lra} & \C(t)^* \wedge \C(t)^*.
\end{array}
$$
It implies that $p_2({\rm Ker}~\delta_2) \subset 2\pi i \wedge \C(t)$. 
Next, let us consider a map
\be \la{mapw}
\omega: \Lambda^2 {\cal O}\lra {\Omega}^1, ~~~~f\wedge g \lms \frac{1}{2}(fdg-gdf).
\ee
The differential equation for the dilogarithm function is 
\be \la{deqdil}
d{\rm L}_2(x) = \frac{1}{2}\cdot \Bigl(- \log (1-x) ~d\log x  +  \log x ~d\log (1-x) \Bigr).
\ee
It just means that the following composition  is zero:
\be \la{deqdilq}
\Z[\C(t)^*-\{1\}]] \stackrel{p_2}{\lra}
\Lambda^2 \C(t) \stackrel{\omega}{\lra} {\Omega}_{t}^1.
\ee
The kernel of the map $\omega: 2\pi i \wedge \C(t) \lra \Omega^1$ is $2\pi i \wedge \C$. 
This implies that 
$$
p_2({\rm Ker}~\delta_2) \in 2\pi i \wedge \C. 
$$
Given a configuration of five distinct 
points $(x_1, ..., x_5)$ on $\C{\Bbb P}^1$, 
denote by $R_2(x_1, ..., x_5)\in \Z[\C]$ the corresponding five-term relation element (\ref{item4}). 
Since it lies in the kernel of the map $\delta_2$, 
applying the map $p_2$ to it we get a constant: 
$$
c(x_1, ..., x_5) := p_2\Bigl(R_2(x_1, ..., x_5)\Bigl) \in 2\pi i \wedge \C.
$$ 
Let us calculate this constant. Similar argument shows that we have constants
$$
b(x) := p_2(\{x\}_2 + \{1-x\}_2) \in  2\pi i \wedge \C, ~~~~
 c(x) := p_2(\{x\}_2 + \{x^{-1}\}_2) \in  2\pi i \wedge \C. 
$$
One has $b(x) = c(x)$. Indeed, they tautologically 
coincide if $x$ solves the equation 
$1-x=x^{-1}$. Thus they must coincide for any $x \in \C^*-1$.
On the other hand, 
switching the last two points in the cross-ratio 
we get $r(x_1, x_2, x_3, x_4) = r(x_1, x_2, x_4, x_3)^{-1}$. Therefore   
$$
c(x_1, x_2, x_3, x_4, x_5) + c(x_1, x_2, x_3, x_5, x_4) = c(x). 
$$
Finally, $b(x) =0$ for $x \in (0,1)$. Indeed, 
$\log (1-x) \wedge \log x + \log x \wedge \log (1-x) =0$, 
and each term of ${\rm L}_2(x)$ is well defined if $x \in (0,1)$. So  it 
is sufficient to 
show that ${\rm L}_2(x) +  {\rm L}_2(1-x) =0$. One has $d({\rm L}_2(x) +  {\rm L}_2(1-x))=0$. 
The limit of ${\rm L}_2(x) +  {\rm L}_2(1-x)$ as $x\to 1$ is 
$0$ due to ${\rm Li}_2(1) = \pi^2/6$.
\end{proof}

Recall that the weight two rational Deligne complex $  \Q^\bullet_{\cal D}(2)$  
  is a complex of sheaves on $X$ in degrees $[0,2]$:
$$
\begin{array}{ccccccc}
&&&&  \Q(2)& &\\
 \Q_{\cal D}^\bullet(2):=&&&&\downarrow &&  \\
&&&&{\cal O}& \stackrel{d}{\lra}&\Omega^1
\end{array}
$$ 

Consider the following version of the exponential Deligne complex, which we call 
the weight two {{\rm Lie}-exponential Deligne complex}, 
and abusing notation denote also by  $  \Gamma_{\cal D}(2)$, obtained  
by replacing the constant sheaf $ \Q(2)$ 
by its {\rm Lie}-exponential resolution $ \Q_{\cal E}^\bullet(2)$. It is 
a complex of sheaves in the classical topology on  $X$ associated with the following bicomplex: 
$$
\begin{array}{ccccccccc}
&&&&{\cal O}(1)&\lra &{\cal O}\wedge {\cal O}&\lra 
&{\cal O}^*\wedge {\cal O}^*\\
 \Gamma_{\cal D}(2):=&&&&\downarrow =&&\downarrow \omega&& \\
&&&&{\cal O}& \stackrel{d}{\lra}&\Omega^1& &
\end{array}
$$ 

\bp \la{16.10.15.1}
There is a canonical morphism of 
complexes of sheaves 
\be \la{RD}
r_{\cal D}:   {\rm B}^\bullet(2) \lra   \Gamma_{\cal D}(2). 
\ee
\ep

\begin{proof} Let us define the  map (\ref{RD}) as a morphism of complexes:
$$
\begin{array}{cccccc}
  &&&{  B}_2({\cal O}) &\lra  &\Lambda^2 {\cal O}^*\\
&&&&&\\
&&&\downarrow p_2&&\downarrow {\rm Id}\\
&&&&&\\
 &{\cal O}(1)&\lra  &\Lambda^2 {\cal O}&\lra  &\Lambda^2 {\cal O}^*\\
&\downarrow =  &&\downarrow \omega &&\\
&{\cal O} &\lra &\Omega^1&&\\
\end{array}
$$ 
Here the top raw is the sheafified Bloch complex, and the bottom two raws 
describe the weight two {\rm Lie}-exponential Deligne complex. 
The morphism of the first raw 
to the second is given by the maps $(p_2, {\rm Id})$. The other components of 
the morphism are zero. 

To show that this is a map of complexes we use two facts:

1. The top right square is commutative by the definition of the map $p_2$.

2. The composition ${  B}_2({\cal O})  \stackrel{p_2}{\lra}
\Lambda^2 {\cal O} \stackrel{\omega}{\lra} {\Omega}^1$ is zero by the 
differential equation for the dilogarithm. 

\end{proof}

\paragraph{Applications to regulators.} Let us look at the dilogarithm 
regulator map  for  ${\rm Spec}(\C)$: 
\be \la{4.8.15.1a}
\begin{array}{cccccccc}
  &&&&&{B}_2({\C}) & \lra & \Lambda^2 {\C}^*\\
&&&&&\downarrow p_2&& \downarrow =\\
& & & {\C}^*(1)&\lra  &\Lambda^2 {\C}& 
\stackrel{\wedge^2{\rm exp}}{\lra} &\Lambda^2 {\C}^*
\end{array}
\ee
It implies that there is a canonical map 
$$
{\rm Ker}\Bigl({B}_2({\C}) \lra \Lambda^2 {\C}^*\Bigr) \lra \C^*(1) = {\rm Ker}\Bigl(\Lambda^2 {\C} \lra \Lambda^2 {\C^*}\Bigr).  
$$
According to a theorem of Suslin \cite{Su84}, one has 
$$
{\rm Ker}\Bigl({B}_2({\C}) \lra \Lambda^2 {\C}^*\Bigr) \otimes \Q = K_3^{\rm ind}(\C) \otimes \Q. 
$$
So we get an explicit construction of Beilinson's regulator map 
$$
K_3^{\rm ind}(\C) \lra \C^*(1). 
$$


\subsection{Regulator maps:  motivic complexes  $\to$  exponential Deligne complexes}

\paragraph{Motivic complexes and regulators.} According to Beilinson \cite{B87}, 
for any scheme $X$ over $\Q$, and for each  integer 
$n \geq 0$,  
one should have  a complex of sheaves $ \Z_{\cal M}(X; n)$ in the Zariski topology on $X$, 
called the {\it weight $n$ 
motivic complex of sheaves on $X$}, 
well defined  in the derived category. 
For example, $ \Z_{\cal M}(X; 0)= \Z$, and $ \Z_{\cal M}(X; 1)={\cal O}^*_X[-1]$. 
Beilinson's formula relates its cohomology 
to the weight $n$ pieces for the Adams filtration on  Quillen's K-groups of $X$: 
\be \la{BF}
H^i( \Z_{\cal M}(X; n)\otimes \Q)  
\stackrel{?}{=} {\rm gr}^n_{\gamma}K_{2n-i}(X)\otimes \Q. 
\ee
Beilinson defined higher regulator maps, with the source understood by (\ref{BF}): 
$$
H_{\rm Zar}^i( \Z_{\cal M}(X; n)\otimes \Q) \lra H^i(X,  \Q_{\cal D}(n)).
$$

Let $X$ be a regular complex algebraic variety. 
We want to have  higher regulator maps on the level of complexes.  
Motivic complexes are complexes of sheaves in the Zarisky topology 
on $X$, while the Beilinson-Deligne complexes are complexes of sheaves 
in the classical topology on $X$. To relate them, let us consider a map of sites
$$
\pi: \mbox{\rm Classical site} \lra \mbox{\rm Zariski site}. 
$$
Then the problem 
is interpreted as a problem of of construction of a map of complexes
\be \la{CLZAR}
{  \Z}_{\cal M}(X;n) \lra 
{\rm R}\pi_* {  \Q}_{\cal D}(X;n).
\ee


We address this problem at the generic point ${\cal X}$ of $X$ -- 
this is sufficient for local explicit formulas for the Chern classes. 
Notice that the ${\rm R}\pi_*$  is highly non-trivial 
since the constant sheaf $  \Q_{\cal X}$ has complicated cohomology at the generic point. 

It  is unlikely 
that one can construct a map just to the  Beilinson-Deligne complex 
 on ${\cal X}$. 

Our point is that replacing the constant sheaf $  \Q_{\cal X}$ by its exponential resolution 
and considering the exponential Deligne complex $\Gamma_{\cal D}({\cal X}; n)$, one should be able to define a map of complexes 
\be \la{111222333}
{  \Z}_{\cal M}({\cal X};n) \lra 
\pi_* {  \Gamma}_{\cal D}({\cal X};n). 
\ee
Combining it with the map  
$
\pi_* {  \Gamma}_{\cal D}({\cal X};n) \to
R\pi_* {  \Gamma}_{\cal D}({\cal X};n) 
$ 
we get a regulator map  
(\ref{CLZAR}) for ${\cal X}$.

Here is our strategy to define a map (\ref{111222333}). 
We make the following  assumption:
\vskip 2mm
{\it  The 
 motivic complex $  \Q_{\cal M}({\cal X}; n)$  
can be constructed as the weight $n$ part of the cobar complex 
of a graded commutative Hopf algebra ${\cal A}_\ast({\cal X})$, the motivic Tate Hopf algebra, 
graded by $\Z_{\geq 0}$.} 
\vskip 2mm

Then the Hodge realisation  provides a map of Hopf algebras 
$$
{\cal A}_\ast({\cal X}) \lra {\cal H}_\ast({\cal X}). 
$$ 
It induces a map of their cobar complexes:
\be \la{4.22.15.3a}
\begin{split}
&\mbox {\it the weight $n$ part of the cobar complex of ${\cal A}_\ast({\cal X})$} \lra  \\
&\mbox{\it the weight $n$ part of the cobar complex of ${\cal H}_\ast({\cal X})$}. 
\end{split}
\ee

Composing (\ref{4.22.15.1A})  and (\ref{CHRMDC1}) we arrive at a map of complexes
\be \la{4.22.15.2}
\begin{split}
{\cal P}:   \Q_{\cal M}({\cal X}; n)\lra  
\mbox{\it the weight $n$ exponential Deligne complex $ {\Gamma}_{\cal D}({\cal X}; n)$}. 
\end{split}
\ee

The induced map on the 
cohomology provides  higher regulators.

Using the polylogarithmic complexes, we can avoid 
assumptions about the existence of the motivic Hopf algebra when $n \leq 3$. 
So in this case the construction goes through unconditionally. 
In general there is the Bloch-Kriz construction of the 
motivic Tate Hopf algebra \cite{BK}. However their Hodge realisation map
 deserves a more explicit construction. 

\input{periodsc2.tex}

\subsection{Explicit formulas for the universal Chern classes.} 

Let us formulate 
our approach to local formulas for the Chern classes. Denote by $BGL^*_{N\bullet}$ the classifying space for $GL_N$. The $\ast$ stands for an open ``generic'' part $BGL^*_{\bullet}$ of Milnor's $BGL_{\bullet}$, which is 
a model of the classifying space.  
One should construct
universal  Chern classes  of $BGL^*_{N\bullet}$ with values in the exponential Deligne complex: 
\be \la{cnan}
c^{\cal D}_n \in H^{2n}(BGL^*_{N\bullet},  {\Gamma}_{\cal D}(n)).
\ee
They  induce  explicit cocycles for the 
Chern classes in a given Cech cover. 

We define the universal Chern classes  in three steps. 

\begin{enumerate}

\item An explicit formula for the Chern classes
with values in the 
  {\it bigrassmannian complexes} 
${\rm BC}(n)$ \cite{G93}. 

\item A map from the bigrassmannian complex to a motivic complex. 

There are several flavors of the problem, depending on our choice  
of the motivic complex. 

When $n=1,2,3$ there is a map of the bigrassmannian complex  
to the polylogarithmic motivic complexes ${\rm B}^\bullet(n)$. The latter reflects
 the motivic nature of the classical 
polylogarithms. For example, ${\rm B}^\bullet(2)$ is the Bloch complex. 
For $n=4$ there is also an explicit map to the motivic complex. 
So for $n \leq 4$ there is a satisfactory 
construction. 

So we should get  the universal motivic Chern class\footnote{Here we do not need to restrict to 
$BGL_N^*$.}
\be \la{cnmot}
c^{\cal M}_n \in H_{\rm Zar}^{2n}(BGL_{N\bullet},  {\Z}_{\cal M}(n)).
\ee

\item A map from the motivic complex to the weight $n$  exponential complex
 $  \Q_{\cal E}^\bullet(n)$, which allows to promote 
the class $c_n^{\cal M}$ (\ref{cnmot}) to the universal Chern class (\ref{cnan}). 
\end{enumerate}


One could probably combine Steps 2 and 3, to define an explicit map 
from the bigrassmannian complex to the weight $n$  exponential complex. Its most non-trivial part 
follows from the motivic construction of the Grassmannian $n$-logarithm in \cite{G11}. 
However the  problem  is open for  $n>4$.

\vskip 3mm

In contrast with this, the problem of explicit construction of Chern classes with values in the 
real Deligne cohomology is solved for all weights $n$: one combines the Step 1 with the  construction 
of a map from bigrassmannian complex to the real Deligne complex given 
in \cite{G95}. 

\vskip 3mm
An approach to construction of Grassmannian polylogarithms was developed by Hanamura and MacPherson 
\cite{HM90}.

\paragraph{Organisation of the paper.}

In Section 2 we recall the definition of the fundamental Hodge-Tate 
Hopf algebra ${\cal H}_\ast$, and then  construct the period morphism.

In Section 3 we calculate the period morphism from the 
polylogarithmic motivic complexes of weights $\leq 4$ to the {\rm Lie}-exponential complexes. 

Section 4 mostly borrowed from \cite{G93}. 
We recall the construction of characteristic classes using the bigrassmannian complex, 
articulating the role of the hypersimplices, and then recall 
the map from the bigrassmannian complex to the motivic
 complexes of weights $\leq 4$. 
Combining with the construction of the period morphisms from Section 2 we get an 
explicit construction of the universal Chern classes of weights $\leq 4$. 

Section 5 is a continuation of Section 2: we show that 
the $\C/\R(n)$-part of the canonical 
map $$
\omega^\bullet_n: \text{ the {\rm Lie}-exponential complex} \lra \text{the de Rham complex 
$\Omega^\bullet$}
$$ is 
homotopic to zero, and construct the homotopy, getting a regulator 
map to the real Deligne complex. 

\paragraph{Acknowledgments.} 
This work was supported by the  NSF grant   DMS-1301776.
A part of the paper was written at the IHES (Bures sur Yvette) during the Summer of 2015.  
I am  grateful to the NSF and IHES for the support.

\section{Period morphisms} 

\subsection{The $\Q$-Hodge-Tate Hopf algebra, and the period morphisms}
 
\paragraph{The algebra background.} Consider  a graded commutative Hopf algebra  over $\Q$ with a unit:
\be \la{theghaA}
A_{\ast} = \oplus_{k=0}^\infty A_k. 
\ee 
Let 
$\Delta: A_{\ast} \lra A_{\ast} \otimes A_{\ast}$  be the coproduct. 
The quotient 
$$
{\rm CoLie}(A_{\ast}):= \frac{A_{\ast}}{A_{>0}\cdot A_{>0}}
$$
is a  graded Lie coalgebra 
with the cobracket $\delta$ induced by the 
coproduct $\Delta$ on $A_{\ast}$. 
Let 
${\rm Lie}(A_{\ast})$ be  its graded dual. Then the 
universal enveloping algebra of the Lie algebra ${\rm Lie}(A_{\ast})$ 
is the graded dual to the Hopf algebra $A_{\ast}$, assuming that all 
graded components are finite dimensional. 

Consider the reduced coproduct
 $$
\Delta':= \Delta - ({\rm Id} \otimes 1 + 1 \otimes {\rm Id}): A_{\ast} \lra A_{>0} \otimes A_{>0} 
$$ 
The reduced cobar complex of the Hopf algebra $A_{\ast}$ is the following complex 
 starting in degree $1$:   
$$
A_{\ast}  \stackrel{\Delta'}{\lra}   
A_{\ast}  \otimes A_{\ast}    \stackrel{\Delta'}{\lra} ... 
\stackrel{\Delta'}{\lra}   \otimes^n A_{\ast} \stackrel{\Delta'}{\lra} \ldots . 
$$ 
$$
\Delta'(a_1\otimes \ldots \otimes a_n) := \sum_{k=1}^n(-1)^ka_1 \otimes \ldots \otimes \Delta'(a_k) \otimes \ldots \otimes a_n.
$$

The standard cochain complex of the Lie coalgebra ${\rm CoLie}(A_{\ast})$ is given by 
$$
{\rm CoLie}(A_{\ast}) \stackrel{\delta}{\lra}  \Lambda^2{\rm CoLie}(A_{\ast}) \stackrel{\delta}{\lra}  
\Lambda^3{\rm CoLie}(A_{\ast}) \stackrel{\delta}{\lra}  \ldots .
$$

These two complexes are canonically quasiisomorphic. The degree $n>0$ part of either of them 
calculates ${\rm RHom}_{A_{\ast}}(\Q(0), \Q(n))$ in the category of graded $A_{\ast}$-comodules, 
or, what is the same, graded ${\rm CoLie}(A_{\ast})$-comodules, 
 where $\Q(n)$ is the trivial one dimensional comodule 
in degree $-n$. 

 \paragraph{The fundamental Hopf algebra of the category of mixed Hodge-Tate structures.} 
For the convenience of the reader I recall some definitions  
from [BGSV]. See details in \cite[Section 4]{G96}. 

A a mixed $\Q$-Hodge structure $H$ is {\it Hodge-Tate} if its weight factors are isomorphic to 
$\oplus \Q(k)$.  A {\it $n$-framing} on  $H$ is a choice of a 
nonzero maps 
$v_0: \Q(0) \to gr^W_{0}H$ and $f^n: gr^W_{-2n}H \to \Q(n)$.  Consider the equivalence
relation $\sim$ on the set of all $n$-framed Hodge-Tate structures induced by the following: 
if there is a map $H_1 \to H_2$ compatible with frames, then 
 $H_1 \sim H_2$.  In particular,  any $n$-framed
Hodge-Tate structure is equivalent to a one $H$ with $W_{-2n-2}H = 0$, $W_{0}H=H$.
Let ${\cal H}_{n}$ be the set of equivalence classes. 
We define on ${\cal H}_{n}$ an abelian group structure as follows:  
$$
(f^n, H, v_0) + (\widetilde f^n, \widetilde H, \widetilde v_0):= 
(f^n+ \widetilde f^n, H\oplus \widetilde H, v_0 + \widetilde v_0); 
$$
$$
-(f^n, H, v_0) := ( f^n, H, -v_0). 
$$

The tensor product of mixed Hodge structures induces the commutative multiplication
$$
\mu: {\cal H}_{k} \otimes {\cal H}_{\ell}\to {\cal H}_{k+ \ell}.
$$

Let us define a coproduct 
 \begin{equation} \label{GH}
\Delta = \bigoplus_{k+\ell=n} \Delta_{k \ell}: {\cal H}_{n}\to \bigoplus_{k + \ell = n} 
{\cal H}_{k}\otimes
{\cal H}_{\ell}.
 \end{equation}
Let $(f^n, H, v_0)\in {\cal H}_{n}$.  Choose a basis $\{v^{(i)}_k\}$ in ${\rm Hom}(\Q(k), gr_{-2k}^W H)$ 
and the dual basis $\{f_{(i)}^k\}$ in 
${\rm Hom}(gr_{-2k}^W H, \Q(k))$.    Then
$$
\Delta_{k,n-k}(f^n, H, v_0) := \sum_i (f^n, H, v^{(i)}_k) \otimes (f_{(i)}^k, H, v_0). 
$$
The graded $\Q$-vector space
$$
{\cal H}_\ast:= \oplus_{n=0}^\infty {\cal H}_n, 
$$
 has a natural structure of a graded Hopf algebra over $\Q$ with the commutative multiplication
$\mu$ and the comultiplication $\Delta$.

\begin{theorem} \label{Theorem 4.1}
 The category of mixed $\Q$-Hodge-Tate structures is canonically equivalent to the category of finite-dimensional graded ${\cal H}_{\ast}$-comodules.
\end{theorem}

Let $\Delta_n' $ be the restriction of the restricted  coproduct $\Delta'$ to $ {\cal H}_n$. Then 
for $n>0$ we have 
$$
 {\rm Ker} (\Delta'_n)    =  
\quad 
 \frac{\C}{(2 \pi i)^{n}\Q}  =  {\rm Ext}_{MHS/\Q}^1(\Q(0), \Q(n)).
$$  

In \cite{G96} we constructed a canonical homomorphism, called the {\it big period map}
\begin{equation} \label{h1}
P_n: {\cal H}_n \quad \lra \quad \C^* \otimes_{\Q} \C(n-2).
\end{equation}
 The restriction of $P_n$ to the subgroup $ {\rm Ker} (\Delta'_n )$ 
provides an isomorphism  
$$
 \frac{\C}{(2 \pi i)^{n}\Q}  =  {\rm Ker} (\Delta'_n)  \quad \lra   \quad  \C^*\otimes (2 \pi i)^{n-1}.
$$

\paragraph{Period morphisms.}

The same construction as above for the category of variations of framed 
mixed Hodge-Tate structures over a manifold $X$ delivers a sheaf ${  {\cal H}}_{\ast}$ 
of graded Hopf algebras in the 
{\it classical} topology on $X $.   
Consider a  complex of  sheaves ${  {\cal H} }^{\bullet}(n) $ given by the weight $n$ part of 
the reduced cobar complex 
  of ${  {\cal H}}_{\ast}$, placed in degrees $[1,n]$:   
$$
{  {\cal H}}_n   \stackrel{\Delta'}{\lra}   
({  {\cal H}} \otimes {  {\cal H}})_n    \stackrel{\Delta'}{\lra} ... 
\stackrel{\Delta'}{\lra}   \otimes^n {  {\cal H}}_1. 
$$
  For $n>0$ one has a quasiisomorphism of complexes of sheaves in the classical topology on $X$: 
$$
  {\rm RHom}_{{\rm MHS}_X}(\Q(0)_X, \Q(n)_X) = {  {\cal H}}_n  \stackrel{\Delta'}{\lra} 
({  {\cal H}} \otimes {  {\cal H}} )_n  \stackrel{\Delta'}{\lra} ... 
\stackrel{\Delta'}{\lra}  \otimes^n {  {\cal H}}_1. 
$$
We can state now precisely Theorem \ref{periods1A}. 
 \begin{theorem} \label{periods1}
There exists a canonical morphism of complexes of sheaves 
$$
 P^{\bullet}_n: {  {\cal H}  }^{\bullet}(n)  
\lra     \Q_{\bf E}^{\bullet}(n),
$$ 
called the \underline{period morphism}, 
which satisfies the properties 1)-3) in Theorem \ref{periods1A}. 
\end{theorem}

A proof of Theorem \ref{periods1} is given in Section \ref{secmorper}.

\subsection{The period homomorphism of algebras ${\rm P}': {\cal H}_\ast \lra \C \otimes \C$}

This Section is an elaborate exposition of   Section 4 of \cite{G96}.

  \paragraph{1. The period operator and the period matrix.} Let $H$ be a mixed Hodge-Tate 
structure over $\Q$.  
Then there is an isomorphism 
 \begin{equation}\label{0.2} 
 H_{\C} = \oplus_{  p  } F^{p}H_{\C} \cap W_{2p}H_{\C}.  
 \end{equation} 
 Furthermore, the following canonical map is an isomorphism:  
 \begin{equation} \label{0.1}
 F^{p}H_{\C} \cap W_{2p}H_{\C} \stackrel{\sim}{\lra} gr^{W}_{2p}H_{\Q} \otimes_{\Q} \C.
 \end{equation}
 Using 
 isomorphisms 
 (\ref{0.1}) and (\ref{0.2}) we get  
  a  canonical morphism
 $$
 S_{HT}: \oplus_{  p  }  gr^{W}_{2p}H_{\Q} \lra H_{\C}.
 $$
  
 On the other hand a splitting of the weight filtration on $H_{\Q}$ also  provides us 
 a morphism 
 $$
 S_W: \oplus_{  p  }  gr^{W}_{2p}H_{\Q} \lra H_{\C}.
 $$
 Both maps became  isomorphisms when    extended   to  
 $\oplus_{  p  }  gr^{W}_{2p}H_{\C}$. Therefore a splitting 
 of the weight filtration on $H_{\Q}$ provides a   
 map, called the {\it period operator}: 
 $$
 S_{HT}^{-1} \circ S_W: \oplus_{  p  }  gr^{W}_{2p}H_{\C} \lra 
 \oplus_{  p  }  gr^{W}_{2p}H_{\C}.  
 $$

 Let $(f_{ n}, H, v_0) $ be a  Hodge-Tate structure over $\Q$,    
 framed by $\Q(0)$ and $\Q(n)$. 
Choose a splitting $s$ over $\Q$ of the weight filtration on $H_{\Q}$.
We define   the  period   of the splitted framed Hodge-Tate structure  
  $(f_{ n}, H, v_0; s)$     as the matrix coefficient of the {period operator}:
$$
 p(f_{ n}, H, v_0; s):= <v_0| S_{HT}^{-1} \circ S_W |f^n\rangle.
$$

  Choose  a basis  in each $\Q$-vector space $gr^{W}_{2p}H_{\Q}$, 
  providing a basis in their direct sum. 
  The {\it period matrix} is the matrix of the period operator  
    in this basis. 
   One can   define a mixed Hodge-Tate structure by exhibiting its period matrix. 
See an example below. 
  
  We define an equivalence relation on the set of all splitted  framed 
   $\Q$-Hodge-Tate structures  as the finest 
   equivalence relation for which  
   any morphism of mixed $\Q$-Hodge  structure $H \to H'$   respecting the 
   splittings and the frames is an equivalence.
   
   Let $\widetilde {\cal H}_n$ be the set of  equivalence classes of splitted $n$-framed 
   Hodge-Tate structures. Then $\widetilde {\cal H}_{\ast}:= 
   \oplus_n \widetilde {\cal H}_n$ is equipped in the usual way 
   with a structure of 
   a graded Hopf algebra. For instance $\widetilde {\cal H}_1 = \C \otimes \Q$.
    In particular  there is a coproduct map 
   $\Delta: \widetilde {\cal H}_{\ast} \to   \widetilde {\cal H}_{\ast}
    \otimes \widetilde {\cal H}_{\ast}$. 
   
   Let $H \to H'$ be a morphism 
   of   Hodge-Tate structures  respecting the frames and splittings. 
Then the periods of $H$ and $H'$ are the same, so  
 we get the period homomorphism 
$$
\widetilde p_n: \widetilde {\cal H}_{n} \to \C. 
 $$

\paragraph{2. The big period map.} Let $A$ and $B$ be operators in a $\Q$-vector space $V$. 
 Let $\{v_k\}$ be a $\Q$-basis in $V$, and $\{f^k\}$ be the dual basis.  
Define  
$$
\langle f^{n}| B \otimes_{\Q} A|v_0\rangle: = \sum_{v_k} \langle f^{n}| B |v_k\rangle \otimes_{\Q} 
\langle f^{k}| A |v_0\rangle  ~ \in  ~ \C \otimes_{\Q} \C, 
$$
where the sum is over all  basis vectors $v_k$.   
 It 
is well defined.

\bd Let $(f^n, H, v_0; s)$ be 
a splitted framed $\Q$-Hodge-Tate structure, and ${\cal M}$ the period operator on  
$\oplus_k gr^W_{-2k}H_{\Q}$. Then we set 
    \begin{equation}  \label{CCP}
   {\rm P}_n'(f^n, H, v_0; s):= \langle f^{ n}|   {\cal M} \otimes_{\Q}    {\cal M}^{-1}  |v_0\rangle \in \C\otimes \C.
\ee
\ed

 \begin{lemma} 
The element (\ref{CCP}) does not depend  on the choice of splitting. 
\end{lemma}

\begin{proof} The normalised period matrix corresponding to a different splitting is given by 
${\cal M}{\rm N}$, where ${\rm N}$ is a rational unipotent 
upper triangular matrix. One has
$$
\langle f^{ n}| {\cal M}{\rm N}  \otimes_{\Q}  ({\cal M}{\rm N})^{-1} |v_0\rangle \quad = \quad 
\langle f^{ n}| {\cal M} \otimes_{\Q}    {\cal M}^{-1}|v_0\rangle. 
$$
\end{proof}

Notice that $\C\otimes_\Z\C$ is an algebra: $(a\otimes b)\cdot (a'\otimes b') = aa'\otimes bb'$. 

Lemma \ref{multbpm} tells that the big period map ${\rm P}_n'$ is multiplicative: 
it takes the tensor product of the splitted framed Hodge -Tate structures into the product in $\C\otimes_\Z\C$. 

 \begin{lemma} \la{multbpm}
Let ${\rm M}$ and ${\rm M}'$ be  splitted 
framed Hodge-Tate structures of weights $m$ and $m'$. Then 
$$
{\rm P}_{m+m'}'(f^m \otimes {f'}^{m'}, {\rm M} \otimes {\rm M}', v_0\otimes v_0'; s\otimes s') = 
{\rm P}_{m}'(f^m, {\rm M}, v_0; s) \cdot {\rm P}_{m'}'({f'}^{m'}, {\rm M'} , v_0'; s'). 
$$
\end{lemma}

\begin{proof} Let ${\cal M}$ (respectively ${\cal M}'$) be the normalised period matrix for the splitted framed Hodge-Tate structure
 ${\rm M}$ (respectively ${\rm M}'$). Then the normalised period matrix describing ${\rm M}\otimes {\rm M}'$ 
is just the tensor product ${\cal M}\otimes {\cal M}'$ of the normalised period matrices ${\cal M}$ 
and ${\cal M}'$. 
Evidently, 
$$
\langle f^p\otimes {f'}^q~|~ {\cal M}\otimes {\cal M}'~|~ e_0\otimes e_0'\rangle = 
\langle f^p|~ {\cal M}~|~ e_0\rangle 
\langle {f'}^q~|~ {\cal M}'~|~ e_0'\rangle. 
$$
The claim follows immediately from this remark.
\end{proof}

\bd The big period map ${\rm P}_n$ is the  composition of the map ${\rm P}'_n$ with the map 
\be \la{IMPPM}
\C \otimes_{\Q} \C \lra \C^* \otimes_{\Q} \C(n-2), \qquad a \otimes b  \lms 
{\rm exp} (2 \pi i \cdot a) \otimes 2 \pi i \cdot b \otimes (2 \pi i)^{n-2}. 
\end{equation}
\ed

Let $U$ be a complex domain. There is a  map
$$
 \omega: {\cal O}_U \otimes_{\Q} {\cal O}_U \lra \Omega^1_U, 
 \quad f \otimes g  \lms (d f)  g.
$$
\begin{theorem} \label{periods}
a) The map $ {\rm P}'_n$  
  is a  homomorphism of abelian groups ${\cal H}_n \lra \C \otimes_{\Q}\C$.

Given  
${\rm H}_m \in {\cal H}_m$ and ${\rm H}_n \in {\cal H}_n$ one has  
$$
{\rm P}'_{n+m}({\rm H}_m \otimes {\rm H}_n) = {\rm P}'_{m}({\rm H}_m) \cdot {\rm P}'_{n}( {\rm H}_n). 
$$
So the collection of the maps $ \{{\rm P}'_n\}$  gives rise to an algebra homomorphism
$$
{\rm P}': {\cal H}_\ast \lra \C \otimes_{\Q}\C. 
$$

b) The restriction   of the map ${\rm P}_n$ to  
 ${\rm Ker}( \Delta'_n)$ coincides with  the natural isomorphism
 \begin{equation}  \label{MAP}
\frac{\C}{(2 \pi i)^n\Q}\quad = \quad  \C_{\Q}^*\otimes (2 \pi i)^{n-1}.
 \end{equation}  

c) Let $H_U$ be a variation of framed mixed Hodge-Tate  
structures over a domain $U$. Then there is a section
$ 
  {\rm P}'(H_U ) \in   {\cal O}_U \otimes_{\Q} {\cal O}_U $,   
and the following composition is zero:
 $$ 
H_U  \stackrel{ {\rm P}'}{\lra} {\cal O}_U \otimes {\cal O}_U 
\stackrel{\omega}{\lra} \Omega^1_U, ~~~~\omega\circ {\rm P}'=0.
 $$ 
\end{theorem}

  \begin{proof} To prove the part a) of Theorem \ref{periods} we rewrite 
the map ${\rm P}_n'$ in terms of  the Hopf algebra  ${\cal H}_{\ast}$. 
This is done in the Appendix. The second statement follows then from Lemma \ref{multbpm}. 

  b) Clear from the definitions. 

c) We will prove it  in paragraph 5 below.  
\end{proof}

 \paragraph{3. Explicit formulas.} Given a variation of splitted framed $\Q$-Hodge-Tate structure $H$, 
choose a basis $\{v_i\}$ over $\Q$ 
in each fiber of a variation. 
We assume that basis vectors $v_i$ are of pure weight ${\rm wt}(v_i)$. Denote by $\{f^i\}$ the dual basis. 
We use  notation $\langle  f| {\cal M} | v \rangle $ for $p(f, H, v;s)$. 

We usually assume that the framing is given by basis vectors $(v_0, f^n)$.

Set ${\cal M}= 1 + {\cal M}_0$. Since ${\cal M}_0$ is nilpotent, expanding 
$(1 + {\cal M}_0)^{-1} =  \sum_{k\geq 0}(-1)^k{\cal M}_0^k$ we get 
\be   \la{CCPV}
\langle f^{ n}|   {\cal M}  \otimes_{\Q}   {\cal M}^{-1}  |v_0\rangle  ~=~ \sum_{k\geq 0} (-1)^k\langle f^{n}| 
      {\cal M} \otimes_{\Q}   {\cal M}_0^k  | v_0\rangle.
   \end{equation}

By (\ref{CCPV}), 
  the big period of a 
  splitted framed Hodge-Tate structure $(H;  f^n, v_0; s)$ is
    $$
{\rm P}'(f^n, H, v_0; s) \in \C \otimes \C.
$$
\begin{equation}  \label{&&**}
      {\rm P}'(f^n, H, v_0; s) = \langle f^{ n}  | {\cal M} |  v_0\rangle \otimes 1 
            +   
    \end{equation}
 $$ 
  \sum_{k \geq 2}\sum_{0 < i_1 < ... < i_{k-1} < n}(-1)^{k-1}       
  \langle f^{ n}| {\cal M} |  v_{ i_{k-1}}\rangle     
\otimes    \langle f^{ i_{k-1}} | {\cal M} | v_{ i_{k-2}}  \rangle  \cdot 
  ... \cdot  \langle f^{ i_1} | {\cal M} |  v_0 \rangle 
  +
  $$ 
  \begin{equation}  \label{&&*}  
   \sum_{k \geq 1}\sum_{0 < i_1 < ... < i_{k-1} < n}(-1)^{k }  \cdot      
1  \otimes  
\langle f^{ i_{n}} | {\cal M} |  v_{ i_{k-1}} \rangle   \langle  f^{ i_{k-1}} | {\cal M} | v_{i_{k-2}}\rangle    \cdot 
  ... \cdot   \langle f^{ i_1} | {\cal M} |  v_0\rangle.  
  \end{equation} 
  The sum is over all nonempty chains of basis vectors $v_{ i} \in   gr^W_{-2i}H_{\Q}$,
   $ 0 < i < n$.

Since the term (\ref{&&*}) disappears after the 
projection  $\C \otimes_{\Q} \C \lra  \C^* \otimes_{\Q} \C(n-2)$, we have 
$$
(2\pi i)^{-n+2}{\rm P}(f^n, H, v_0; s) \in \C^*\otimes \C. 
$$
\begin{equation}  \label{&&**}
      (2\pi i)^{-n+2}{\rm P}(f^n, H, v_0; s) = {\rm exp}(2\pi i \cdot \langle f^{ n}| {\cal M} |  v_0 \rangle)  
\otimes 2\pi i +   
    \end{equation}
 $$ 
  \sum_{k \geq 2}(-1)^{k-1}\sum_{0 < i_1 < ... < i_{k-1} < n}        
  {\rm exp}(2\pi i \cdot \langle f^{ i_{n}} | {\cal M} |  
v_{ i_{k-1}}\rangle )\otimes  2\pi i \cdot \langle  f^{ i_{k-1}} | {\cal M} | v_{i_{k-2}}\rangle  \cdot 
  ... \cdot  \langle f^{ i_1} | {\cal M} |  v_0\rangle.
  $$

\paragraph{4. Examples.} 1. Let us define a Hodge-Tate structure  
$M$ using a normalised period matrix: 
$$  
{\cal M}:= \quad \left (\begin{array}{cccc}
1&&& \\
 x_1 & 1 &&\\ 
x_2 & y_1 &
     1&\\
x_3  & y_2 
     &   z_1 
       &1 \\\end{array}\right ),~~~~~~~~x_i, y_j, z_1\in \C. 
$$
 Let  
${\cal I}$  be the matrix of the operator acting by 
$(2 \pi i)^{-k}$ on $gr^W_{2k}H_{\Q}$. 
Then the period matrix is  
$$  
\widetilde {\cal M}:=   {\cal M}{\cal I}=
\quad \left( \begin{array}{cccc}
1& & & \\
2 \pi i \cdot x_1  & 2 \pi i &&  \\    
(2 \pi i)^2 \cdot x_2   & (2 \pi i)^2 \cdot  y_1  &
     (2 \pi i)^2 & \\
(2 \pi i)^3 \cdot x_3 & (2 \pi i)^3  \cdot  y_2 
  &    (2 \pi i)^3 \cdot  z_1
   &(2 \pi i)^3 \\ \end{array}\right ). 
$$ 
\paragraph{Remark.} The matrix $\widetilde {\cal M}$ is the period matrix which appear naturally in algebraic geometry. 
The normalized period matrix ${\cal M}$ is more convenient when we 
work with the big period. 

Precisely, if  $M$ is the Hodge 
 realization of a mixed Tate motive, 
   the entries of the canonical period matrix are  periods of rational algebraic differential 
forms over relative cycles. The $\widetilde {\cal M}$ is the matrix of 
the comparison isomorphism 
$M_{DR}\otimes \C \lra M_{Betti} \otimes \C$ in the natural $\Q$-bases in $M_{DR}$ and $M_{Betti}$. 

\vskip 3mm

Let $C_i$ be the $i$-th column of the matrix $\widetilde {\cal M}$. Let $e_{-j}$ be the 
column whose only non zero entry is $1$ on $j$-th place. 
We define the weight filtration $W_\bullet$ and the Hodge filtration $F^\bullet$  by
$$
W_{-6}M = \langle C_3\rangle_{\Q}, ~~ W_{-4}M = \langle C_2, C_3\rangle_{\Q}, ~~ W_{-2}M =
 \langle C_1, C_2, C_3\rangle_{\Q}, ~~
W_{0}M = \langle C_0, C_1, C_2, C_3\rangle_{\Q}. 
$$
$$
F^{0}M = \langle e_0\rangle, ~~ F^{-1}M = \langle e_0, e_{-1}\rangle, ~~ F^{-2}M 
= \langle e_0, e_{-1}, e_{-2}\rangle, ~~
F^{-3}M = \langle e_0, e_{-1}, e_{-2}, e_{-3}\rangle. $$

The   splitted Hodge-Tate structure $M$ has a framing given 
by $e_0$ and $(2 \pi i)^{-3} f^{3}$. Its period is $ x_3$. 
 The big period is
 \begin{equation}  \label{PP'} 
{\rm P}'_3(M) =       x_3   \otimes   1
 +     y_2  \otimes   (-x_1) 
 \end{equation}   
$$
 +     z_1  \otimes  ( - x_2 +  x_1  y_1)
        + 
1\otimes  ( - x_3 + x_1  y_2  + 
  x_2  z_1 -  x_1  y_1  z_1)  \in \C\otimes_{\Q} \C.
$$
   
The period ${\rm P}_3$ is given by 
$$
{\rm P}'_3(M) =        {\rm exp}(2 \pi i \cdot  x_3)   \otimes   2 \pi i 
 +   {\rm exp}(2 \pi i \cdot y_2)  \otimes  {\rm exp}(2 \pi i \cdot (-x_1))  
$$
$$
+   {\rm exp}(2 \pi i \cdot   z_1 ) \otimes   {\rm exp}(2 \pi i \cdot ( - x_2 +  x_1  y_1)).
$$

2. Here is a classical example of the period matrix for the  variation of Hodge-Tate structures related to the dilogarithm (due to Deligne):
$$
\widetilde {\cal L}_2 := \qquad \left( \begin{array}{ccc}
1&0 &  \\   
-{\rm Li}_1(z) & 2 \pi i &  \\     
 -{\rm Li}_2(z) & 2 \pi i\cdot \log z  &
     (2 \pi i)^2 \\ \end{array}\right ). 
$$

Then 
$$
{\rm P}'_2({\cal L}_2 ) = 
-\frac{{\rm Li}_2(z)}{(2 \pi i)^2}  \otimes 1 + \frac{\log z}{2\pi i}\otimes   \frac{\log(1-z)}{2\pi i} + 
1 \otimes \frac{{\rm Li}_2(z) - \log z \log(1-z)}{(2 \pi i)^2}\in \C \otimes \C. 
$$
$$
{\rm P}_2({\cal L}_2 ) = 
{\rm exp}(\frac{-{\rm Li}_2(z)}{2 \pi i})  \otimes 2 \pi i + z\otimes   \log(1-z)\in \C^* \otimes \C. 
$$
The invariant ${\rm P}_2({\cal L}_2 )$ was first written by S. Bloch [Bl3].  Generalizing it,  
 $Sym_{\Q}^{n-1} \C \otimes \C^*$-valued invariants of the Hodge-Tate structures related to classical $n$-logarithms where constructed in [BD]  and [Bl4]. However the approach 
of these papers is different from ours; it uses the  specific structure of the Hodge-Tate structures related to classical polylogarithms,   which can not be generalized to other mixed Tate motives.

\paragraph{5. Differential equations on periods and the Griffiths transversality condition.} A variation of mixed Hodge structures  satisfies the 
Griffiths transversality condition. We say that a partial period 
$\langle  f^{l}| {\cal M} |  v_{k}\rangle$, where $v_k \in gr^W_{-2k}H$ and 
$f^l \in (gr^W_{-2l}H)^*$, has {\it amplitude} $l-k$.  
 
\bt \label{GRIFF} Let ${\cal M }$ 
be a normalised period matrix of a variation of
splitted framed Hodge-Tate structures $(H_U; v_0, f^n; s)$. Then 
 
i) The connection $\nabla$ on the variation is given by 
\be \la{12.4.15.1}
\nabla(v_k) = -\sum_{\{v_{k+1}\}}\langle f^{ k+1}|  d  {\cal M} |v_k\rangle\cdot v_{k+1}.
\ee
The sum is over basis vectors $\{v_{k+1}\}$  of weight $-2(k+1)$. 

ii) The Griffiths transversality condition is equivalent to the following 
 differential equations on the entries of the normalised period matrix 
${\cal M }$:     
 \be \la{12.4.15.2}
 \langle f^{ k+s}|  {\cal M}^{-1} d  {\cal M}\Bigr|v_k\rangle ~ = ~ 0 ~~~~\forall s>1.
 \ee

iii) The period $\langle v_{0} | {\cal M} |  f^{n}\rangle$ 
satisfies a differential equation
\be \la{DE1}
d \langle f^{n}| {\cal M} |  v_{0} \rangle  = \sum_{\{v_{n-1}\}} \langle f^{n}| {\cal M} |  v_{n-1} \rangle
d\langle f^{n-1}| {\cal M} |  v_{0} \rangle. 
\ee

iv) The Griffiths transversality is equivalent to  differential  
  equations (\ref{DE1}) for all partial periods of amplitudes $\geq 2$. 
\et 

\begin{proof} The 
vectors $\sum_j\langle f^j | {\cal M} | v_i\rangle v_j$ are flat sections of the 
connection $\nabla$ on then variation:
$$
0 = \nabla \Bigl(\sum_j\langle f^j| {\cal M} | v_i \rangle v_j \Bigr) = 
\sum_j\langle f^j| {\cal M} | v_i \rangle \cdot \nabla (v_j) 
+ \sum_jd \langle f^j | {\cal M} | v_i\rangle \cdot v_j.
$$
Therefore
$$
\nabla(v_i) = -\sum_j\langle f^{ j}|  {\cal M}^{-1} d  {\cal M} |v_i\rangle\cdot v_j.
$$

i) To check (\ref{12.4.15.1}) notice that 
the only way ${\cal M}^{-1} d  {\cal M} $ can have a non-zero matrix coefficient 
of amplitude $1$ is that it is the matrix coefficient of amplitude $1$ of $d  {\cal M}$.

ii) The Griffiths transversality just means that all matrix coefficients of 
${\cal M}^{-1} d  {\cal M} $ of amplitude bigger then $1$ are zero, which is just what 
(\ref{12.4.15.2}) says. 

iii) - iv). Let us write, using (\ref{12.4.15.2}) and assuming $n>1$, 
$$
0 = \langle f^{n}| {\cal M}^{-1} d {\cal M} |  v_{0} \rangle = 
$$
$$
\langle f^{n}| d {\cal M} |  v_{0} \rangle  - \sum_{\{v_{n-1}\}} d\langle f^{n} | {\cal M} | v_{n-1} \rangle 
\langle  f^{n-1} | {\cal M} | v_{0}\rangle 
 + \sum_{k\leq n-2}\langle f^{n} |  {\cal M}^{-1} d  {\cal M} |v_k\rangle
\langle f^{k} | {\cal M} |  v_{0}\rangle. 
$$
The last summand is zero by (\ref{12.4.15.2}). So we get differential equation (\ref{DE1}), and the Claim iv). 
\end{proof}

 \paragraph{Remark.} Define a homomorphism $\Omega_n: \widetilde {  {\cal H}}_{n} \lra \Omega^1_U$ as the composition
$$
 \widetilde {  {\cal H}}_{n} \stackrel{\Delta_{n-1,1}}{\lra} 
\widetilde {  {\cal H}}_{n-1}\otimes 
 \widetilde {  {\cal H}}_1 \stackrel{p   d p}{\lra}\Omega^1_U.
$$ 
 The differential equation (\ref{DE1}) for the period $\langle f^{n} | {\cal M} | v_{0} \rangle$ 
can be rewritten as 
  \begin{equation} \label{DE11}
 d \langle f^{n}| {\cal M} |  v_{0} \rangle  =  \Omega_n(H_U;  f^n, v_0; s).
 \end{equation} 

\paragraph{6. The  big period map via the Hopf algebra  ${\cal H}_{\ast}$.} 
We use notation as in Section 2.1.
Projecting  to $ \otimes^pA_{>0}$ the coproduct and the  iterated coproduct, we get their 
  reduced versions: 
$$
\Delta':  A_{\ast} \lra A_{>0} \otimes A_{>0}, \quad {\Delta'}^{(p)}:  A \lra \otimes^p A_{>0}.
$$

For $n \geq 2$, let us consider  an algebra map 
$$
m_n: A  \lra A_{>0} \otimes A_{>0}
$$
 given by the $n$-iterated reduced coproduct followed by the 
product of the first $n-1$ factors: 
$$
A_{\ast}  \stackrel{
{\Delta'}^{(n)}}{\lra }   \otimes^{n-1} A_{>0}\otimes A_{>0}
  \stackrel{\mu^{(n-1)} \otimes id}{\lra } A_{>0} \otimes A_{>0}.
$$
Let $m_1: A_{\ast} \lra A_{\ast} \otimes A_{\ast}$, $a \lms 1 \otimes a $. Now set:
$$
m: A_{\ast} \lra A_{\ast} \otimes A_{\ast}, \quad m:= \sum_{n \geq 1}(-1)^{n-1 } m_n. 
$$

Let us define a map $\widetilde m: A_{\ast} \lra A_{\ast} \otimes 1 \hookrightarrow A_{\ast}\otimes A_{\ast}$ by setting
$$
\widetilde m:= \sum_{n \geq 1}^{\infty} (-1)^n \mu^{(n)} \circ \Delta^{(n)}: A_{\ast} \lra A_{\ast} = A_{\ast}\otimes 1 \hookrightarrow A_{\ast}\otimes A_{\ast}, \qquad  \Delta^{(1)} = \mu^{(1)} = {\rm id}. 
$$

We apply   this to the Hopf algebra ${\cal H}_{\ast}$. 
We 
get a map
$$
m + \widetilde m: {\cal H}_{\ast} \lra 
 {\cal H}_{\ast} \otimes {\cal H}_{\ast}. 
$$
The explicit formula for the map ${\rm P}_n'$ just means the following. 
\begin{lemma} \label{qwq}
The big period map ${\rm P}'_n$ is equal to a composition
$$
{\cal H}_{\ast}  \quad \stackrel{ m + \widetilde m}{\lra} ~
{\cal H}_{\ast} \otimes {\cal H}_{\ast} 
~ \stackrel{2 \pi i \cdot p \otimes 2 \pi i \cdot p}{\lra} ~ 
\C \otimes_{\Q} \C.   
$$
\end{lemma}

Then ${\rm P}_n$ is the  composition 
\begin{equation} \label{P1}
{\cal H}_{\ast}  \quad \stackrel{ m }{\lra} \quad 
{\cal H}_{\ast} \otimes  {\cal H}_{\ast} 
\quad \stackrel{ p \otimes  p}{\lra} \quad 
\C \otimes_{\Q} \C \lra  \C \otimes_{\Q} \C^*(n-2).
\end{equation} 
The term (\ref{&&*}) 
corresponding to $ \widetilde m$ disappears after the projection  $\C \otimes_{\Q} \C \lra  \C \otimes_{\Q} \C^*(n-2)$.

\subsection{Construction of period morphisms and proof of Theorem \ref{periods1}} \la{secmorper}

\paragraph{Step 1. The map ${\rm P}_n^{\bullet}$.} 
 Let us define a homomorphism of abelian groups
\be \la{MPNK}
{P}_n^{k}: \otimes^k{  {\cal H}}_{\ast} \lra 
\underbrace{{\cal O}^* \otimes 
\ldots \otimes {\cal O}^*}_{\text{k times}}  \otimes  {\cal O}\otimes \underbrace{2\pi i \otimes \ldots \otimes 2 \pi i}_{\text{n-k-1 times}}.
\ee

First, there is an associative algebra structure on $\otimes^{\bullet -1} {\cal O}$ given by 
$$
(\otimes^{k+1} {\cal O}) * (\otimes^{l+1} {\cal O}) \lra \otimes^{k+l+1}{\cal O},
$$ 
$$
  (a_0 \otimes ... \otimes a_{k})  * 
(b_{0} \otimes ... \otimes b_{ l }) \lms a_0 \otimes ... \otimes a_{k} \cdot b_{0} \otimes ... \otimes b_l.   
$$

Let $H_i \in {  {\cal H}}_{\ast}$. We set 
$$
{{\rm P}'}_n^{k} (H_1 \otimes ...  \otimes H_k) :=  {P}'_n (H_1) * ... *  {P}'_n (H_k)  \in \otimes^{k+1} {\cal O}. 
$$
Next, consider a map 
\begin{equation} \label{EEX}
{\rm Exp}^{(k)}\otimes 2\pi i\cdot {\rm Id}:  ~~\otimes^{k+1} {\cal O} \lra 
\underbrace{{\cal O}^* \otimes 
\ldots \otimes {\cal O}^*}_{\text{k times}}\otimes  {\cal O}.
\end{equation}
$$
{\rm Exp}^{(k)}:=   \underbrace{{\rm exp}
( 2 \pi i\cdot *) \otimes  ... \otimes  {\rm exp}( 2 \pi i\cdot *)}_{\text{k times}}. 
$$
We define the map (\ref{MPNK}) by setting 
\begin{equation} \label{P_n}
{P}_n^{k}:= \underbrace{2\pi i \otimes \ldots \otimes 2 \pi i}_{\text{n-k-1 times}} \otimes 
\left ({\rm Exp}^{(k)}  \otimes (2\pi i\cdot {\rm Id} ) \right) \circ {{\rm P}'}_n^{k}. 
 \end{equation}

 \paragraph{Step 2. The maps $\{(-1)^k{P}_n^{k}\}$ provide a morphism of complexes.} 
Equivalently, we have to show that the following diagram is a bicomplex, 
where $\Delta'$ is the restricted coproduct: 
$$
\begin{array}{ccccc}
{  {\cal H}}_n  & \stackrel{\Delta'}{\lra} & 
({  {\cal H}} \otimes {  {\cal H}} )_n  & \stackrel{\Delta'}{\lra} ... 
\stackrel{\Delta'}{\lra} & \otimes^n {  {\cal H}}_1 \\
&&&&\\
\downarrow {\rm P}_n^{1} &&\downarrow {\rm P}_n^{2}   &&\downarrow {\rm P}_n^{n} \\
&&&&\\
 {\cal O}^*  \otimes {\cal O}(n-2)& \stackrel{d}{\lra} &  {\cal O}^*\otimes   {\cal O}^* \otimes {\cal O}(n-3)& \stackrel{d}{\lra} ... \stackrel{d}{\lra} & \otimes^n  {\cal O}^*\\
\end{array}
$$ 

Let us show  that the left square is anticommutative. 

For the restricted coproduct $\Delta'$, we have the following 
 element  in ${\cal O}  \otimes {\cal O} \otimes {\cal O}$: 
\begin{equation} \label{FORM111}
{{\rm P}'}_n^{2} (\Delta' ({\cal M})) = \sum 
\langle f^n|{\cal M}|v_l\rangle \otimes \langle f^l|{\cal M}^{-1}|v_m\rangle \cdot   \langle f^m|{\cal M}|v_k\rangle\otimes \langle f^k |{\cal M}^{-1}|v_0\rangle .
\end{equation}
The sum is over all basis vectors $v_i$ satisfying the following two conditions
\begin{equation} \label{FORM}
wt(v_n) \leq wt(v_l) \leq wt(v_m) \leq wt(v_k)  \leq 0. 
\end{equation}
\begin{equation} \label{FORM*}
wt(v_n)< wt(v_m) < 0. 
\end{equation}
Condition (\ref{FORM*}) results 
from taking the     restricted coproduct $\Delta'$ rather then the coproduct $\Delta$.  
 
Let us compute  the  image of element (\ref{FORM111}) 
under the map 
\be \la{POPO}
 {\rm Exp}^{(2)}  \otimes 2\pi i: {\cal O}  \otimes {\cal O} \otimes {\cal O}\lra {\cal O}^* \otimes {\cal O}^* 
  \otimes {\cal O},
\ee
$$
a \otimes b \otimes c \lms {\rm exp}(2\pi i a) \otimes {\rm exp}(2\pi i b)\otimes 2\pi i c. 
$$

Observe that since ${\cal M}{\cal M}^{-1} = I$, we have 
 \begin{equation} \label{FORM2}
\sum_{wt(v_l) \leq wt(v_m) \leq wt(v_k) } \langle f_l|{\cal M}^{-1}|v_m\rangle \cdot \langle f^m|{\cal M}|v_k\rangle = \delta_{0,l}.
 \end{equation}

Now there are three cases of the summation.

i) If $wt(v_k)<0$, and $wt(v_n)< wt(v_l)$, then thanks to  conditions 
(\ref{FORM}) - (\ref{FORM*}) and  formula (\ref{FORM2}), the corresponding sum in (\ref{FORM111}) 
collapses to 
\begin{equation} \label{FORM1KK}
\sum\langle f^n|{\cal M}|v_k\rangle\otimes 1 \otimes \langle f^k|{\cal M}^{-1}|v_0\rangle .
\end{equation}

ii) If $ wt(v_k)=0$, then, thanks to (\ref{FORM*})  the corresponding sum in (\ref{FORM111}) is
$$
\sum_{v_l} \sum_{v_m \not = v_0} \langle f^n|{\cal M}|v_l\rangle \otimes 
\langle f^l|{\cal M}^{-1}|v_m \rangle\cdot \langle f^m|{\cal M}|v_0\rangle  \otimes 1  \stackrel{(\ref{FORM2})}{= }
$$
 \begin{equation} \label{FORM3}
 - \sum_{v_l}  \langle f^n |{\cal M}|v_l\rangle\otimes  \langle f^l |{\cal M}^{-1}|v_0\rangle  \otimes 1.
\end{equation} 
Indeed, since $\langle f_0 |{\cal M} |v_0\rangle = 1$, formula (\ref{FORM2}) implies
$$
\sum_{v_m \not = v_0} \langle f^l|{\cal M}^{-1}|v_m \rangle\cdot \langle f^m|{\cal M}|v_0\rangle  = -   
\langle f^l|{\cal M}^{-1}|v_0 \rangle.
$$

iii) If $wt(v_l) = wt(v_n)$, we similarly get 
$$
 \sum_{v_m \not = v_n}1\otimes 
\langle f^n|{\cal M}^{-1}|v_m\rangle \cdot \langle f^m |{\cal M}|v_k\rangle 
\otimes \langle f^k|{\cal M}^{-1}|v_0\rangle .
$$
\begin{equation} \label{FORM4a}
- \sum_{v_k} 1 \otimes   \langle f^n |{\cal M}|v_k\rangle \otimes \langle f^k|{\cal M}^{-1}|v_0\rangle  .
\end{equation} 

Since ${\rm exp}(2\pi i)=1$ is the neutral element in ${\cal O}^*$, 
 (\ref{FORM1KK}) and (\ref{FORM4a}) contributes zero after applying  map (\ref{POPO}).
So applying the map (\ref{POPO}) to the expression (\ref{FORM3}) we get
$$
- \sum_{v_l}   {\rm exp}( 2\pi i \cdot \langle f^n|{\cal M}|v_l\rangle) \otimes 
{\rm exp}(2\pi i \cdot \langle f^l|{\cal M}^{-1}|v_0 \rangle)
\otimes  2\pi i . 
$$

On the other hand, by the definition 
of ${\rm P}_n'$ in (\ref{CCP}), 
$$
{{\rm P}_n'}({\cal M}) = \sum_{v_l}   \langle f^n|{\cal M}|v_l\rangle\otimes \langle f^l|{\cal M}^{-1}|v_0 \rangle.
$$
Therefore thanks to the definition 
of ${\rm P}_n = {\rm P}_n^{1}$ in (\ref{IMPPM}), and the definition of $d$ in 
(\ref{ana1*}),
$$
d\circ {\rm P}_n^{1}({\cal M}) =  
\sum_{v_l} {\rm exp}( 2\pi i \cdot \langle f^n|{\cal M}|v_l\rangle) \otimes   
{\rm exp}(2\pi i \cdot \langle f^l|{\cal M}^{-1}|v_0 \rangle)\otimes 2\pi i .
$$
We conclude that
$$
{\rm P}_n^{2}(\Delta' ({\cal M})) + d\circ {\rm P}_n^{(1)}({\cal M})=0.
$$

In general we have to check that the 
composition of the restricted coproduct   
$$
{\cal M}_1 \otimes ...  \otimes {\cal M}_k \lra \sum(-1)^{i-1} {\cal M}_1 \otimes ... \otimes  
\Delta' ({\cal M}_i) \otimes ...  \otimes {\cal M}_k 
$$
with the map ${\rm P}_n^{k+1} $ is equal to  $-d \circ {\rm P}_n^{k}({\cal M}_1 \otimes ...  \otimes {\cal M}_k)$.  Notice that  
$$
{\rm P}_n^{k+1} ({\cal M}_1 \otimes ... \otimes \Delta' ({\cal M}_i )\otimes ...  \otimes {\cal M}_k 
) = 0 \qquad if \quad  i> 1.
$$
Indeed, ${\cal M}_1 \otimes ... \otimes \Delta' ({\cal M}_i )\otimes ...  \otimes {\cal M}_k$ has three terms, just like 
(\ref{FORM1KK}), (\ref{FORM3}), and (\ref{FORM4a}). Each of them has the $j$-th factor 
$1$, where $j=i+2, i+1, i$. So each of them vanishes when we apply the 
${\rm Exp}^{(k)} \otimes (2\pi i \cdot {\rm Id}) $ map (\ref{EEX}).      
In the case $i=k$ only the very right factor survives, contributing 
 $-d \circ {\rm P}_n^{k}({\cal M}_1 \otimes ...  \otimes {\cal M}_k)$.

 \paragraph{Step 3. The composition $\Omega_n^{k}\circ {\rm P}_n^{k}=0$ for $k < n$.} 
It is enough to check an equivalent claim for 
$d{{\rm P}'_n}^{k}$. For $k=1$, 
Theorem \ref{GRIFF}ii) implies, since $n>1$,  
$$
d{{\rm P}'_n}^{1}({\cal M}) = \sum_k\langle f^n|d {\cal M}|v_k\rangle\otimes \langle f^k|{\cal M}^{-1}|v_0\rangle  =0. 
$$

For $k=2$ we have 
$$
d{{\rm P}'_n}^{2}({\cal M} \otimes {\cal N}) = 
$$ 
$$
\sum\langle f^n|d{\cal N}|v_l\rangle\otimes \langle f^l|{\cal N}^{-1}|v_m\rangle 
\cdot \langle f^m|d {\cal M}|v_k\rangle  \otimes \langle f^k|{\cal M}^{-1}|v_0\rangle  + 
$$
$$
\sum  \langle f^n|d {\cal N}|v_l\rangle \otimes
\langle f^l|d{\cal N}^{-1}|v_m\rangle \cdot \langle f^m|{\cal M}|v_k\rangle  
\otimes \langle f^k|{\cal M}^{-1}|v_0\rangle .
$$
Since $n>k=2$, either $m>1$ or $n-m>1$. Then in the first line is zero since the factor 
of amplitude $>1$ is zero by   Theorem \ref{GRIFF}ii). 
The second line is always zero. 

For general $k$ we proceed just as in the case $k=2$. The expression $d{{\rm P}'_n}^{k}$ consists of 
sums of $k$ factors. If one of them has two differentials, it is  zero.   
Otherwise each has just one differential, and one is of amplitude  $>1$, and so 
vanishes by Theorem \ref{GRIFF}ii). 
Theorem \ref{periods1} is proved.

\subsection{A variant: {\rm Lie}-exponential complexes and {\rm Lie}-period morphisms} \la{Sec2}

Let $X$ be a  manifold, either a real  
or a complex analytic one. 
\begin{definition} \label{ana5}
The   weight $n$   {\rm Lie}-exponential complex  $  \Q_{\cal E}^\bullet(n) $  is 
a complex of sheaves 
on $X$, concentrated in degrees $[0, n]$: 
\begin{equation} \label{ana1}
  {\cal O}(n-1) \lra \Lambda^2 {\cal O}(n-2) \lra ... 
\lra \Lambda^{n}{\cal O}
\stackrel{\wedge^n {\rm exp}}{\lra} \Lambda^{n}{\cal O}^*.
\end{equation} 
The differentials are given by 
$$
(2 \pi i)^{n-k} \otimes a_1 \wedge ... \wedge a_k \lms (2 \pi i)^{n-k-1}  \otimes 2 
\pi i \wedge a_1 \wedge ... \wedge a_k, ~~~~k<n,
$$
$$
a_1 \wedge ... \wedge a_n \lms {\rm exp}(a_1) \wedge ... \wedge {\rm exp}(a_n).
$$
\end{definition}

For example, 
the complex $  \Q_{\cal E}^\bullet(2)$ is
$$
 {\cal O}(1) \stackrel{\delta}{\lra} \Lambda^{2}{\cal O} 
\stackrel{\wedge^2 {\rm exp}}{\lra} \Lambda^{2}{\cal O}^*. 
$$ 

Take the $n$-th symmetric power of the 
complex $\Q(1) \hra {\cal O}$  in degrees $[0,1]$. 
It is augmented by the exponential map to   $\Lambda^{n}{\cal O}^*[-n]$. 
There is an isomorphism of complexes
\be \la{CEC1}
  \Q(n) \lra   \Q_{\cal E}^\bullet(n)  ~~= ~~
{\rm Cone}\Bigl({\rm Sym}^n\Bigl(\Q(1) \hra {\cal O}\Bigr) \stackrel{\wedge^n {\rm exp}}{\lra} \Lambda^{n}{\cal O}^*[-n]\Bigr).
\ee
Therefore  the complex (\ref{ana1}) is a resolution of the constant sheaf $  \Q(n)$.  

\paragraph{Mapping {\rm Lie}-exponential complexes to differential forms.}

Recall the holomorphic de Rham complex   $\Omega^\bullet$ on a complex manifold $X$. 
There is a natural map from the weight $n$ {\rm Lie}-exponential complex to 
 the holomorphic de Rham complex:
$$
\omega_n^\bullet: {  \Q}_{\cal E}^\bullet(n) ~~\lra ~~ \Omega^\bullet. 
$$
Precisely, we have the following Lemma, proved by a  simple check, which is left to a reader. 
\bl \la{FLe} There is 
a canonical morphism of  complexes of sheaves on $X$: 
$$
\begin{array}{ccccccc}
  {\cal O}(n-1) &    \lra &\Lambda^2{\cal O}(n-2) &\lra \ldots 
\lra & \Lambda^{n}   {\cal O} & \lra &\Lambda^{n}   {\cal O}^*\\
&&&&&& \\
\downarrow \omega^{(0)}_n  
& ... &\downarrow \omega^{(1)}_n && \downarrow \omega^{(n-1 )}_n&&\downarrow \omega^{(n )}_n\\
&&&&&& \\
 \Omega^0  &    
 \stackrel{d}{\lra}& \Omega^1  &  \stackrel{d}{\lra} \ldots  \stackrel{d}{\lra} &\Omega^{n-1}  & \lra &  \Omega^{n}_{  cl}
\end{array}
$$ 
Here 
$$
\omega^{(m)}_n\Bigl(  (2 \pi i)^{n-m-1} \otimes (f_0 \wedge f_1 \wedge ... \wedge f_m)\Bigr) := 
$$
$$
  (2 \pi i)^{n-m-1} m! ~
\sum_{j=0}^{m}(-1)^{j} f_j ~df_0  \wedge \ldots \wedge \widehat {d f_j} \wedge \ldots \wedge d f_m,   ~~~~~  0\leq m<n,
$$ 
$$
 \omega^{(n )}_n(f_1 \wedge
... \wedge f_n) := n!~d \log  f_1  \wedge ... \wedge d \log f_n.
$$
\el

\paragraph{{\rm Lie}-period morphisms of complexes.} 
The graded commutative Hopf algebra ${\cal H}_\ast$ 
gives rise to a graded Lie coalgebra $({\cal L}_\ast, \delta)$: 
$$
{\cal L}_\ast:= \frac{{\cal H}_{>0}}{{\cal H}_{>0} \cdot {\cal H}_{>0} }. 
$$
Let ${  {\cal L}  }^{\bullet}(n)$ be the weight $n$ part of the
 standard cochain complex of the graded Lie coalgebra ${\cal L}_\ast$:
$$
{  {\cal L}  }^{\bullet}(n):= ~~~~~~{\cal L}_n \stackrel{\delta}{\lra} (\Lambda^2 {\cal L})_n \stackrel{\delta}{\lra} 
 (\Lambda^3 {\cal L})_n \stackrel{\delta}{\lra}  \ldots 
$$

\begin{conjecture} \label{periods1small}
There exists a  {canonical} morphism of complexes of sheaves on $X$ 
$$
 p^{\bullet}_n:   {\cal L}^{\bullet}(n)  
\lra     \Q_{\cal E}^{\bullet}(n),
$$ 
called the  {{\rm Lie}-period morphism}:
\be \la{4.14.15.1}
\begin{array}{ccccccccc}
&&{  {\cal L}}_n  & \stackrel{\delta}{\to} & 
({  {\cal L}} \wedge {  {\cal L}} )_n  & \stackrel{\delta}{\to} ... 
\stackrel{\delta}{\to} & (\Lambda^{n-1}{  {\cal L}})_n&\stackrel{\delta}{\to}
& \Lambda^n {  {\cal L}}_1\\
&&&&&&&&\\
&&\downarrow p_n^{1} &&\downarrow p_n^{2}   &&\downarrow p_n^{n-1} &&\downarrow p_n^{n} \\
&&&&&&&&\\
{\cal O}(n-1)&\to &\Lambda^2{\cal O}(n-2) & \to & \Lambda^3{\cal O}(n-3) & \to ... \to & \Lambda^n  {\cal O}& \stackrel{\rm exp}{\to} &\Lambda^n  {\cal O}^*\\
\end{array}
\ee
such that 

\begin{enumerate}

\item After the identification  
${  {\cal L}}_1 = {\cal O}^* $ the map $p_n^{n}$ is the identity map.

\item The composition   $\omega_n^{\bullet}\circ  p_n^{\bullet}$ is zero everywhere except on the very right. 
\end{enumerate} 
\end{conjecture}

The condition 2) just means that the following composition is zero:
$$
\begin{array}{ccccc}
{  {\cal L}}_n  & \stackrel{\delta}{\lra} & 
({  {\cal L}} \wedge {  {\cal L}})_n  & \stackrel{\delta}{\lra} ... 
\stackrel{\delta}{\lra} & (\Lambda^{n-1} {  {\cal L}})_n \\
\downarrow  &&\downarrow    &&\downarrow   \\
 \Lambda^2{\cal O}(n-2) & \lra & \Lambda^3{\cal O}(n-3) & \lra ... \lra &  \Lambda^n{\cal O}\\
\downarrow   &&\downarrow    &&\downarrow   \\
\Omega^1_U & \stackrel{d}{\lra} & \Omega^2_U &  \stackrel{d}{\lra} ... \stackrel{d}{\lra} &  \Omega^{n-1}_U
\end{array}
$$

Let us explain the meaning of  ``canonical'' in Conjecture \ref{periods1small}. 
A canonical map of complexes 
\be \la{CMCOMPL}
{\cal L}^\bullet(n) 
\stackrel{}{\lra}   \R_{\cal D}(n) 
\ee
was defined in \cite{G08}. There we defined, on the level of appropriate complexes,  a product 
$\R_{\cal D}(a) \otimes \R_{\cal D}(b) \lra \R_{\cal D}(a+b)$ making 
$
\oplus_{a=0}^\infty \R_{\cal D}(a)
$ 
into a DG commutative algebra.  The key property of 
the map (\ref{CMCOMPL}) is that its components describe a  map DG commutative algebras
$$
S^\bullet({\cal L}^\bullet[-1]) \lra \oplus_{a=0}^\infty \R_{\cal D}(a).
$$
So this map is completely determined its restriction to the Lie coalgebra ${\cal L}^\bullet$. 
The map $p_n^{\bullet}$, combined with a map $s_n^\bullet$ from
Section  6 
 or its  modification, 
 should deliver   canonical map (\ref{CMCOMPL}).

\subsection{The {\rm Lie}-period map} \la{sec4m} 
Recall the 
graded commutative Hopf algebra  over $\Q$ with a unit 
$A_{\bullet} = \oplus_{k=0}^\infty A_k$, see (\ref{theghaA}), coming with 
a product $\mu:  A_{\bullet} \otimes A_{\bullet} \lra A_{\bullet}$ and
 a coproduct 
$\Delta: A_{\bullet} \lra A_{\bullet} \otimes A_{\bullet}$.

Let $\mu^{(p)}: A_{\bullet}^{\otimes p} \lra A_{\bullet}$ be the product map:
 $a_1 \otimes ...  \otimes a_p \lms a_1 \cdot ... \cdot a_p$. 

Let us consider the iterated coproduct maps
$$
\Delta^{(p)}: A_{\bullet} \lra A_{\bullet}^{\otimes p}. 
$$
They are defined inductively:
$$
\Delta^{(p)}:= (\Delta \otimes {\rm Id}^{(p-2)})\circ \Delta^{(p-1)}, ~~~~
$$
$$
(\Delta \otimes {\rm Id}^{(p-2)})(a_1 \otimes \ldots \otimes a_{p-1}):= \Delta(a_1) \otimes a_2 \otimes \ldots \otimes a_{p-1}.
$$ 
Equivalently, they are dual to the product maps $\mu^{(n)}$ for the dual Hopf algebra.

Let us consider the following map:
$$
l: A_{\bullet} \lra A_{\bullet}, ~~~~l(M):= \sum_{p=1}^\infty \frac{(-1)^p}{p} \mu^{(p)}\circ \Delta^{(p)}(M).
$$
Elaborating this: 
$$
l:M \lms M - \frac{1}{2}\mu^{(2)}\circ \Delta^{(2)}(M) + \frac{1}{3}\mu^{(3)}\circ \Delta^{(3)}(M) -\ldots .
$$

The map $l$ has the following geometric interpretation. 
Denote by $G$ the pro-nilpotent group with the Lie algebra ${\rm Lie}(A_{\bullet})$. 
Then ${\cal O}(G) = A_{\bullet}$ as algebras. Let ${\rm Log}$ be the inverse of the exponential map. 
Then the map $l$ reads as follows: 
$$
l: {\cal O}(G) \lra {\cal O}(G), ~~~~l(F)(g):= \langle dF, {\rm Log}(g)\rangle.  
$$ 
So evidently the map $l$ is zero on $A_{>0}\cdot A_{>0}$. Therefore  we get a canonical map of graded spaces
$$
{\rm CoLie}(A_{\bullet}) \lra A_{\bullet}.
$$

Let us define a map, which we call the {\it {\rm Lie}-period map}:  
$$
{\cal P}_n:  {\rm CoLie}_n({\cal H}_\bullet) \lra \Lambda^2\C.   
$$
Consider the composition of the map $l$ with the big period map ${\rm P}_n'$: 
$$
{\cal P}_n =  {\rm P}'_n \circ l: {\cal H}_n \stackrel{l}{\lra} {\cal H}_n \stackrel{{\rm P}'_n}{\lra} \C \otimes \C. 
$$
\bp
The map ${\cal P}_n$ provides a  map 
$$
{\cal P}_n: {\rm CoLie}_n({\cal H}_\bullet) \lra \Lambda^2\C. 
$$
\ep

\begin{proof} The map ${\cal P}_n$ is a map 
$
{\rm CoLie}_n({\cal H}_\bullet) \to \C\otimes \C. 
$ 
We  need to check that its image lies in $\Lambda^2\C$. 
\end{proof}

\paragraph{Functions ${\rm L}_n(z)$ obtained from classical polylogarithms via the {\rm Lie}-period map.}
Set
$$
\sum_{k \geq 0} \beta_{k}t^k = \frac{t}{e^{t} -1}, ~~~~\beta_{k}:=  \frac{B_k}{k!}.
$$
So $\beta_{2m+1} = 0$ for $m\geq 1$, and  
$
\beta_{0} = 1, \beta_{1} = -\frac{1}{2}, \beta_{2} = \frac{1}{12}, \beta_{4} = -\frac{1}{720},  ...
 $ 
Let us consider a  function
$$
{\rm L}_n(z):= 
\sum_{k = 0}^{n-1}\beta_k {\rm Li}_{n-k}(z)\log^k z. 
$$
The right hand side is defined as follows. Take a path 
$\gamma$ from $a \in (0,1)$ to a point $z \in \C$ and continue analytically 
along this path the functions 
${\rm Li}_1(z), ... , {\rm Li_n}(z)$ using the inductive formula ${\rm Li}_m(z):= \int_{\gamma}
{\rm Li}_{m-1}(t) d\log t$. Then make the sum on the right hand using these brunches. 
So
$$
{\rm L}_2(z) = 
{\rm Li}_2(z) - \frac{1}{2}{\rm Li}_1(z) \log z.
$$%
$$
{\rm L}_3(z) = 
{\rm Li}_3(z) - \frac{1}{2}{\rm Li}_2(z) \log z + \frac{1}{12}{\rm Li}_1(z) \log^2 z.
$$
$$
{\rm L}_4(z) = 
{\rm Li}_4(z) - \frac{1}{2}{\rm Li}_3(z) \log z + \frac{1}{12}{\rm Li}_2(z) \log^2 z.
$$

The real version 
$
\pi_n\Bigl(\sum_{k = 0}^{n-1}\beta_k {\rm Li}_{n-k}(z)\log^k|z|\Bigr)
$ of the function ${\rm L}_n(z)$, 
where $\pi_n(a+ib) = a$ for odd $n$ and $ib$ for even $n$, 
was considered by Zagier in \cite{Za}, who showed that 
it is single valued. Its Hodge-theoretic interpretation  
was given by Beilinson and Deligne in \cite{BD}.

Denote by $\langle {\rm Li}_n(z)\rangle = (f^n, {\rm Li}_n(z), v_0)$ 
the $n$-framed Hodge-Tate structure assigned to the classical $n$-logarithm, whose normalised 
period matrix is given as follows: 
 $$
\left(\begin{array}{cccccc}
1&&&&&\\
-\frac{{\rm Li}_1(z)}{2\pi i}
&1&&&&\\
-\frac{{\rm Li}_2(z)}{(2\pi i)^2}&\frac{\log z}{2\pi i}&1&&&\\
-\frac{{\rm Li}_3(z)}{(2\pi i)^3}&\frac{\log^2 z}{2\cdot (2\pi i)^2}&\frac{\log z}{2\pi i}&1&&\\
-\frac{{\rm Li}_4(z)}{(2\pi i)^4}&\frac{\log^3 z}{3!\cdot (2\pi i)^3}&\frac{\log^2 z}{2\cdot (2\pi i)^2}&\frac{\log z}{2\pi i}&1&\\
-\frac{{\rm Li}_5(z)}{(2\pi i)^5}&\frac{\log^4 z}{4!\cdot (2\pi i)^4}&\frac{\log^3 z}{3!\cdot (2\pi i)^3}&
\frac{\log^2 z}{2\cdot 2\pi i}&\frac{\log z}{2\pi i}&1\\
\end{array}\right)
$$
Notice that in the normalised period matrix all entries are of weight zero. 

\bp The maximal period of $n$-framed Hodge-Tate structure $l\langle {\rm Li}_n(z)\rangle$ is:
$$
\langle f^n~|~ l\langle {\rm Li}_n(z)\rangle~|~ v_0 \rangle = -{\rm L}_n(z). 
$$
\ep

\begin{proof}
Let us do an example first, the 5-logarithm. 
The  calculation gives $(2\pi i)^{-5}$ times  
$$
{\rm Li}_5  - \frac{1}{2}\cdot \Bigl({\rm Li}_4(z) \log z + {\rm Li}_3(z) \frac{\log^2 z}{2} +  
{\rm Li}_2(z) \frac{\log^3 z}{3!} + {\rm Li}_1(z) \frac{\log^4 z}{4!}\Bigr)
$$
$$
+ \frac{1}{3}\cdot \Bigl({\rm Li}_3(z) \log^2 z + (\frac{1}{2}\cdot 1  +  1 \cdot \frac{1}{2})\cdot 
{\rm Li}_2(z) \log^3 z +  
(\frac{1}{3!}\cdot 1  +  1 \cdot \frac{1}{3!} + 
 \frac{1}{2}\cdot \frac{1}{2} )\cdot {\rm Li}_1(z) \log^3 z\Bigr)
$$
$$
- \frac{1}{4}\cdot \Bigl({\rm Li}_2(z) \log^3 z + (\frac{1}{2}\cdot 1 \cdot 1 +  1 \cdot \frac{1}{2}\cdot 1 + 
1 \cdot 1 \cdot \frac{1}{2} )\cdot {\rm Li}_1(z) \log^4 z\Bigr) +
 \frac{1}{5}\cdot {\rm Li}_1(z) \log^4 z.
$$
In general we need to sum the following series in $x$ (where $x=\log z$ in our application):
$$
S(x):= 
1-\frac{1}{2}\cdot (e^x-1) + \frac{1}{3}\cdot (e^x-1)^2 - \frac{1}{4}\cdot (e^x-1)^3 + \frac{1}{5}\cdot (e^x-1)^5 - ...  
$$
One has 
$ 
S(x)(e^x-1) = \log (1 + e^x-1) = x. 
$
Therefore 
$
S(x) = \frac{x}{e^x-1}.
$
\end{proof}

\paragraph{Calculating the monodromy.} 
Let $\gamma_0$ (resp. $\gamma_1$) be a small counterclockwise 
loop around $0$ (resp. $1$).  
Let $T_{\gamma_0}$ (resp. $T_{\gamma_1}$)  be the monodromy operator around loop 
$\gamma_0$ (resp. $\gamma_1$). 

\begin{lemma} \label{anal6} 
One has 
\begin{equation} \label{anal8}
\frac{1}{2\pi i} (T_{\gamma_1} -{\rm Id}): \quad {\rm L}_n(x) \lms -  (-1)^{n-1}\beta_{n-1} \log^{n-1}(x).
\end{equation}
\end{lemma}

\begin{proof}
One has 
$$
\frac{1}{2\pi i} (T_{\gamma_1} -{\rm Id}): -{\rm Li}_n(z) \lms 
\frac{\log^{n-1} z}{(n-1)!}.
$$
From the definition of  the Bernoulli polynomials, 
$
\frac{x e^{tx}}{e^x -1} = \sum_{k=0}^{\infty} \frac{B_n(t)}{n!}x^n.
$ So 
$$
B_n(t) = \sum_{k=0}^{n} {n\choose k} B_kt^{n-k}, \quad \mbox{and} \quad B_n(1) = (-1)^{n-1}B_n 
\quad \mbox{for $n \geq 1$}.
$$
Therefore
$
\sum_{k=0}^{n-1} \frac{B_k}{k!(n-k-1)!} = \frac{B_{n-1}}{(n-1)!}.
$ 
Using this identity we get  the formula (\ref{anal8}). 
\end{proof}

{\bf Examples}. 
$$
 \frac{1}{2\pi i} (T_{\gamma_0}-{\rm Id}): \quad  \log x \lms 1, \quad {\rm L}_1(x) \lms 0, \quad {\rm L}_2(x) \lms  - 
\frac{1}{2}\cdot {\rm L}_1(x), 
$$
$$
{\rm L}_3(x) \lms - \frac{1}{2}\cdot {\rm L}_2(x)  -
\frac{1}{12}\cdot {\rm L}_1(x) \cdot \log x + \frac{1}{12}\cdot {\rm L}_1(x), 
$$
$$
{\rm L}_4(x) \lms  - \frac{1}{2}\cdot {\rm L}_3(x)  -
\frac{1}{12}\cdot {\rm L}_2(x) \cdot \log x + \frac{1}{12}\cdot {\rm L}_2(x).
$$

$$
\frac{1}{2\pi i} (T_{\gamma_1}-{\rm Id}): \quad  \log x \lms 0, \quad {\rm L}_1(x) \lms -1, \quad {\rm L}_2(x) \lms -
\frac{1}{2}\cdot \log x, \quad 
$$
$$
{\rm L}_3(x) \lms -\frac{1}{12}\cdot \log^2 x, \quad {\rm L}_4(x) \lms 0.
$$

\section{Period morphisms  on polylogarithmic motivic complexes of weights $\leq 4$.}

Given a field $F$, let us recall the inductive 
definition of the groups ${\cal B}_n(F)$ \cite{G91}. One can set ${\cal B}_2(F) = B_2(F)$. 
There is a map 
$$
\Z[F^*-\{1\}] \stackrel{\delta_n}{\lra} {\cal B}_n(F)\otimes F^*, ~~~~\{x\} \lms \{x\}_{n-1}\otimes x, ~~n>2. 
$$
Let us define a subgroup ${\cal A}_n(F) \subset {\rm Ker}~\delta_n$. 
Given an element $\sum_in_i\{f_i(t)\}$ in the kernel of $\delta_n$ for the field $F(t)$, 
the element $\sum_in_i(\{f_i(t_0)\} - \{f_i(t_1)\})$, where 
$t_0, t_1 \in F^* -\{1\}$, lies in ${\cal A}_n(F) $, and the subgroup 
${\cal A}_n(F)$ is generated by such elements.

\paragraph{Our goal.} We are going to construct, for $n\leq 4$, a morphism of complexes
\be \la{9.30.15.1}
\begin{array}{ccccccc}
{\cal B}_n(\C) &\lra &{\cal B}_{n-1}(\C)\otimes \C^* & \lra ... \lra & 
{\cal B}_{2}(\C)\otimes \Lambda^{n-2}\C^* &\lra &\Lambda^{n}\C^*\\
&&&&&&\\
\downarrow l_n^1 && \downarrow l_n^2 &&\downarrow l_n^{n-1}&&\downarrow l_n^n\\
&&&&&&\\
\Lambda^{2}\C (n-2)& \lra &\Lambda^{3}\C (n-3)&\lra ... \lra&\Lambda^{n}\C & \stackrel{\wedge^n{\rm exp}}{\lra} 
&\Lambda^{n}\C^*
\end{array}
\ee
such that its composition with $\omega^\bullet_n$
is zero, and 
the map $l_n^n$ is the identity.

\vskip 2mm
{\bf Remark.} 
If $n=4$ it will not be  the canonical map from Conjecture \ref{periods1small}. 
\vskip 2mm

We start with a few general observations which help to construct 
the map $l_n^\bullet$. 

\bp \la{htplat} Let ${\cal O}:= \C(t)$. Let us suppose 
that we have maps $l_n^1$ and $l_n^2$ such that:  

i) The following diagram commutative:
$$
\begin{array}{ccc}
\Z[{\cal O}^*-\{1\}] &\stackrel{\delta_n}{\lra}
 & {\cal B}_{n-1}({\cal O}) \otimes {\cal O}^*\\
&&\\
\downarrow l_n^1 &&\downarrow l_n^2 \\
&&\\
\Lambda^{2}{\cal O} (n-2)&\lra &\Lambda^{3}{\cal O} (n-3)
\end{array}
$$

ii) The following composition  is zero:
\be \la{comp}
\omega\circ l_n^1: \Z[{\cal O}^*-\{1\}] \lra \Lambda^{2}{\cal O} (n-2) \lra \Omega^1, ~~~~ 
\omega(f\wedge g) = fdg-gdf.
\ee
Then 
\be \la{mainpr}
l_n^1({\cal A}_n({\cal O})) =0.
\ee
\ep

\begin{proof} Consider the following diagram: 
$$
\begin{array}{ccccccc}
{\cal A}_n({\cal O})& \lra &{\rm Ker}~\delta_n& \lra &\Z[{\cal O}] &\stackrel{\delta_n}{\lra}
 & {\cal B}_{n-1}({\cal O}) \otimes {\cal O}^*\\
&&&&&&\\
&&\downarrow&&\downarrow l_n^1 &&\downarrow l_n^2 \\
&&&&&&\\
&&{\cal O} (n-2)& \lra &\Lambda^{2}{\cal O} (n-2)&\lra &\Lambda^{3}{\cal O} (n-3)
\end{array}
$$
Then 
$$
l_n^1({\rm Ker}~\delta_n) \subset {\rm Ker}\Bigl(\Lambda^{2}{\cal O} (n-2) 
\lra \Lambda^{3}{\cal O} (n-3)\Bigr) = {\cal O} (n-2).
$$
The kernel of the restriction of the map 
$\omega: \Lambda^{2}{\cal O} \lra \Omega^1, ~~ f\wedge g \lms fdg=gdf$ to the subgroup 
$2\pi i \wedge {\cal O} \subset \Lambda^{2}{\cal O}$ is $2\pi i \wedge \C$. 
Therefore, since the composition (\ref{comp}) is zero, we have $l_n^1({\rm Ker}~\delta_n) \subset \C$. 
This implies  (\ref{mainpr}). 
\end{proof}

\bl \la{htplat11} Let $U \subset \C^*-\{1\}$ be an open subset. Suppose 
that we have a commutative diagram 
$$
\begin{array}{ccc}
\Z[{U}] &\stackrel{\delta_n}{\lra}
 & {\cal B}_{n-1}({\C}) \otimes {\C}^*\\
&&\\
\downarrow l_n^1 &&\downarrow l_n^2 \\
&&\\
\Lambda^{2}{\C} (n-2)&\lra &\Lambda^{3}{\C} (n-3)
\end{array}
$$
where the maps $l_n^1, l_n^2$ are 
given by products of $\log z, {\rm Li}_k(z)$, and 
the following composition  is zero:
\be \la{comp1}
\omega\circ l_n^1: \Z[{U}] \lra \Lambda^{2}{\C} (n-2) \lra \Omega^1_{\C/\Q}, ~~~~ 
\omega(f\wedge g) = fdg-gdf.
\ee
Then the map $l_n^1$ extends to a well defined map on $\Z[\C]$.
\el

\begin{proof}
Since the map $l_n^1$ is given by polylogarithms, 
it can be analytically continued to a multivalued function on 
$\C^*-\{1\}$ with values in $\Lambda^{2}{\C} (n-2)$. we need to prove that this function is 
single-valued. Take the monodromy around some loop minus the identity map. 
We get a multivalued function on 
$\C^*-\{1\}$ with values in $\Lambda^{2}{\C} (n-2)$ which is annihilated by the map 
$\Lambda^{2}{\C} (n-2)\lra \Lambda^{3}{\C} (n-3)$, and thus takes values in 
$\C(n-1)$. Since it is also killed by $\omega$, it is a constant, and then 
one can easily see that it must be zero. 
\end{proof}



\paragraph{Non-associative $*$-product. }
Any ring $A$ provides a  $\ast$ - product  
$$
\Lambda^{k+1} A \ast \Lambda^{l+1} A \lra \Lambda^{k +l+1} A,
$$  
$$ 
  (a_0 \wedge ... \wedge a_{k})  * (b_{0} \wedge ... \wedge b_{ l }) := \sum (-1)^{k-j+i}
a_0 \wedge ... \wedge \widehat a_i  \wedge ... \wedge a_{ k} \wedge a_i\cdot b_{j} \wedge b_0 \wedge ...  \wedge \widehat b_j  \wedge ... \wedge b_l. 
$$   
For instance 
$$
(a_0 \wedge a_1) * (b_0 \wedge b_1) =    a_0 \wedge a_1   b_0 \wedge b_1 - 
a_1 \wedge a_0   b_0 \wedge b_1 + a_0 \wedge a_1   b_1 \wedge b_0 - a_0 \wedge a_1   b_1 \wedge b_0.
$$

If $A$ is a commutative ring, then $(\Lambda^{\bullet -1}A, \ast)$ 
is a supercommutative non-associative algebra:
$$
(a_0 \wedge ... \wedge a_m) * (b_0 \wedge ... \wedge b_n) = 
(-1)^{mn} (b_0 \wedge ... \wedge b_n) * (a_0 \wedge ... \wedge a_m). 
$$

We  define the $*$-product on $\Lambda^{\bullet -1}{\cal O} $  
using the following algebra structure on   
$ {\cal O}$:
$$
a\ast b:= \frac{1}{2\pi i} ab. 
$$

It is useful to note the following formula: 
$$
(2\pi i \wedge a_1 \wedge ... \wedge a_m)\ast(2\pi i \wedge b_1 \wedge ... \wedge b_n) = 
(m+n+1) \cdot 2\pi i \wedge a_1 \wedge ... \wedge a_m \wedge b_1 \wedge ... \wedge b_n.
$$


\paragraph{Example 1.} Let us define a homomorphism of complexes
$$
\begin{array}{ccc}
{\cal B}_2(\C)& \lra & \Lambda^2\C^* \\
&&\\
\downarrow {\Bbb L}_2&&\downarrow =\\
&&\\
\Lambda^2\C&  \lra &\Lambda^2\C^*
\end{array}
$$ 
$$
{\Bbb L}_2: \frac{1}{2} \cdot\{x\}_2  \lms  
2 \pi i \wedge \frac{1}{2\pi i } ~{\rm L}_2(x)  - \frac{1}{2} {\rm Li}_1(x) \wedge \log x.
$$

\paragraph{Example 2.} 
Set 
\be \la{lnnm1}
l_n^{n-1}: \{x\}_2 \otimes y_1 \wedge ... \wedge y_{n-2} \lms {\Bbb L}_2(x) 
\ast (2 \pi i \wedge \log y_1 \wedge ... \wedge \log y_{n-2}). 
\ee

\bl \la{oooo}
The map (\ref{lnnm1}) gives rise to a group homomorphism 
$$
l_n^{n-1}: {\cal B}_2(\C) \otimes \Lambda^{n-2}\C^* \lra \Lambda^n\C.
$$ 
It makes the following diagram commutative:
\be \la{lSDQ}
\begin{array}{ccc}
{\cal B}_2(\C)\otimes \Lambda^{n-2}\C^*& \lra & \Lambda^n\C^* \\
&&\\
\downarrow l_n^{n-1}&&\downarrow =\\
&&\\
\Lambda^n\C&  \stackrel{\rm exp}{\lra} &\Lambda^n\C^*
\end{array}
\ee 
The following composition is zero:
\be \la{10001}
\omega_n^{n-1}\circ l_n^{n-1}: 
{\cal B}_2(\C) \otimes \Lambda^{n-2}\C^* \lra \Lambda^n\C \lra \Omega^{n-1}_{\C/\Q}.
\ee
\el 

\begin{proof} The maps 
$y_1 \wedge \ldots \wedge y_m \lms 2 \pi i \wedge \log y_1 \wedge ... \wedge \log y_{m}$ and 
${\Bbb L}_2: {\cal B}_2(\C) \lra \Lambda^2\C$ are well defined group homomorphisms.  
 Therefore the map (\ref{lnnm1}) is  a well defined group homomorphism. 
The commutativity is evident. 

Let us check that the composition (\ref{10001}) is zero. 
We write $d\left((a_1\wedge a_2) \ast (b_1 \wedge ... \wedge b_m)\right)$ as a sum
 with certain coefficients $\lambda, \mu$, skewsymmetrising with respect to 
$\{a_1, a_2\}$ as well as  $\{b_1, ... , b_m\}$:
\be
\begin{split}
&\lambda \cdot {\rm Alt}_{(a_1, a_2), (b_1, ..., b_m)} 
\Bigl((b_1a_1)\wedge da_2 - d(b_1a_1)\wedge a_2\Bigr) \wedge db_2 \wedge ... \wedge db_m \\
&+ \mu \cdot {\rm Alt}_{(a_1, a_2), (b_1, ..., b_m)} 
\Bigl(d(b_1a_1)\wedge da_2 \Bigr) \wedge b_1 \wedge db_2 \wedge ... \wedge db_m. \\
\end{split}
\ee
In our case 
$$
a_1 da_2 - a_2 da_1=0, ~~~\mbox{and} ~~~db_1 \wedge db_2 \wedge ... \wedge db_m=0.
$$ 
The first condition 
 implies that the second line is zero. 
The first and second condition  imply that the first line is zero. 
\end{proof}

\paragraph{Example 3.} Let us define a homomorphism of complexes
$$
\begin{array}{ccccc}
{\cal B}_3(\C)& \lra &{\cal B}_2(\C)\otimes \C^* & \lra & \Lambda^3\C^* \\
&&&&\\
\downarrow l_3^1&&\downarrow l_3^2&&\downarrow =\\
&&&&\\
\Lambda^2\C(1)& \lra & \Lambda^3\C& \lra &\Lambda^3\C^*
\end{array}
$$ 

Set 
\be
\begin{split}
&l_3^2: \{x\}_2\otimes y \lms  \frac{1}{2}\cdot {\Bbb L}_2(\{x\}_2) \ast ( 2 \pi i \wedge \log y) = \\
&\Bigl( 2 \pi i \wedge \frac{1}{2\pi i}{\rm L}_2(x) - 
\frac{1}{2}\cdot {\rm L}_1(x) \wedge \log x\Bigr) \ast ( 2 \pi i \wedge \log y) = \\
&3 ~\cdot ~2 \pi i \wedge \frac{1}{2\pi i}{\rm L}_2(x)\wedge \log y -
 {\rm L}_1(x) \wedge \log x \wedge \log y \\
&+ \frac{1}{2}~\cdot~2\pi i \wedge  \Bigl(\frac{1}{2\pi i}\log y ~{\rm L}_1(x) \wedge \log x + 
{\rm L}_1(x) \wedge 
\frac{1}{2\pi i}\log y\log x\Bigr).
\end{split}
\ee

By Lemma \ref{oooo}, the map $l_3^2$ is 
well defined, makes the second square commute, and   
$\omega_3^2 \circ l_3^2 =0$.

Set 
$$
{\Bbb L}_3: -\frac{1}{6}\cdot \{x\}_3 \lms  \quad 2\pi i \wedge 
{\rm L}_3(x)  
-\frac{1}{2}\cdot \frac{1}{2\pi i}{\rm L}_2(x) \wedge \log x  
-  \frac{1}{12}\cdot  \Bigl({\rm L}_1(x) \wedge \log x\Bigr)\ast \log x. 
$$

One checks that $\omega_3^1 \circ l_3^1 =0$ thanks to the differential 
equations for the polylogarithms. 

This map makes the first square commutative. Indeed, we have 
$$
l_3^2: \{x\}_2\otimes x \lms 
2 \pi i \wedge \Bigl( 3 \cdot {\rm L}_2(x) \wedge \log x  + 
 \frac{1}{2}\cdot  \Bigl({\rm L}_1(x) \wedge \log x\Bigr)\ast \log x \Bigr). 
$$

Thanks to Lemma \ref{htplat11} the map ${\Bbb L}_3$ provides a  single-valued map 
$$
{\Bbb L}_3: \C^*-\{1\} \lra \Lambda^2\C(1). 
$$ 

Proposition 
\ref{htplat} implies that it gives rise to a homomorphism
$$
{\Bbb L}_3: {\cal B}_3(\C) \lra \Lambda^2\C(1). 
$$
Therefore we get a well defined  morphism of complexes. 


\paragraph{Example 4.} We define a homomorphism of complexes
$$
\begin{array}{ccccccc}
{\cal B}_4(\C)& \lra &{\cal B}_3(\C)\otimes \C^*  & \lra &{\cal B}_2(\C)\otimes \Lambda^2\C^* & \lra & 
\Lambda^4\C^* \\
&&&&&&\\
\downarrow l_4^1&&\downarrow l_4^2&&\downarrow l_4^3&&\downarrow l_4^4\\
&&&&&&\\
\Lambda^2\C(2)& \lra & \Lambda^3\C(1)& \lra & \Lambda^4\C& \lra &\Lambda^4\C^*
\end{array}
$$

We set
$$
l_4^3(\{x\}_2 \otimes y_1 \wedge y_2):=  
{\Bbb L}_2(x) \ast (2 \pi i \wedge \log y_1 \wedge \log y_2) = 
$$
$$
   \Bigl(2 \pi i \wedge \frac{1}{2\pi i}{\rm L}_2(x) - 
\frac{1}{2 }{\rm L}_1(x) \wedge \log x\Bigr)\ast (2 \pi i \wedge \log y_1 \wedge \log y_2).
$$
By Lemma \ref{oooo} the map $l_4^3$ is 
well defined, makes the last square commute, and $\omega_4^3\circ l_4^3=0$. 

Next, set 
$$
l_4^2(\{x\}_3 \otimes y):= 
$$
\be
\begin{split}
&2 \pi i \wedge \Bigl( -12\cdot \frac{1}{(2\pi i)^2}{\rm L}_3(x) \wedge \log y - 
2\cdot (\frac{1}{2\pi i}{\rm L}_2(x) \wedge \log x) \ast \log y\\
  &   -\frac{1}{2} \cdot \frac{1}{2\pi i} \Bigl({\rm L}_1(x)  \log x\Bigr) \wedge \log x \log y  -
  \frac{1}{2} \cdot \frac{1}{2\pi i} \Bigl({\rm L}_1(x)  \log y\Bigr) \wedge (\log x)^2 \Bigr) \\
&+4\cdot \frac{1}{2\pi i}{\rm L}_2(x) \wedge \log x  \wedge \log y   + 
\frac{1}{2} \cdot \Bigl({\rm L}_1(x) \wedge \log x \Bigr) \ast \Bigl(\log x  \wedge \log y\Bigr).\\  
\end{split}
\ee

Direct check shows that the 
middle square has all the desired properties. 

Finally, we set

$$
{\Bbb L}_4: \frac{1}{24}\cdot \{x\}_4 \lms  
$$
\be
\begin{split}
&2 \pi i \wedge \frac{1}{(2\pi i)^3}{\rm L}_4(x) 
- \frac{1}{2} \cdot \frac{1}{(2\pi i)^2}{\rm L}_3(x) \wedge \log x\\  
&-\frac{1}{12} \cdot \Bigl(\frac{1}{2\pi i}{\rm L}_2(x) \wedge \log x\Bigr) \ast \log x 
- \frac{1}{24} \cdot \frac{1}{2\pi i}\Bigl({\rm L}_1(x) \log x\Bigr) \wedge \log^2x.\\
\end{split}
\ee

One checks that the left square is formally 
commutative. Thanks to Lemma \ref{htplat11} and Proposition 
\ref{htplat} the map $l_4^1:= {\Bbb L}_4$ is a well defined homomorphism of abelian groups. 
Finally, we check that $\omega_4^1 \circ l_4^1=0$ by using the differential equations for the 
classical polylogarithms. 

\paragraph{Example: the regulator map on the weight three motivic complex.} 

Let $X$ be a regular complex projective curve. Then the  
motivic complex $\Z_{\cal M}(X;3)$ is the total of the following complex, where 
${\cal O}_{\cal X}:= \C(X)$ and 
 ${\rm Res}$ stands for the tame symbol on the right and the map 
$\{f\}_2 \otimes g \lms \sum_{x\in X}{\rm val}_x(g) \{f(x)\}_2$ in the middle: 
\be \la{motcom}
\begin{array}{ccccc}
{\cal B}_3({\cal O}_{\cal X}) & \stackrel{\delta}{\lra} &{\cal B}_2({\cal O}_{\cal X}) \otimes {\cal O}_{\cal X}^*&\stackrel{\delta}{\lra} &\Lambda^3{\cal O}_{\cal X}^*\\
&&&&\\
&&\downarrow {\rm Res}&&\downarrow {\rm Res}\\
&&&&\\
&&\coprod_{x\in X}{\cal B}_2(\C)  &\lra &\coprod_{x\in X}\Lambda^2\C^*
\end{array}
\ee
The top line is mapped to the weight three {\rm Lie}-exponential complex at the generic point ${\cal X}$:
$$
\begin{array}{ccccccccc}
&&{\cal B}_3({\cal O}_{\cal X}) & \stackrel{\delta}{\lra} &{\cal B}_2({\cal O}_{\cal X}) \otimes {\cal O}_{\cal X}^*&\stackrel{\delta}{\lra} &
\Lambda^3{\cal O}_{\cal X}^*\\
&&\downarrow &&\downarrow &&\downarrow \\
{\cal O}_{\cal X}(2)& \lra &\Lambda^2{\cal O}_{\cal X}(1)&\lra &\Lambda^3{\cal O}_{\cal X}& \lra & \Lambda^3{\cal O}_{\cal X}^*\\
\end{array}
$$
An important property of the {\rm Lie}-period is that the element  
${\Bbb L}_3(f(x)) \in \Lambda^2{\cal O}_{\cal X}(1)$ is non-singular:
$$
{\Bbb L}_3(f(x)) \in \Lambda^2{\cal O}(1). 
$$
So there is a  map 
$$
{\cal B}_3({\cal O}_{\cal X}) \lra \Lambda^2{\cal O}(1). 
$$

The element $l_3^2(\sum \{f_i(x)\}_2 \otimes g_i(x)) \in \Lambda^3{\cal O}_{\cal X}$ can have singularities 
at the divisors  of the functions $g_i$. 
To guarantee that the singularity at $y\in X$ is absent it is sufficient 
to require that the residue of that element at $y$ is zero. 
So there is a map on the kernel of the residue map: 
$$
{\rm Ker}\Bigl({\cal B}_2({\cal O}_{\cal X}) \otimes {\cal O}_{\cal X}^* 
\stackrel{\rm Res}{\lra} \coprod_{x\in X}{\cal B}_2(\C)\Bigr) \lra \Lambda^3{\cal O}. 
$$
So  we get a map  
\be
\begin{split} &Z^2({\rm B}^\bullet(X; 3)) := {\rm Ker}\Bigl({\cal B}_2({\cal O}_{\cal X}) \otimes {\cal O}_{\cal X}^*
\stackrel{}{\lra} \coprod_{x\in X}{\cal B}_2(\C)\oplus \Lambda^3{\cal O}_{\cal X}^*\Bigr) \\
&\lra {\rm Ker}\Bigl(\Lambda^3 {\cal O} \stackrel{\wedge^3{\rm exp}}{\lra} 
\Lambda^3 {\cal O}^*\Bigr).
\end{split}
\ee

It gives rise to an explicit map
$$
H^2({\rm B}^\bullet(X; 3)) \lra H^2(X, \Gamma_{\cal D}(3)).
$$
It can be described explicitly as follows. Take a cycle $A\in Z^2({\rm B}^\bullet(X; 3)) $.  
Then we have 
$$
l_3^2(A) \in {\rm Ker}\Bigl(\Lambda^3 {\cal O} \stackrel{\wedge^3{\rm exp}}{\lra} 
\Lambda^3 {\cal O}^*\Bigr) = \Lambda^2 {\cal O}(1).  
$$
Pick an open cover $\{U_i\}$ of $X$ by small discs. 
On each cover we get can find an element 
$$
C(U_i) \in \Lambda^2{\cal O}_{U_i}(1): ~~dC(U_i) = A_{|U_i} \in \Lambda^3{\cal O}_{U_i}. 
$$
Then  $d (C(U_i)-C(U_j))=0$ on $U_{ij}$. So we can 
find a $C(U_i, U_j) \in {\cal O}_{U_{ij}}(2)$ such that $dC(U_i, U_j) = C(U_i)-C(U_j)$. 
Similarly we find $C(U_i, U_j, U_k) \in \Q(3)$. Now taking the image of the cocycle 
$(C(U_i), C(U_i, U_j), C(U_i, U_j, U_k))$ in the {\rm Lie}-exponential Deligne complex 
$  \Gamma_{\cal D}(X;3)$ we get a cycle representing the regulator of $A$. 

Unlike the de Rham complex,   the exponential complex is exact in a trivial way:  
finding a primitive does not require  integration. 
So our construction is effective. 

One can generalize the above construction to the case when $X$ is an arbitrary 
regular complex variety. In this case the motivic complex 
we use is the Gersten resolution of the weight three polylogarithmic complex. 
It is obtained by adding to (\ref{motcom}) the contributions of the codimension two and three cycles. 
The construction remains the same. 


\section{A local combinatorial construction of characteristic classes}



 \subsection{A map: decorated flags  complex$~\to~$ Bigrassmannian complex} \label{sec3.2.4}

\paragraph{Configuration complexes.} Let $X$ be a set. Let $G$ be a group acting on $X$.
{\it Configurations} of $m$ elements in $X$ are orbits of the group $G$ acting on $X^m$. 
The complex of configurations $C'_\ast(X)$ is the complex of the $G$-coinvariants of the chain complex of the
simplex with the vertices parametrized by  $X$:
$$
\stackrel{d}{\lra} C'_m(X) \stackrel{d}{\lra} C'_{m-1}(X) \stackrel{d}{\lra}\ldots \stackrel{d}{\lra} C'_1(X). 
$$
So  $C'_m(X)$ 
is the free abelian group generated by configurations. Denote by $(x_1, ..., x_m)$ the generator 
provided by the configuration  corresponding to 
the $G$-orbit of an $m$-tuple  $\{x_1, ..., x_m\}$.   
The differential is 
$$
d: C'_{m+1}(X) \lra C'_m(X), ~~~~
(x_0, ..., x_m) \lms \sum_{i=0}^{i}(-1)^i(x_0, ..., \widehat x_i, ... , x_m).
$$
\vskip 2mm

Let us assume now that $X$ is an algebraic variety over $\Z$, and $\G$  an algebraic group over $\Z$ acting 
on $X$. Then for any field $F$ 
there is a $\G(F)$-set  $X(F)$. So we get complexes of 
configurations of $X(F)$. Abusing notation, we skip the field $F$ from the notation.

Suppose that we have a notion of generic configurations of points 
in $X$, stable under the operation of 
forgetting a point. We assume that generic configurations of $m$ points in 
 $X$ are parametrised by a variety 
${\rm Conf}_m^*(X)$. So forgetting the $i$-th point provides a  map 
$$
f_i: {\rm Conf}_m^*(X) \lra {\rm Conf}_{m-1}^*(X). 
$$
Consider the free abelian group generated by the $F$-points of ${\rm Conf}^*_m(X)$. 
$$
C_m(X):= \Z[{\rm Conf}^*_m(X)(F)].
$$
We get  a subcomplex of the complex $C'_\bullet(X)$, called the {\it complex of  generic configurations}:
$$
C_\bullet(X) :~~~~~~\stackrel{d}{\lra} C_m(X) \stackrel{d}{\lra} C_{m-1}(X) \stackrel{d}{\lra}\ldots \stackrel{d}{\lra} C_1(X). 
$$

\paragraph{An example: Grassmannian complexes \cite{Su84}.} Let ${\rm Conf}^*_m(q)$ be 
the variety of generic configurations of $m$ vectors in a vector space of dimension $q$. 
A configuration  is generic if any $k\leq q$ of the vectors are linearly independent. 
Observe that the configuration spaces assigned to isomorphic vector spaces are {\it canonically} isomorphic. 

The variety ${\rm Conf}^*_m(q)$ is defined over ${\rm Spec}(\Z)$: a collection of generic vectors is given by a 
$q \times m$ matrix with non-zero principal minors. 
So we get abelian groups 
$$
C_m(q):= \Z[{\rm Conf}^*_m(q)(F)]. 
$$
They form  the 
{\it weight $q$ Grassmannian complex}:
$$
\stackrel{d}{\lra} C_m(q) \stackrel{d}{\lra} C_{m-1}(q) \stackrel{d}{\lra}\ldots \stackrel{d}{\lra} C_1(q). 
$$

\paragraph{The Bigrassmannian \cite{G93}.} 
Given a configuration of $(m+1)$ vectors $(l_0, ..., l_m)$ in a $q$-dimensional vector space 
$V_q$, there are two 
ways to get a configuration 
of $m$ vectors:
\begin{enumerate}

\item Forgetting the $i$-th vector $l_i$, we get a map 
$$
f_i: {\rm Conf}^*_{m+1}(q) \lra {\rm Conf}^*_{m}(q), ~~~~(l_0, ..., l_m)\longmapsto 
(l_0, ..., \widehat l_i, ..., l_m).
$$

\item Projecting the vectors $(l_0, \ldots, \widehat l_j, \ldots , l_m)$ to the quotient $V_q/(l_j)$ by the subspace spanned by $l_j$, we get a map 
$$
p_j: {\rm Conf}^*_{m+1}(q) \lra {\rm Conf}^*_{m}(q-1), ~~~~(l_0, ..., l_m)\longmapsto 
(l_j~|~ l_0, ..., \widehat l_i, ...  l_m). 
$$
\end{enumerate}

Denote by ${\bf G}_m(q)$ the Grassmannian 
of $q$-dimensional subspaces in a vector space of dimension $m$ with 
a \underline{given} basis $(e_1, ..., e_m)$, in generic position to the coordinate hyperplanes. 

There is
a canonical isomorphism
$$
{\bf G}_m(q) = {\rm Conf}^*_{m}(q).
$$ 
It assigns to a generic $q$-plane $\pi$ 
 a configuration of vectors in the dual space 
$\pi^*$ given by the restrictions of the linear coordinate functionals $x_i$ 
dual to the basis.  

Using this, we organise the spaces ${\rm Conf}^*_m(q)$ into a single object,  
  the {\it Bigrassmannian}:

\be \label{bigrass}
\begin{array}{ccccccc}
&&&&\ldots &\stackrel{\lra}{\stackrel{\ldots}{\lra}} &{\bf G}_5(4)\\
&&&&\downarrow... \downarrow&&\downarrow... \downarrow \\
&&\ldots &\stackrel{\lra}{\stackrel{\ldots}{\lra}} &{\bf G}_5(3)&\stackrel{\lra}{\stackrel{\ldots}{\lra}} &{\bf G}_4(3)\\
&&\downarrow... \downarrow &&\downarrow... \downarrow &&\downarrow ... \downarrow\\
\ldots &\stackrel{\lra}{\stackrel{\ldots}{\lra}} &{\bf G}_5(2) &\stackrel{\lra}{\stackrel{\ldots}{\lra}}
 &{\bf G}_4(2) &\stackrel{\lra}{\stackrel{\ldots}{\lra}} &{\bf G}_3(2) \\
\downarrow... \downarrow &&\downarrow... \downarrow &&\downarrow ... \downarrow&&\downarrow ... \downarrow\\
{\bf G}_5(1)&\stackrel{\lra}{\stackrel{\ldots}{\lra}} &{\bf G}_4(1)&\stackrel{\lra}{\stackrel{\ldots}{\lra}} &{\bf G}_3(1)& \stackrel{\lra}{\stackrel{\ldots}{\lra}} & {\bf G}_2(1)
\end{array} 
\ee

Applying the functor  $X \to \Z[X(F)]$ to the 
Bigrassmannian we get the {\it Grassmannian bicomplex}:
$$
\begin{array}{ccccccc}
&&&&\ldots &\stackrel{p}{\lra} &C_5(4)\\
&&&&\downarrow p&&\downarrow p\\
&&\ldots &\stackrel{f}{\lra} &C_5(3)&\stackrel{f}{\lra} 
&C_4(3)\\
&&\downarrow p&& \downarrow p&& \downarrow p\\
\ldots &\stackrel{f}{\lra} &C_5(2) &\stackrel{f}{\lra}
 &C_4(2) &\stackrel{f}{\lra} &C_3(2) \\
 \downarrow p&& \downarrow p&& \downarrow p&& \downarrow p\\
C_5(1)&\stackrel{f}{\lra} &C_4(1)&\stackrel{f}{\lra} &C_3(1)& \stackrel{f}{\lra} & C_2(1)
\end{array} 
$$
Here the maps $f$ and $p$ are the alternating sums of the maps $f_j$, and  $p_i$:
$$
f = \sum_{s=0}^m(-1)^sf_s, ~~~~ p  = \sum_{s=0}^m(-1)^sp_s.
$$

Denote by $BC_\ast$ the sum of the groups on the diagonals:
$$
BC_{m}:= \bigoplus_{q=1}^{m-1} C_{m}(q).
$$ 
Changing the signs of the differentials in the bicomplex, we get the {\it Bigrassmannian complex}
$$
\ldots \lra BC_5\lra BC_4 \lra BC_3\lra BC_2.
$$

\paragraph{Decorated flags.} 

\bd A {\it decorated flag} $F_\bullet$ in an $N$-dimensional vector 
space is  a collection of subspaces

\be \label{flag}   F_0 \subset F_1 \subset F_2 \subset\cdots 
\subset F_N\,, ~~~{\rm dim}F_i=i,
\ee   
together with a choice of a non-zero vector $f_i \in F_{i}/F_{i-1}$ for each $i=1, ..., N$. 
 \ed
 
A collection of $m+1$ decorated flags $(F_{0, \bullet}, F_{1, \bullet}, ..., F_{m, \bullet})$ in 
an $N$-dimensional vector 
space $V_N$ is  {\it generic}, if 
for any integers $a_0, ..., a_m$ which sum to $N$ one has an isomorphism
$$
F_{0, a_0} \oplus ... \oplus F_{m, a_m} = V_N. 
$$

Denote by ${\cal A}_N$ the variety of all decorated flags in $V_N$, and by 
${\rm Conf}^*_{m}({\cal A}_N)$ the variety of generic configurations  of $m$ decorated flags. 
It is defined over ${\rm Spec}(\Z)$. 
So for any field $F$ there is  the complex of generic configurations of decorated flags
$$
\ldots \lra C_m({\cal A}_N) \lra \ldots \lra C_2({\cal A}_N)  \lra C_1({\cal A}_N).
$$

\paragraph{From configurations of decorated flags to configurations of vectors.} 
We start with a collection of $m+1$ generic decorated flags in an $N$-dimensional vector space $V_N$: 
\be \label{flags}
(F_{0, \bullet}, ..., F_{m, \bullet}).
\ee

Given a partition
\be \label{pcf}
{\bf a} = \{a_0, \ldots , a_m\}, \quad a_0 + \ldots + a_m=N-(q+1), \quad  a_i\geq 0,
\ee 
consider a codimension $q+1$  linear subspace 
of $V_N$ given by the sum of the flag subspaces $F_{i, a_i}$: 
\be 
F_{0, a_0}\oplus F_{1, a_1}\oplus \ldots \oplus F_{m, a_m} \subset V_N.
\ee
Take the quotient  by this subspace
\be Q_{\bf a} = \frac{V_N }{ F_{0, a_0}\oplus F_{1, a_1}\oplus \ldots \oplus F_{m, a_m}}.
\ee

We use the decorations to produce a configuration of $(m+1)$ vectors in the quotient $Q_{\bf a}$. 
Namely, the ``next'' decoration vector $f_{a_i +1} \in F_{i, a_i+1}/ F_{i, a_i}$ in the decorated flag $F_i$ 
provides a vector in the quotient, denoted by $l_i$.
The vectors $\{l_0, ..., l_m\}$ in the space $ Q_{\bf a}$ provide a configuration 
$(l_0, \ldots , l_m)$. 
So we put
$$
\pi_{\bf a}(F_{0, \bullet}\ldots, F_{m, \bullet}):= (l_0, \ldots, l_m)\in {\rm Conf}^*_{m+1}(q+1).
$$

So a partition ${\bf a}$ gives rise to a  projection
$$
\pi_{\bf a}: {\rm Conf}^*_{m+1}({\cal A}_N) \lra {\rm Conf}^*_{m+1}({q+1}). 
$$

\paragraph{The main construction \cite[Section 2]{G93}.}

\begin{itemize}

\item 
{\it Given a configuration $(F_{0, \bullet}, F_{1, \bullet}, ..., F_{m, \bullet})$ of decorated flags in $V_N$, we 
assign to {\it every} partition 
${\bf a}$ as in (\ref{pcf}) the configuration  
of vectors $\pi_{\bf a}(F_{0, \bullet}, F_{1, \bullet}, ..., F_{m, \bullet})$ in a $(q+1)$-dimensional vector space, 
and take the} {\bf sum over all $q$ and all partitions ${\bf a}$}:
$$
c_m: (F_{0, \bullet}, F_{1, \bullet}, ..., F_{m, \bullet}) \lra \sum_{{\bf a}}\pi_{\bf a}(F_{0, \bullet}, F_{1, \bullet}, ..., F_{m, \bullet})\in BC_{m+1}.
$$
{\it We extend the map to a homomorphism of abelian groups} 
$$
c_m: C_{m+1}({\cal A}_N)\lra BC_{m+1}. 
$$
\end{itemize}

The following crucial result was proved in Lemma 2.1 from \cite{G93}. 
\begin{theorem} \la{MTSG}
The collection of maps $c_n$ gives rise to a homomorphism of complexes
\be \label{dfltobigr}
\begin{array}{cccccccc}
\lra &C_5({\cal A}_N)&\lra &C_4({\cal A}_N)& \lra &C_3({\cal A}_N)&\lra &C_2({\cal A}_N)\\
&\downarrow c_4&&\downarrow c_3 &&\downarrow c_2&&\downarrow c_1\\
\lra &BC_5&\lra &BC_4&\lra &BC_3&\lra &BC_2
\end{array}
\ee
\end{theorem}

Our next goal is to give an interpretation of this map via hypersimiplicial decompositions. 

\subsection{Hypersimplicial decompositions of simplices and a  proof of Theorem \ref{MTSG}} 

\paragraph{Hypersimplices \cite{GGL}.} Let $p,q \geq 0$ be a pair of non-negative integers.  
Set $p+q=m-1$. A hypersimplex $\Delta^{p,q}$ is a hyperplane section of 
the $(m+1)$-dimensional unit cube:
$$
\Delta^{p,q}:= \{(x_0, ..., x_m) \in [0,1]^{m+1} ~|~ \sum_{i=0}^m  x_i = q+1\}, ~~~~p+q=m-1.
$$ 

The hypersimplex $\Delta^{p,q}$ is a  convex polyhedron 
isomorphic to the convex hull of the centers of $q$-dimensional faces 
of an $m$-dimensional simplex. 
\vskip 2mm

The hypersimplices $\Delta^{p, 0}$ and $\Delta^{0, q}$ are just simplices. 
The hypersimplex $\Delta^{1,1}$ 
is the octahedron. 
It is the convex hull of the centers of the edges of a tetrahedron. 
\vskip 2mm

The boundary of a hypersimplex $\Delta^{p,q}$ is a union of $m+1$ hypersimplices 
$\Delta^{p-1,q}$ and $m+1$ hypersimplices $\Delta^{p,q-1}$. 
They are given by the intersections with the hyperplanes $x_i=1$ and $x_i=0$ of the unit cube. 
For example, the boundary of the octahedron 
$\Delta^{1,1}$ consists of four $\Delta^{1, 0}$-triangles  and four $\Delta^{0, 1}$-triangles.

\paragraph{Hypersimplicial $N$-decomposition of a simplex \cite[Section 10.4]{FG1}.} 
It is a canonical decomposition of an $m$-dimensional simplex 
into hypersimplices which depend on an additional natural number $N$. 

Consider the standard coordinate space $\R^{m+1}$. It contains the integral lattice $\Z^{m+1}$. 
The integral 
hyperplanes  $x_i =s$, $s\in \Z$, cut the space into unit cubes with vertices at  integral points. 
Take an $m$-dimensional 
simplex  given by the intersection of the 
hyperplane $\sum x_i =N$ with the positive octant:
$$
\Delta_{(N)}^m = \{(x_0, ..., x_m) ~|~ x_i\geq 0, ~\sum_{i=0}^m x_i =N\}.
$$

The integral hyperplanes $x_i =s$ cut this simplex 
into a union of hypersimplices. 
Indeed, the hyperplane 
$\sum x_i =N$ intersects each of the standard unit  lattice cubes either by an empty set, or by a 
hypersimplex. We call it 
{\it a hypersimplicial $N$-decomposition of an $m$-dimensional simplex}. 
A hypersimplicial $N$-decomposition of a simplex induces a hypersimplicial $N$-decomposition of 
each face of the simplex.

\bl The $\Delta^{p,q}$-hypersimplices 
 of the hypersimplicial $N$-decomposition of an $m$-simplex match 
partitians ${\bf a} = (a_0, ..., a_m)$, where 
\be \label{a-equation}
a_0 + \ldots + a_m=N-(q+1), ~~~ a_i \in \Z_{\geq 0}.
\ee
\el

\begin{proof} The standard hypersimplex $\Delta^{p,q}$ consists of the points of the unit cube 
with coordinates 
$(x_0, ..., x_m)$ satisfying $x_0+ ... + x_m = q+1$. So any partition ${\bf a}$ provides a hypersimplex 
$$
(a_0, ..., a_m) + (x_0, ..., x_m) \subset \Delta^m_{(N)}. 
$$
So we parametrise hypersimplices in $\Delta_{(N)}^m$  by the coordinates   
$(a_0, ..,., a_m)$ of their ``lowest'' vertices. 
\end{proof}
\noindent
Let $\Delta_{\bf a}^{p,q}$ be the hypersimplex of a hypersimplicial $N$-decomposition 
assigned to a partition ${\bf a}$.

\paragraph{Examples.} 1. The $N$-decomposition of a segment $\Delta^1$ is a decomposition into $N$ little 
$\Delta^{0,0}$-segments. They match partitians $a_0+a_1 = N-1$. 

2. The $N$-decomposition of a triangle $\Delta^2$ is a decomposition into 
triangles of two types, $\Delta^{1, 0}$ and $\Delta^{0,1}$. The $\Delta^{1,0}$-triangles 
 match partitians $a_0+a_1+a_2 = N-1$. The $\Delta^{0,1}$-triangles 
match partitians $a_0+a_1+a_2 = N-2$. 

3. The $N$-decomposition of a tetrahedron $\Delta^3$ has tetrahedrons of two types and octahedron. 
The $\Delta^{2, 0}$-tetrahedrons 
 match partitians $a_0+a_1+a_2+a_3 = N-1$. The $\Delta^{1,1}$-octahedrons 
match partitians $a_0+a_1+a_2+a_3 = N-2$. The $\Delta^{0,2}$-tetrahedrons 
 match partitians $a_0+a_1+a_2+a_3 = N-3$. 

\vskip 3mm

Recall that a hypersimplex $\Delta^{p,q}$ has $2(m+1)$ codimension one faces:  
$(m+1)$ of them are hypersimplices of type $\Delta^{p-1,q}$, and the other $(m+1)$ are
hypersimplices of type  $\Delta^{p, q-1}$. 

Each hypersimplex $\Delta^{p,q}_{\bf a}$ 
is surrounded by $m+1$ hypersimplices $\Delta^{p+1,q-1}_{\bf b}$, 
sharing with it a codimension one face of type $\Delta^{p, q-1}$. 
The ${\bf b}$'s are obtained from ${\bf a}$ by adding $1$ to one of the coordinates 
$(a_0, ..., a_m)$. So the collection of ${\bf b}$'s is
$$
(a_0+1, a_1, a_2, ..., a_m), ~~(a_0, a_1+1, a_2, ..., a_m), ~~\ldots ~~, 
(a_0, a_1, a_2, ..., a_m+1).
$$

Each hypersimplex $\Delta^{p,q}_{\bf a}$ is also surrounded by $m+1$ hypersimplices $\Delta^{p-1,q+1}_{\bf c}$, 
sharing with it a codimension one face of type $\Delta^{p-1,q}$.  
The ${\bf c}$'s are 
obtained from ${\bf a}$ by subtracting $1$ from one of the coordinates 
$(a_0, ..., a_m)$. So the collection of ${\bf c}$'s is 
$$
(a_0-1, a_1, a_2, ..., a_m), ~~(a_0, a_1-1, a_2, ..., a_m), ~~\ldots ~~, 
(a_0, a_1, a_2, ..., a_m-1).
$$

\vskip 3mm
The combinatorics of hypersimplices is related  \cite{GM82} to the geometry of the 
Grassmannians. 
$$\mbox{\it 
The Grassmannian ${\bf G}_{p+q+2}(q+1)$ matches the hypersimplex $\Delta^{p,q}$.}
$$ 
Precisely, consider the action of the coordinate torus ${\rm T}_{p+q+2} = {\Bbb G}^{p+q+2}_m$ 
on the Grassmannian ${\bf G}_{p+q+2}(q+1)$. Then the closure of each of the generic ${\rm T}_{p+q+2}$-orbits 
is a $(p+q+1)$-dimensional toric variety, and combinatorics of its boundary strata coincides with the 
structure of the hypersimplex $\Delta^{p,q}$. Alternatively, it follows from the general 
Convexity Theorem of Atiyah \cite{A82} that the image of ${\bf G}_{p+q+2}(q+1)$ under the moment map 
assigned to the torus action 
is the hypersimplex $\Delta^{p,q}$. 

\paragraph{A  proof of Theorem \ref{MTSG}.} 
Our key construction provides a map 
\be \label{hdflbc}
\mbox{\it Complex of decorated flags} ~\lra ~ 
\mbox{\it Bigrassmannian 
complex}.
\ee
To see that it commutes with differentials, we rephrase it as a correspondence from the 
variety ${\rm Conf}_{m+1}^*({\cal A}_N)$ to the Bigrassmannian: 

\begin{itemize}

\item 
{\it Given a generic configuration of $(m+1)$ decorated flags
 $(F_{0, \bullet}, F_{1, \bullet}, ..., F_{m, \bullet})$  in $V_N$, we define a collection of points in the 
Grassmannians ${\bf G}_{m+1}(\ast)$. 
These points are parametrised by the hypersimplices 
of the hypersimplicial $N$-decomposition of an $m$-dimensional simplex: 

Each hypersimplex $\Delta^{p,q}_{\bf a} \subset \Delta^m_{(N)}$ 
gives rise to a point of the Grassmannian ${\bf G}_{m+1}(q+1)$:}
\be \la{4.19.15.1}
\pi_{\bf a}(F_{0, \bullet}, F_{1, \bullet}, ..., F_{m, \bullet}) \in {\bf G}_{m+1}(q+1).
\ee
\end{itemize}

 Furthermore, the $2(m+1)$ elements provided by the boundary of the element (\ref{4.19.15.1}) 
match the ones assigned to the boundaries of the hypersimplex $\Delta^{p,q}_{\bf a} \subset \Delta^m_{(N)}$. 
The sum of the boundaries of all these hypersimplices is, of course, the boundary 
of the simplex $\Delta^m_{(N)}$ presented as a sum of its own hypersimplices. 
This just means that we get a homomorphism of complexes.

\vskip 3mm
We defined   homomorphisms of complexes (\ref{dfltobigr}):
\be \label{hdflblc1}
 \mbox{\it Complexes of generic decorated flags in $V_N$}~\lra ~ \mbox{\it the Bigrassmannian 
complex}. 
\ee 

We will review in Section \ref{sec3.2.4n}  homomorphisms of complexes,
\be \label{hdflblc2}
\mbox{\it the Bigrassmannian 
complex} ~\lra ~ 
\mbox{\it weight $n$ motivic complex}, ~~~~n\leq 4.
\ee

Finally, we defined in Section \ref{sec4m} for $n \leq 4$ maps 
\be \label{hdflblc3}
\mbox{\it Weight $n$ polylogarithmic  
complex} ~\lra ~
\mbox{\it weight $n$ {\rm Lie}-exponential complex}, ~~n\leq 4. 
\ee 

Combining these three maps, we get explicit cocycles for the Chern classes 
with values in the Deligne cohomology for $n\leq 3$. The  
$n=4$ case needs a  more general map (\ref{hdflblc3}), since 
the weight four motivic complex in (\ref{hdflblc2}) 
is no longer the polylogarithmic complex, it rather, see \cite{G00}:
\be \la{100}
G_4(F) \lra {\cal B}_3(F) \otimes F^* \lra {\cal B}_2(F) \otimes \Lambda^2 F^* \lra \Lambda^4 F^*.
\ee
However, using the big period map on the  ${\cal H}_4$, one can  
extend  (\ref{hdflblc3}) to this case.

\subsection{Maps Bigrassmannian complex $~\to~$ motivic complexes} \label{sec3.2.4n}

\paragraph{1. Bigrassmannian complex $~\lra~$ Bloch complex.}
We construct a  map of complexes 
\begin{equation} \label{bctobloch2}
\begin{array}{ccccccc}
{BC}_5&
 \lra &{BC}_4&\lra & {BC}_3 & \lra & {BC}_2\\
\downarrow  &&\downarrow &&\downarrow &&\downarrow \\
0 & \lra &{\cal B}_2(F)& \stackrel{}{\lra} &\Lambda^2F^*&\lra &0
\end{array} 
\ee
It is defined at the Grassmannian bicomplex, raw by raw. 
The bottom raw goes to zero. 
The map on the second raw amounts to the following
 map of complexes, defined  
in (\ref{BCMAP}), Section 1:
\be \label{rs2}
\begin{array}{ccccc}
{C}_5(2) &\lra
 &{C}_4(2) &\lra &{C}_3(2) \\
\downarrow  &&\downarrow l_1 &&\downarrow l_2\\
0 &\lra&{\cal B}_2(F) & \stackrel{}{\lra} &\Lambda^2F^*
\end{array} 
\ee

Combining the homomorphism (\ref{dfltobigr})= (\ref{hdflblc1}) 
 with the homomorphism  (\ref{bctobloch2}), we arrive at 
  a homomorphism from the complex of decorated flags in $V_N$ to the Bloch complex: 

\be \label{dftobc1}
\begin{array}{ccccccc}
\ldots &\lra &C _5({\cal A}_N) &\lra
 &C _4({\cal A}_N) &\lra &C _3({\cal A}_N) \\
&&\downarrow 0 &&\downarrow  &&\downarrow \\
\ldots &\lra &0& \lra &{\cal B}_2(F) & \stackrel{}{\lra} &\Lambda^2F^*
\end{array} 
\ee
It is the main ingredient 
of the cocycle for the second motivic Chern class in \cite{G93}:
\be \label{mschc}
C_2^{\cal M} \in H^4(BGL_N, \Z_{\cal M}(2)).
\ee

\paragraph{2. Bigrassmannian complex $~\to~$ weight three motivic complex} \label{sec3.2.4n}

Let us construct a  map of complexes
\begin{equation} \label{bctobloch}
\begin{array}{ccccccccc}
{BC}_7 &\lra &{BC}_6 &\lra
 &{BC}_5 &\lra &{BC}_4&\lra &\ldots \\
\downarrow &&\downarrow  &&\downarrow  &&\downarrow &&\downarrow \\
0 &\lra &{\cal B}_3(F)& \lra &{\cal B}_2(F)\otimes F^*& \stackrel{}{\lra} &\Lambda^3F^*& \stackrel{}{\lra} &0
\end{array} 
\ee

We define it by looking at the Grassmannian bicomplex, and defining the map raw by raw. 

We send the bottom two raws to zero. 
The map on the third raw amounts to a construction of the following map of complexes:
\be \label{rs}
\begin{array}{ccccccc}
{C}_7(3) &\lra &{C}_6(3) &\lra
 &{C}_5(3) &\lra &{C}_4(3) \\
\downarrow &&\downarrow  &&\downarrow  &&\downarrow \\
0 &\lra &{\cal B}_3(F)& \lra &{\cal B}_2(F) \otimes F^*& \stackrel{}{\lra} &\Lambda^3F^*
\end{array} \ee
This has been done 
in Section 3.2 in \cite{G95a}), see some additions in Section 5 in \cite{G95b}.

Combining  homomorphism (\ref{dfltobigr}) 
 with the homomorphism  (\ref{bctobloch}) from the Bigrassmannian complex to the 
weight three motivic complex, we arrive at 
  a homomorphism of complexes

\be \label{dftobc1}
\begin{array}{ccccccc}
\ldots &\lra &C_6({\cal A}_N) &\lra
 &C_5({\cal A}_N) &\lra &C_4({\cal A}_N)  \\
&&\downarrow  &&\downarrow  &&\downarrow \\
\ldots &\lra &{\cal B}_3(F) & \lra &{\cal B}_2(F) \otimes F^*& \stackrel{}{\lra} &\Lambda^3F^*
\end{array} 
\ee
It is the main ingredient 
of the  cocycle for the third motivic Chern class in \cite{G93}: 
\be \label{mschc3}
C_3^{\cal M} \in H^6(BGL_N, \Z_{\cal M}(3)).
\ee
\paragraph{Bigrassmannian complex $~\lra~$ weight four motivic complex.} 
We will treat it in a different place, since it requires an elaborate exposition.

\paragraph{Remark.} Motivic Chern classes 
$C_n^{\cal M} \in H^{2n}(BGL_N, \Z_{\cal M}(n))$ are defined for $n \leq 4$ 
on Milnor's simplicial model of the classifying space $BGL_N$, and take values in 
the motivic complexes there. We construct  
cocycles representing these classes at the generic point of $BGL_N$.
It is a key property of the construction that these cocycles extend to  cocycles 
on the whole space $BGL_N$ with the values in the motivic complex defined 
using the Gersten resolution, see details in \cite{G93} for the weights 2 and 3, 
and even more details in Section 4 of \cite{G95a} for the weight 3.

Contrary to this,  our  construction of cocycles representing 
the  Deligne cohomology classes 
$$
C_n^{\cal D} \in H^{2n}(BGL^*_N, \Z_{\cal D}(n))
$$
works at the generic point only. This is sufficient 
for the goal, since $BGL^*_N$ is a model of the classifying space for the $GL_N$. 
 And this is sufficient to get explicit formulas for the Chern classes of vector bundles. 
Yet it is desired to extend the construction to $BGL_N$.

\input{realperiods.tex}

\end{document}

%% file: periodsc2.tex
\subsection{A local formula for the second Chern class of a two-dimensional vector bundle.} \la{sec1.2}
We consider complex vector bundles on real manifolds, and  produce a 
local formula for a Cech cocycle representing the topological second  Chern class, as well as 
the second  Chern class in the integral Deligne cohomology. 
All constructions can be applied to 
vector bundles over complex  manifolds. 
The algebraic part of the construction makes sense in Zariski topology. 

\vskip 3mm
Given a two-dimensional vector bundle $E$ on a manifold $X$, pick a cover $\{U_i\}$ of $X$ by 
open sets such that all intersections $U_{i_0 \ldots i_k}:= U_{i_0} \cap ... \cap U_{i_k}$ are empty or contractible. 
Choose a non-zero regular section $s_i$ on each open set $U_i$. 
Then, 

\begin{itemize}

\item For a three open sets $U_1, U_2, U_3$ 
there are three sections $s_1, s_2, s_3$ over  $U_{123}$. They provide  a section
$$
l_2(s_1, s_2, s_3) \in {\cal O}^*_{U_{123}} \otimes_\Z {\cal O}^*_{U_{123}}.
$$
Namely, pick a volume form $\omega \in {\rm det}(E^\vee_{U_{123}})$ 
on the restriction of $E$ to  $U_{123}$. Set
$$
\Delta(s_i, s_j):= \langle \omega, s_i\wedge s_j\rangle, 
$$
\be \la{item1}
l_2(s_1, s_2, s_3):= \Delta(s_1, s_2) \wedge  \Delta(s_2, s_3) + \Delta(s_2, s_3) \wedge \Delta(s_1, s_3) 
+ \Delta(s_1, s_3) \wedge \Delta(s_1, s_2).
\ee
This expression does not depend on the choice of the volume form $\omega$. 

\item For any four open sets $U_1, U_2, U_3, U_4$ take
 the cross-ratio of the restriction of the four sections to  $U_{1234}$:
\be \la{item2}
r(s_1, s_2, s_3, s_4) := \frac{\Delta(s_1, s_4) \Delta(s_2, s_3) }{\Delta(s_1, s_3) \Delta(s_2, s_4) }\in 
{\cal O}_{U_{1234}}^*.
\ee
The Pl\"ucker identity implies that it satisfies the crucial relation
\be \la{item3}
(1 - r(s_1, s_2, s_3, s_4)) \wedge r(s_1, s_2, s_3, s_4) = 
\ee
$$
l_2(s_2, s_3, s_4) - l_2(s_1, s_3, s_4) + l_2(s_1, s_2, s_4) - l_2(s_1, s_2, s_3). 
$$
\end{itemize}

Recall the map 
$\delta: \Z[F^*-\{1\}] \lra {\wedge}^2F^*$, given by $ \{x\} \lms (1-x) \wedge x.$ 

Recall  the 
subgroup $R_2(F) \subset \Z[F^*-\{1\}]$ generated by the ``five term relations'' 
(\ref{item4}). 

\bl
One has $\delta (R_2(F)) =0$.
\el  

\begin{proof} 
Denote by $C_n(k)$ the free abelian group generated by the configurations of 
$n$ vectors in generic position in a $k$-dimensional vector space over a field $F$. 
It follows from (\ref{item3}) that  there is a map of complexes 
\be \la{BCMAP}
\begin{array}{ccccc}
C_5(2) & \stackrel{d}{\lra} &C_4(2) &\stackrel{d}{\lra}&C_3(2)  \\
\downarrow l_0&&\downarrow l_1&&\downarrow l_2\\
R_2(F)&\hra&\Z[F^*-\{1\}]&\stackrel{\delta}{\lra}&\Lambda^2F^*
\end{array}
\ee
Here the map $l_2$ is given by (\ref{item1}), the map $l_1$ is given by 
$(s_1, ..., s_4) \lms \{r(s_1, ..., s_4)\}$, and the map 
$l_0$ assigns to a configuration of five generic vectors $(s_1, ..., s_5)$ 
the configuration of the corresponding five points on ${\rm P}^1$.  
\end{proof} 

Setting 
$B_2(F):= \Z[F^*-\{1\}]/R_2(F)$ we get the Bloch complex 
$\delta: B_2(F) \lra \Lambda^2 {F}^*$.

Our construction delivers a Cech cochain $C_\bullet$ 
for the covering $\{U_i\}$ of total degree four with values 
in the sheafified Bloch complex
\be \la{SBC}
  {\rm B}^\bullet(2)= ~~~  B_2({\cal O}) \lra \Lambda^2 {\cal O}^*.
\ee 
It 
 has two components given by (\ref{item1}) and (\ref{item2}): 
$$
C_3(U_i, U_j, U_k) \in \Lambda^2 {\cal O}_{U_{ijk}}^*~~~\mbox{and}~~~ 
C_4(U_i, U_j, U_k, U_l) \in B_2({\cal O}_{U_{ijkl}}).
$$
Condition (\ref{item3}) plus the five term relations (\ref{item4}) 
just mean that it is a cocycle. It represents the {\it second motivic Chern class} of the vector bundle $E$: 
\be \la{8.9.15.1}
c^M_2(E) \in H^4(X,   {\rm B}^\bullet(2)). 
\ee

\paragraph{Remark.} The   name  refers to a construction of the second 
universal motivic Chern class of Milnor's simplicial model $BGL_{2\bullet}$ of the classifying space $BGL_{2}$:
$$
c^{\cal M}_2\in H^4(BGL_{2\bullet},   \Z_{\cal M}(2)).
$$
Here $  \Z_{\cal M}(2)$ is the weight two motivic complex, which is a complex of sheaves 
in Zariski topology  on the simplicial scheme 
 $BGL_{2\bullet}$. It is defined by applying the Gersten resolution to the Bloch complex at 
the generic point. A complex two dimensional vector bundle on a manifold $X$ 
equipped with a Cech cover can be described as the pull back of the universal bundle over $BGL_{2\bullet}$. 
Then the class $c^{\cal M}_2$ pulls back to the class (\ref{8.9.15.1}).
\vskip 3mm

We use the classical topology,  aiming 
at a local formula for the topological Chern class
$$
c_2(E) \in H^4(X, \Z(2)).
$$
To get it from the motivic one (\ref{8.9.15.1}) is a non-trivial problem. 
Although it is asking for  the dilogarithm, 
we have  to deal with its complicated multivalued nature. 
We employ the weight two {\rm Lie}-exponential complex 
to handle this problem, and construct a cocycle 
representing the second Chern class  in the weight two {\rm Lie}-exponential Deligne complex
\be \la{CHECHCC}
c^{\cal D}_2(E) \in H^4(X,   \Gamma_{\cal D}(2)).
\ee
\vskip 3mm

Recall the  weight two {\rm Lie}-exponential complex
of sheaves:
\be \la{EXPCf}
  \Q_{\cal E}^\bullet(2):= ~~~~  {\cal O}(1)\lra  \Lambda^2 {\cal O} \stackrel{\wedge^2{\rm exp}}{\lra} \Lambda^2 {\cal O}^*.
\ee 
\be \la{DEXPCf}
2\pi i\otimes a \lms 2\pi i \wedge a, ~~~~
a \wedge b \lms {\rm exp}(a) \wedge {\rm exp}(b). 
\ee
 It is a resolution of $ \Q(2)$. 
Recall a map of complexes (\ref{4.8.15.1aaa}): 
\be \la{4.8.15.1aaa1}
\begin{array}{cccccccc}
  &&&{  B}_2({\cal O}) & \lra & \Lambda^2 {\cal O}^*\\
&&&\downarrow p_2&& \downarrow =\\
 &{\cal O}(1)&\lra  &\Lambda^2 {\cal O}& 
\stackrel{\wedge^2{\rm exp}}{\lra} &\Lambda^2 {\cal O}^*
\end{array}
\ee
 Here 
$$
{\rm L}_2(x) :=  {\rm Li}_2(x) + \frac{1}{2}\cdot  \log (1-x) \log x +\frac{(2\pi i)^2}{24},  
$$
$$
p_2(\{x\}_2):= \frac{1}{2}\cdot \log(1-x) \wedge \log x + 2\pi i \wedge 
\frac{1}{2\pi i}{\rm L}_2(x).
$$

Recall the weight two {{\rm Lie}-exponential Deligne complex} $  \Gamma_{\cal D}(2)$: 
$$
\begin{array}{ccccccccc}
&&&&{\cal O}(1)&\lra &{\cal O}\wedge {\cal O}&\lra 
&{\cal O}^*\wedge {\cal O}^*\\
 \Gamma_{\cal D}(2):=&&&&\downarrow =&&\downarrow \omega&&\downarrow  \\
&&&&{\cal O}& \stackrel{d}{\lra}&\Omega^1& \stackrel{}{\lra}&0
\end{array}
$$ 

By Proposition \ref{16.10.15.1}, 
there is a canonical morphism of 
complexes of sheaves 
\be \la{RDf}
r_{\cal D}:   {\rm B}^\bullet(2) \lra   \Gamma_{\cal D}(2). 
\ee

The Cech cocycle 
$(C_3, C_4)$ representing a class in $H^4(X,   {\rm B}^\bullet(2))$, combined with a morphism of complexes 
 (\ref{RDf}),  delivers 
a  Cech cocycle representing the second Chern class $c_2^{\cal D}(E)$ in 
(\ref{CHECHCC}). 

Namely, we start with the Cech cocycle 
$(C_3, C_4)$ with values in the Bloch complex: 
$$
C_4 \in {  B}_2({\cal O})  \stackrel{\delta}{\lra} C_3 \in \Lambda^2 {\cal O}^*.
$$
Let us define a Cech cochain $(C_3, \widetilde C_3, \widetilde C_4, \widetilde C_5)$ 
with values in the weight two {\rm Lie}-exponential complex, 
organised as follows:
\be \la{4.8.15.1}
\begin{array}{cccccccc}
& \widetilde C_5 \in \Z(2)&\stackrel{d}{\lra} &  \widetilde C_4 \in {\cal O}(1)&\stackrel{d}{\lra} & \widetilde C_3 \in \Lambda^2 {\cal O}& 
\stackrel{\wedge^2{\rm exp}}{\lra} & C_3\in \Lambda^2 {\cal O}^*
\end{array}
\ee
\begin{itemize}

\item For any three open sets $U_1, U_2, U_3$, let us define $$
\widetilde C_3(U_1, U_2, U_3)\in \Lambda^2{\cal O}_{U_{123}}. 
$$
Namely, we choose a branch of each $\log \Delta(s_i, s_j)$  
on $U_{123}$, and  set 
\be \la{item1a}
\widetilde C_3(U_1, U_2, U_3):= 
\ee
\be \la{item1ab}
\log \Delta(s_1, s_2) \wedge \log \Delta(s_2, s_3) + \log \Delta(s_2, s_3) \wedge \log \Delta(s_1, s_3) 
+ \log \Delta(s_1, s_3) \wedge \log \Delta(s_1, s_2).
\ee

\item We assign to any four open sets $U_1, U_2, U_3, U_4$ an element
$$
\widetilde C_4(U_1, U_2, U_3, U_4) \in {\cal O}_{U_{1234}}(1).
$$
To define it, we use an isomorphism, see (\ref{EXPC}) - (\ref{DEXPC}):
$$
{\cal O}_U(1)/\Z(2) ~\stackrel{\sim}{=}~ \Z(1) \wedge  {\cal O}_U(1) ~=~ 
 {\rm Ker}\Bigl(\Lambda^2{\cal O}_U \stackrel{\wedge^2 {\rm exp}}{\lra} \Lambda^2{\cal O}^*_U \Bigr).
$$
So we exhibit an element in $\Lambda^2{\cal O}_{U_{1234}}$ which is in 
the kernel of the $\wedge^2 {\rm exp}$ map:
$$
\widetilde C_4(U_1, U_2, U_3, U_4) := 
(\delta_{\rm Cech}\circ \widetilde C_3)(U_1, U_2, U_3, U_4) - 
$$
\be \la{item2a}
2 \pi i \wedge \frac{1}{2\pi i}{\rm L}_2(r(s_1, s_2, s_3, s_4)) + \log (1 - r(s_1, s_2, s_3, s_4)) \wedge 
\log r(s_1, s_2, s_3, s_4).
\ee

To find 
$\widetilde C_4(U_1, U_2, U_3, U_4)$ explicitly 
we start with an equality in $\wedge^2{\cal O}_{U_{1234}}^*$:
\be \la{EQTME}
(\delta_{\rm Cech}\circ C_3)(U_1, U_2, U_3, U_4) + (\delta_{\rm Bloch}\circ C_4)(U_1, U_2, U_3, U_4) = 0, 
\ee
 which
 is just equivalent to (\ref{item3}). It follows that 
$$
(\delta_{\rm Cech}\circ \widetilde C_3)(U_1, U_2, U_3, U_4) + 
(p_2\circ C_4)(U_1, U_2, U_3, U_4)  
$$
$$
= 2 \pi i \wedge \frac{1}{2\pi i}{\rm L}_2(r(s_1, s_2, s_3, s_4)) + 
2\pi i \wedge \log~F.
$$ 
So we set, ``dropping'' $2\pi i \wedge $ in the last formula: 
$$
\widetilde C_4(U_1, U_2, U_3, U_4):= \frac{1}{2\pi i}{\rm L}_2(r(s_1, s_2, s_3, s_4)) + \log~F. 
$$

The ``correction term'' $2\pi i \wedge \log~F$  shows up as follows. 
Since 
$\log (fg) - \log (f) - \log (g) $ is a locally constant function with 
values in $2\pi i \Z$, an equality 
$\sum_i f_i \wedge g_i =0$, which in our case is just the equality 
(\ref{EQTME}),  implies only that, after we choose  branches of 
$\log (f_i)$ and $\log (g_i)$ on a contractible set, $\sum_i \log (f_i) \wedge \log (g_i) = 2\pi i \wedge \log ~F$. Notice that in our case the choices of the branches of $\log$ consist of  the choices 
made in (\ref{item1ab}) and (\ref{item2a}) 

\item Finally, to any five open sets $U_1, U_2, U_3, U_4, U_5$ we assign an element
$$
\widetilde C_5(U_1, U_2, U_3, U_4, U_5) := \sum_{i=1}^5 (-1)^i \widetilde C_4(U_1, \ldots , 
\widehat U_i, \ldots , U_5) \in (2\pi i)^2\Q.
$$
A priory this sum lives in ${\cal O}_{U_{12345}}$. We claim that it is annihilated by the 
differential in the exponential complex. Indeed, the Cech coboundary of the first line (\ref{item2a}) 
is zero due to  the five term relation for the $\Lambda^2\C$-valued dilogarithm, 
that is since the map $p_2$ sends the five term relation to zero. For the second line this is just 
$\delta_{\rm Cech}^2=0$. Therefore 
$C_5^{\cal D} \in (2\pi i)^2\Q$.

\end{itemize}

We get a cocycle in the  
Cech complex with coefficients in the {\rm Lie}-exponential Deligne complex. It   
 represents the second Chern class $c_2^{\cal D}(E)$, 
and hence the usual Chern class.

%% file: realperiods.tex
\section{Appendix: a map to the real Deligne complex}  

\paragraph{An outline.} 
Let $(S^{\bullet}, d)$ be the de Rham complex of {\it smooth} real valued forms on a manifold $X$. 
Recall that we constructed a map of complexes
$$
\omega_n^{(\bullet)}: \Q_{\cal E}^\bullet(n) ~~\lra~~ \Omega^\bullet.  
$$
Consider the canonical projection:
$$
\pi_n: \C \lra \C/\R(n) = \R(n-1); ~~~~\pi_n(a+ib):= \begin{cases} 
a & n~ \mbox{odd},\\
 ib & n ~\mbox{even}.
\end{cases}
$$
The map $\pi_n$ induces a projection 
of the de Rham complex of complex valued smooth forms to the de Rham complex of 
$\R(n-1)$-valued forms: 
$$
\pi_n: \Omega^\bullet ~~\lra ~~ S^\bullet(n-1). 
$$
So we get a canonical map from the exponential complex:
\be \la{MAPPHI}
\varphi_n^{(\bullet)}:= \pi_n\circ \omega_n^{(\bullet)}:   \Q_{\cal E}^\bullet(n) \lra S^\bullet(n-1).  
\ee
We will show that the map $\varphi_n^{(\bullet)}$ is canonically homotopic to zero 
by  constructing 
a homotopy 
\be \la{MAPS}
s_n^{(\bullet)}:   \Q_{\cal E}^\bullet(n) \lra S^\bullet(n-1)[-1], ~~~~d\circ s_n^{(\bullet)} +s_n^{(\bullet)} \circ d =  
\varphi_n^{(\bullet)}.
\ee

Let us assume that we have a map, conjectured in 
Conjecture \ref{periods1small}, from the weight $n$ part ${\cal L}^\bullet(n)$ of 
cochain complex of the 
$\Q$-Hodge-Tate Lie coalgebra ${\cal L}$ to the {\rm Lie}-exponential complex: 
$$
p^{(\bullet)}_n: {\cal L}^\bullet(n) \stackrel{}{\lra}   \Q_{\cal E}^\bullet(n). 
$$
Recall an important feature of the map (\ref{MAPPHI}):
$$
\mbox{the composition}~~ \varphi_n^{(\bullet)}\circ p^{(\bullet)}_n: {\cal L}^\bullet(n) \stackrel{}{\lra}   \Q_{\cal E}^\bullet(n) 
\stackrel{}{\lra} S^\bullet(n-1)~~\mbox{is zero}.
$$
Therefore the composition $s_n^{(\bullet)}\circ p^{(\bullet)}_n$ is  a map of complexes:
\be \la{10.17.2015.1}
s_n^{(\bullet)}\circ p^{(\bullet)}_n: {\cal L}^\bullet(n) \stackrel{}{\lra}   \Q_{\cal E}^\bullet(n)
\stackrel{}{\lra} S^\bullet(n-1)[-1].
\ee

Recall that the weight $n$ real Deligne complex is given by the cone
$$
  \R_{\cal D}(n) = ~~{\rm Cone}\Bigl(\pi_n: F^n\Omega^\bullet \lra S^\bullet(n-1)[-1]\Bigr).
$$
\bl
The map (\ref{10.17.2015.1}) gives rise to a morphism to the weight $n$ real Deligne complex:
\be \la{10.17.2015.2}
(s_n^{(\bullet)}\circ p^{(\bullet)}_n, \omega_n^{(n)}\circ p_n^{(n)}): {\cal L}^\bullet(n) 
\stackrel{}{\lra}   \R_{\cal D}(n).
\ee
\el
\begin{proof}
The map (\ref{10.17.2015.1}) gives the component of the map (\ref{10.17.2015.2}) 
 in $(S^0 \to ... \to S^{n-1})(n-1)[-1]$. The only other non-trivial component 
is the standard map $\omega_n^{(n)}\circ p_n^{(n)}: \Lambda^n{\cal L}_1 \lra \Omega^n$. 
\end{proof}

In particular, combining this with a regulator map from the motivic complex to 
${\cal L}^\bullet(n)$ we would get a homomorphism from the motivic complex to the weight 
$n$ real Deligne complex.

\paragraph{The morphism $\varphi_n^{(\bullet)}$.} 
It is  a morphism of complexes which looks as follows:
\be \la{HOMD}
\begin{array}{ccccccccccc}
{\cal O}(n-1) &\to&\Lambda^2 {\cal O}(n-2) & \stackrel{\delta}{\to}&...
&\stackrel{\delta}{\to}&\Lambda^n  {\cal O}&\stackrel{\rm exp}{\to}&\Lambda^n  {\cal O}^*\\
&&&&&&&& \\
\downarrow \varphi_n^{(0)}&&\downarrow \varphi_n^{(1)}& ...&...&...
&\downarrow \varphi_n^{(n-1)}&&\downarrow \varphi_n^{(n)}\\
&&&&&& &\\
S^0(n-1)&\to &S^1(n-1)&   \stackrel{d}{\to}& ... &
\stackrel{d}{\lra}&S^{n-1}(n-1)&\stackrel{d}{\to}&S^{n}(n-1) 
\end{array}
\ee

Namely, 
$$
\varphi_n^{(n)}: \Lambda^{n} {\cal O}^*\lms S^{n}(n-1),~~~~F_1\wedge ... \wedge F_n \lms \pi_n(d\log F_1 
\wedge ... \wedge d\log F_n),
$$
and for $k=1, ..., n$, 
$$
\varphi_n^{(k-1)}: \Lambda^{k } {\cal O}(n-k )\lms S^{k-1}(n-1), 
$$
\be \la{FWKMOF}
\begin{split}
&(2 \pi i)^{n-k } \cdot f_1 \wedge ... \wedge  f_{k } \lms 
  \pi_{n} \circ 
d^{-1} \Bigl((2 \pi i)^{n-k } \cdot df_1 \wedge ... \wedge df_{k }\Bigr) := \\
&(k-1)!~\pi_{n}\Bigl( (2 \pi i)^{n-k }\cdot \sum_{i=1}^{k } (-1)^i f_i ~
df_1 \wedge ...  \wedge {\widehat {df_i}} \wedge ... \wedge df_{k }\Bigr).
\end{split}
\ee

\paragraph{A homotopy $s_n^{(\bullet)}$.} 
For example for $n=2$ we are going to get a diagram of maps
 \be
\begin{array}{ccccccc}
&&{\cal O}(1)&\stackrel{\delta}{\lra}&\Lambda^2 {\cal O}& 
\stackrel{\rm exp}{\lra}&\Lambda^2 {\cal O}^*\\
&&&&&&\\
&\swarrow s_2^{(0)}&\downarrow \varphi_2^{(0)}&\swarrow s_2^{(1)}&\downarrow \varphi_2^{(1)}&\swarrow s_2^{(2)} &\downarrow \varphi_2^{(2)}\\
&&&&&&\\
0& \lra&S^0(1) & \stackrel{d}{\lra} &S^1(1) & \stackrel{d}{\lra} &S^2(1)
\end{array}
\ee

Let   ${\rm
Alt}_n F(x_1,...,x_n):= \sum_{\sigma \in S_n}(-1)^{|\sigma|}F(x_{\sigma
(1)},...,x_{\sigma (n)})$, and ${\rm Im}(x+iy) := iy$. Let us set
$$
\widetilde r_{n-1}: \Lambda^n  {\cal O}  \lra S^{n-1}(n-1), \quad f_1 \wedge ... \wedge f_n \lms  
d^{-1} \circ \pi_n(d  f_1  \wedge ... \wedge  d f_n):=
$$
$$
 {\rm Alt}_n \sum_{j\geq 0} c_{j,n}{\rm Re}f_1 ~d{\rm Re}f_2 
\wedge ... \wedge d{\rm Re}f_{2j+1} \wedge d{\rm Im} f_{2j+2}\wedge ... \wedge
d{\rm Im} f_{n}.
$$
Here 
$c_{j,n}:= \frac{1}{(2j+1)!(n-2j-1)!}$. 
For example, 
$$
\widetilde r_{1}: \Lambda^2  {\cal O}  \lra S^{1}(1), \quad f_1 \wedge f_2 \lms  
d^{-1} \circ \pi_2(d  f_1  \wedge  d f_2):= {\rm Re}f_1 ~d{\rm Im} f_{2} -{\rm Re}f_2 ~d{\rm Im} f_{1}.  
$$
A primitive $d^{-1} \circ \pi_n(d  f_1  \wedge ... \wedge  d f_n)$ is not uniquely defined. 
Our choise has a  property that 
\be \la{CHOISE}
r_{k-1}( 2 \pi i  \wedge  f_2 \wedge ... \wedge f_k) = 0. 
\ee

Set $s_n^{(n)}:= \frac{1}{n!} r_{n-1}$.  Let us define   maps 
$$
s_n^{(k )}: \Lambda^{k+1 } {\cal O}(n-k -1) \lra S^{k-1}(n-1)\qquad  1 \leq k \leq n-1
$$ 
by setting   
$$
 s_n^{(k )}:  (2 \pi i)^{n-k-1 } \otimes f_0 \wedge ... \wedge f_k ~\lms~
\frac{ (2 \pi i)^{n-k-1}}{n!} {\rm Alt}_{k+1}\Bigl({\rm Im}f_0 \cdot \widetilde r_{k-1}(f_1\wedge ... \wedge f_k)\Bigr).
$$

\begin{theorem} The map $s_n^{(\bullet)}$ is a homotopy 
between the map $\varphi_n^{(\bullet)}$ and zero:
$$
s_n^{(k+1)}\circ \delta  + d \circ s_n^{(k)} = \varphi_n^{(k)} \quad  
\mbox{for} \quad  1 \leq k \leq n-1
$$  
\end{theorem}

\begin{proof}
Let us prove the statement for the   diagram 
$$
\begin{array}{ccccc}
&&\Lambda^k {\cal O}(n-k)& \stackrel{\delta}{\lra}&\Lambda^{k+1} {\cal O}(n-k-1)\\
&&&&\\
&\swarrow s_n^{(k-1)}&\downarrow \varphi_n^{(k-1)}&\swarrow s_n^{(k)}&\\
&&&&\\
S^{ k-2 }(n-1)& \stackrel{d}{\lra}&S^{ k-1}(n-1)&&
\end{array}
$$

Thank to (\ref{CHOISE}),  for $k \leq n-1$  
one has 
\be\la{WHOKNW1}
 \begin{split}
&s_n^{(k )}\circ  \delta \Bigl( (2 \pi i)^{n-k} \otimes f_1 \wedge ... \wedge f_k\Bigr)= 
\\&s_n^{(k)} \Bigl( (2 \pi i)^{n-k-1} \otimes 2 \pi i \wedge f_1 \wedge ... \wedge f_k\Bigr)= 
\\&k!~\frac{(2 \pi i)^{n-k}}{n!}\widetilde r_{k-1}(f_1 \wedge ... \wedge f_k).
\end{split}
\ee
It is easy to see that the same result is valid also for $k=n$. 
On the other hand 
\be \la{WHOKNW}
 \begin{split}
 &d\circ s_n^{(k )}   \Bigl( (2 \pi i)^{n-k} \otimes f_1 \wedge ... \wedge f_k\Bigr)=\\
&d\frac{  (2 \pi i)^{n-k} }{n! }  {\rm Alt}_{k }\Bigl({\rm Im}f_1 \cdot \widetilde r_{k-2}(f_2 \wedge ... \wedge f_k)\Bigr) = \\
 &(k-1)!~\frac{  (2 \pi i)^{n-k} }{n!}  \cdot d\Bigl( \sum_{i=1}^k(-1)^{i-1}{\rm Im}f_i \cdot \widetilde r_{k-2}(f_1\wedge ... \wedge \widehat f_i \wedge ... \wedge f_k)\Bigr).
\end{split}
\ee

Putting together (\ref{FWKMOF}), (\ref{WHOKNW1}) and (\ref{WHOKNW}),  
and dividing by $(k-1)!$, the statement reduces to the following  
basic identity:
\be
 \begin{split}
&d\Bigl( \sum_{i=1}^k(-1)^{i-1}{\rm Im}f_i \cdot \widetilde r_{k-2}(f_1\wedge ... \wedge \widehat f_i \wedge ... \wedge f_k)\Bigr) +
k \cdot \widetilde r_{k-1}(f_1\wedge ... \wedge f_k) =\\
&\pi_k\Bigl( \sum_{i=1}^k(-1)^{i-1} f_i \cdot df_1\wedge ... 
\wedge\widehat df_i \wedge ... \wedge df_k\Bigr).
\end{split}
\ee
We can rewrite it in its natural form:
\be
\begin{split}
&k \cdot \Bigl(\pi_k \circ d^{-1} - d^{-1} \circ \pi_k\Bigr) (f_1\wedge ... \wedge f_k) =\\ 
&d \Bigl( \sum_{i=1}^k(-1)^{i-1}{\rm Im}f_i \cdot \widetilde r_{k-2}(f_1\wedge ... \wedge \widehat f_i \wedge ... \wedge f_k \Bigr).
\end{split}
\ee

{\it Proof of the basic identity}. We  need the following simple observation:
$$
\widetilde  r_{k-1 }(f_1\wedge ...   \wedge f_k) =    
 \widetilde  r_{k-2 }(f_1\wedge   ... \wedge  f_{k-1}) \wedge d{\rm Im}f_k ~~+ ~~ \mbox{terms without $d {\rm Im} f_k$}.
$$
 
We prove the basic identity by induction. 
Let $k=2$. Then it boils down to  
$$
d\Bigl( {\rm Im} f_1 {\rm Re} f_2 -  {\rm Im} f_2 {\rm Re} f_1 \Bigr) + 
2 \Bigl({\rm Re} f_1 d {\rm Im} f_2 -  {\rm Re} f_2 d {\rm Im} f_1\Bigr) 
$$
$$
{\rm Re} f_1 d {\rm Im} f_2 -  {\rm Re} f_2 d {\rm Im} f_1 + {\rm Im} f_1 d {\rm Re} f_2 -  {\rm Im} f_2 d {\rm Re} f_1,
$$
which is easy to check. 

Let us assume that the identity was already proved for $k-1$. 
We compute  first the parts of each 
of 
the  sides containing the term $d {\rm Im} f_k$.  The contribution of the right hand side is
$$
\pi_{k-1}\Bigl(  \sum_{i=1}^{k-1}(-1)^{i-1} f_i \cdot df_1\wedge ... \wedge \widehat df_i \wedge ... \wedge df_k \Bigr) \wedge d {\rm Im} f_k.
$$
By the induction assumption this is equal to
$$
\Bigl(d  \sum_{i=1}^{k-1}(-1)^{i-1}  {\rm Im}f_i \cdot \widetilde r_{k-3}(f_1\wedge ... \wedge \widehat f_i \wedge ... \wedge  f_{k-1})  + (k-1)  \widetilde r_{k-2}(f_1\wedge   ... \wedge  f_{k-1})\Bigr)  \wedge d {\rm Im} f_k.
$$
We have to show that this expression is equal to the $d {\rm Im} f_k
$-content of the left hand side of the basic equality, i.e. to 
$$
 - \widetilde r_{k-2}(f_1\wedge   ... \wedge  f_{k-1})  \wedge d {\rm Im} f_k
+ k   \widetilde r_{k-2}(f_1\wedge   ... \wedge  f_{k-1})  \wedge d {\rm Im} f_k+
$$
$$
\Bigl( \sum_{i=1}^{k-1}(-1)^{i-1} d\widetilde r_{k-2}(f_1\wedge ... \wedge \widehat f_i \wedge ... \wedge  f_{k -1})  {\rm Im} f_i\Bigr) \wedge d {\rm Im} f_k.
$$
This is obvious. 
It remains to check that the $d {\rm Im} f_k$-free parts of the basic equality also coincide. The right hand side gives us
\begin{equation} \label{dwa}
 \sum_{i=1}^k(-1)^{i-1} \pi_k(f_i) (d{\rm Re}f_1\wedge ... \wedge \widehat d{\rm Re}f_i \wedge ... \wedge  d{\rm Re}f_{k }).
 \end{equation}
Let us assume first that $k$ is odd. Then the  left hand side is 
$$
\sum_{i=1}^k(-1)^{i-1} {\rm Re} f_i  (d{\rm Re}f_1\wedge ... \wedge \widehat d{\rm Re}f_i \wedge ... \wedge  d{\rm Re}f_{k }), 
$$
which coincides with (\ref{dwa}) since $ \pi_k(f_i) = {\rm Re} f_i$ if $k$ is odd. If $k$ is even the first term contributes
$$
\sum_{i=1}^k(-1)^{i-1} {\rm Im} f_i \pi_{k-1}(f_1\wedge ... \wedge \widehat f_i \wedge ... \wedge  f_{k }),
$$
which coincides with (\ref{dwa}) since ${\rm Im} f_i = \pi_k(f_i)$ in this case. \end{proof}